\documentclass{article}

\usepackage{amsmath}
\usepackage{amsthm}
\usepackage{amsfonts}
\usepackage{amssymb}
\usepackage{epsfig, graphicx}
\usepackage{graphicx}%
\usepackage[colorlinks=true, linkcolor=blue, citecolor=green, urlcolor=magenta]{hyperref}
\usepackage[nameinlink, capitalise]{cleveref}
\usepackage{mathrsfs}
\usepackage{color}
\usepackage{tikz-cd}
\usepackage{enumerate} 
\usepackage{enumitem}
\usepackage{aliascnt}
\usepackage{comment}
\usepackage[letterpaper, left=2.4cm, right=2.4cm, top=3.5cm, bottom=3.5cm]{geometry}

\providecommand{\U}[1]{\protect\rule{.1in}{.1in}}
\newtheorem{theorem}{Theorem}[section]

\newtheorem{corollary}[theorem]{Corollary}

\newtheorem{definition}[theorem]{Definition}

\newtheorem{lemma}[theorem]{Lemma}
\newtheorem{notation}[theorem]{Notation}
\newtheorem{question}[theorem]{Question}
\newtheorem{proposition}[theorem]{Proposition}
\newtheorem{remark}[theorem]{Remark}

\def\N{{\mathbb N}}
\def\S{{\mathcal S}}
\def\Z{{\mathbb Z}}
\def\D{{\mathcal D}}
\def\U{{\mathcal U}}

\def \act{{\Gamma \curvearrowright X}}

\newcommand{\ZZ}{\mathbb{Z}}			
\newcommand{\NN}{\mathbb{N}}			
\newcommand{\R}{\mathbb{R}}			

\newcommand{\Map}{{\rm Map}}
\newcommand{\IE}{\mathrm{IE}}
\newcommand{\IN}{\mathrm{IN}}
\newcommand{\IT}{\mathrm{IT}}
\newcommand{\ev}{\mathrm{ev}}
\newcommand{\SE}{\mathrm{SE}}
\newcommand{\RP}{\mathrm{RP}}
\newcommand{\Q}{\mathrm{Q}}
\newcommand{\A}{\mathrm{A}}

\newcommand{\Sym}{{\rm Sym}}

\newcommand{\supp}{
	\operatorname{\mathrm{supp}}%
}

\newcommand{\htop}{h_{\rm{top}}}
\newcommand{\hnaive}{h^{\rm{nv}}_{\rm{top}}}
\newcommand{\hsof}{h^{\Sigma}_{\rm{top}}}

\newcommand{\mdim}{{\rm mdim}}
\newcommand{\ord}{{\rm ord}}

\newcommand{\cB}{{\mathcal B}}
\newcommand{\cH}{{\mathcal H}}
\newcommand{\cK}{{\mathcal K}}
\newcommand{\cM}{{\mathcal M}}

\newcommand{\cU}{{\mathcal U}}

\newcommand{\sB}{{\mathcal B}}

\newcommand{\rStab}{{\rm Stab}}
\newcommand{\mult}{{\rm mult}}

\newcommand{\Cb}{{\mathbb C}}
\newcommand{\Zb}{{\mathbb Z}}

\newcommand{\Rb}{{\mathbb R}}
\newcommand{\Nb}{{\mathbb N}}

\newcommand{\ox}{{\mathbf{x}}}

\newcommand{\comp}{{\rm c}}
\newcommand{\eq}{{\rm eq}}

\newcommand{\oA}{{\boldsymbol{A}}}

\newcommand{\upind}{{\overline{\rm I}}}

\title{Local entropy theory and applications}
\author{Felipe García-Ramos and Hanfeng Li}
\date{}

\newcommand{\Addresses}{{
  \bigskip

\hskip-\parindent F.~Garc\'ia-Ramos, \textsc{Unidad Cuernavaca-Instituto de Matemáticas, Universidad Nacional Autónoma de México, M\'exico.\\
Faculty of Mathematics and Computer Science, Uniwersytet Jagielloński, Poland}\par\nopagebreak
  \textit{E-mail address}: \texttt{felipe.garciaramos@im.unam.mx}

 \medskip

\hskip-\parindent  H.~Li, \textsc{College of Mathematics and Statistics, Center of Mathematics, Chongqing University, Chongqing 401331, P.R. China.}\par\nopagebreak
  \textit{E-mail address}: \texttt{hfli@cqu.edu.cn}

}}

\begin{document}

\maketitle
\abstract{
This paper contains a survey about recent developments in the local entropy theory for topological dynamical systems and continuous group actions, with particular emphasis on the connections with other areas of dynamical systems and mathematics.  Besides the survey, we will present new results regarding regionally proximal tuples and IT-tuples.} 
\tableofcontents
\section{Introduction}

A measure-preserving transformation has completely positive entropy if every non-trivial factor has positive entropy. This is equivalent to the statement that every non-trivial finite partition of the phase space has positive entropy. Bernoulli transformations are the prototypes for completely positive entropy, and it was an open question for some time whether these were the only examples. Nowadays there exist plenty of examples with completely positive entropy that are not Bernoulli.

Blanchard introduced two notions for topological dynamical systems that emulate the behaviour of completely positive entropy \cite{Blanchard1992} in a topological context. A topological dynamical system has \textit{completely positive entropy} if every non-trivial factor has positive topological entropy, and has \textit{uniform positive entropy} if the topological entropy of every non-dense finite open cover is positive. In contrast with measurable systems, these properties turned out to be different.
In order to comprehend them better, Blanchard introduced \textit{entropy pairs} in \cite{Blanchard1993}. Entropy pairs allow us to understand ``where'' the entropy lies in the system. The answer to the question ``where'' is not in the phase space, but in the product of the phase space. Informally, a pair of points is an entropy pair if every finite open cover that separates them has positive topological entropy.

It turns out that this approach of studying topological dynamics, nowadays known as local entropy theory, is not only useful for understanding positive entropy and completely positive entropy. With time it was discovered that this perspective is related to other dynamical objects like Li-Yorke chaos, Ellis semigroup, maximal equicontinuous factor, sequence entropy, homoclinic points, mean dimension, spectrum, nilsystems, and cardinality of ergodic measures. Furthermore, it has paved the way for novel interactions between dynamical systems and other areas of mathematics like operator algebras, combinatorics, group theory, descriptive set theory, and continuum theory.

This paper is not the first survey on local entropy theory (see \cite{GlasnerYe2009}); nonetheless, we consider that many important contributions have appeared since the publication of that survey, in particular many applications and connections with other fields. 

In this paper we also include several new results, including a proof of \cite[Conjecture 1]{huang2021minimal} and an affirmative answer of \cite[Question 1.8]{liu2025sequenceIT}. The main new contributions are \cref{RP-pair to tuple}, \cref{chumin1} and \cref{T-length}. 

\textbf{Acknowledgements}: 

The first author was partially supported by the Simons Foundation Award 663281 granted to the Institute of Mathematics of the Polish Academy of Sciences, and the Grant:U1U/W16/NO/01.03 of the Strategic Excellent Initiative program at the Jagiellonian University. The second author was partially supported by the National Science Foundation of China (NSFC) grant 12671218.

We thank Chunlin Liu for valuable comments on an earlier version of this manuscript, and are grateful to David Burguet and Ruxi Shi for allowing us to include their result \cref{P-BS}.

The first author would like to thank Yonatan Gutman, Dominik Kwietniak and Feliks Przytycki for the invitation to give a mini-course on this topic at IMPAN.

\section{Basic Definitions}
Throughout this paper $\NN$ denotes the set of positive integers and $\Gamma$ denotes a countably infinite group with identity $e_{\Gamma}$.

A (left) \textbf{action} of the group $\Gamma$ on $X$ is represented by $\act$. From Section 2 to Section 6 we will assume that given an action $\act$, $X$ is a compact metrizable space and that $\Gamma$ acts by homeomorphisms. On Section 7 we will not assume any action is continuous. 

We usually denote compatible metrics by $\rho$.

Let $\Gamma \curvearrowright X$ and $\Gamma\curvearrowright Y$ be two actions. A function $\pi\colon X \rightarrow Y$ is \textbf{$\Gamma$-equivariant} if $g\pi(x)=\pi(gx)$ for every $g \in \Gamma$ and $x \in X$, in this situation, we write $\pi\colon \act \rightarrow \Gamma\curvearrowright Y$. If a $\Gamma$-equivariant function, $\pi\colon X\rightarrow Y$, is continuous and surjective, we say $\Gamma\curvearrowright Y$ is a \textbf{factor of} $\act$ and $\pi$ is a \textbf{factor map}. Furthermore, if $\pi$ is bijective, we say $\act$ and $\Gamma\curvearrowright Y$ are \textbf{conjugate}.

The following proposition is standard (e.g. see \cite[Appendix E.11-3.]{de2013elements}).
\begin{proposition}\label{prop:factor-eq_rel}
    If $\pi\colon \act \rightarrow \Gamma\curvearrowright Y$ is a factor map, then $$R_\pi=\{(x,x'):\, \pi(x)=\pi(x')\}\subset X\times X$$ is a closed $\Gamma$-invariant equivalence relation. Conversely, if $R\subset X\times X$ is a closed $\Gamma$-invariant equivalence relation, then $\pi\colon \act \rightarrow \Gamma\curvearrowright X/R$ is a factor map where $\Gamma\curvearrowright X/R$ is the induced action $gy=\pi(g\pi^{-1}(y))$. 
\end{proposition}

 We denote by $\mathcal{M}(X)$ the set of Borel probability measures on $X$. 
 Let $\act$ be an action. We say $\mu\in \mathcal{M}(X)$ is \textbf{invariant or $\Gamma$-invariant} if the measure is invariant under the action of each element of $\Gamma$.  We denote the set of $\Gamma$-invariant measures by $\mathcal{M}(\act)$.
 We say  $\mu\in \mathcal{M}(X)$ is \textbf{ergodic} if $\mu\in \mathcal{M}(\act)$, and every $\Gamma$-invariant Borel set has full or null measure.  We denote the set of ergodic measures by $\mathcal{M}_e(\act)$.
 
 We say $\act$ is \textbf{uniquely ergodic} if there is only one $\Gamma$-invariant Borel probability measure.

 A $\Gamma$\textbf{-subshift} over a finite set $A$ is a closed subset of $A^G$  that is invariant under the shift action of $\Gamma$. Thus, for every $\Gamma$-subshift $X$ one can naturally assign an action $\act$. When we describe properties of a subshift, we usually mean properties of its associated shift action.
 
\section{IE-tuples for group actions}  

In this section we provide an introduction to the basic results of local entropy theory and combinatorial independence. We will provide the definition of entropy in different contexts, but we will not go into details about the basic results of entropy theory. For a more complete exposition of this topic, we refer the reader to \cite{Downarowicz2011, pollicott1998dynamical, Walters1982} for entropy theory of $\ZZ$-actions, and to \cite{KerrLiBook2016, Ollagnier1985book} for entropy theory of  actions of amenable and sofic groups.
\subsection{Amenable groups} 
If $K$ is a non-empty finite subset of $\Gamma$, we write $K \Subset \Gamma$.
	
	Let $\delta >0$ and $K \Subset \Gamma$. A set $F \Subset \Gamma$ is said to be $(K,\delta)$-invariant if $|KF \Delta F| \leq \delta |F|$. A sequence of non-empty  finite subsets, $\{F_{n}\}_{n\in\mathbb{N}}$, of $\Gamma$ is said to be \textbf{F\o lner} if it is eventually $(K,\delta)$-invariant for every $K \Subset \Gamma$ and $\delta >0$. A countable group $\Gamma$ is \textbf{amenable} if it admits a F\o lner sequence. Abelian groups and finitely generated groups of sub-exponential growth are always amenable.

 	Given two finite open covers $\mathcal{U},\mathcal{V}$ of $X$, we define their \textbf{join} by $$\mathcal{U} \vee \mathcal{V} = \{U \cap V : U \in \mathcal{U}, V\in \mathcal{V}\}.$$ Let $g \in \Gamma$. We define $g\mathcal{U} = \{gU : U \in \mathcal{U}\}$ and $N(\mathcal{U})$ as the smallest cardinality of a subcover of $\mathcal{U}$. If $F\Subset \Gamma$, denote by $\mathcal{U}^F$ the join
	
	$$\mathcal{U}^F = \bigvee_{g \in F}g^{-1}\mathcal{U}.$$

		Let $\Gamma$ be an amenable group, $\act$ an action, $\mathcal{U}$ a finite open cover of $X$ and $\{F_n\}_{n \in \NN}$ a F\o lner sequence for $\Gamma$. We define the \textbf{topological entropy of
			$\act$ with respect to $\mathcal{U}$} as%
		\[
		\htop(\act,\mathcal{U})=\lim_{n\to\infty}\frac{1}{\left\vert
			F_{n}\right\vert }\log N(\mathcal{U}^{F_{n}})\in [0,\infty).
		\]
	The function $F \mapsto \log N(\mathcal{U}^{F})$ is subadditive and thus the limit exists and does not depend on the choice of the F\o lner sequence (see, for instance, ~
 \cite[page 220]{KerrLiBook2016}). The \textbf{topological entropy} of $\act$ is defined as
 \[
	\htop(\act)=\sup_{\mathcal{U}}\htop(\act,\mathcal{U})\in [0,\infty],
	\]
where $\U$ ranges over finite open covers of $X$. 
 
 The concept of entropy pairs was introduced by Blanchard in~\cite{Blanchard1993}. Entropy tuples are a natural generalization that has been studied in several papers.

A \textit{tuple} refers to an element of a product space $Z^k$ for some $k\geq 1$. When $k=2$, we simply call them pairs. 

	 For each $k\in \NN$, we define 
\begin{gather*}
  \Delta_k(X)=\{(x_1,\dots,x_k)\in X^k: x_1=x_2=\dots=x_k\}\text{, and} \\ 
  \Delta^2_k(X)=\{(x_1,\dots,x_k)\in X^k: \text{there exist }n\neq m\in \{1,\dots,k\}\text{ such that } x_n= x_m\}.
\end{gather*}

Note that $\Delta^2_1(X)=\emptyset$. 
  
\begin{definition}
\label{def:entropy tuple}
Let $\Gamma$ be an amenable group, $\act$ an action and $k\geq 2$.
We call $\ox = (x_1, . . . , x_k) \in X^k \setminus \Delta_k(X)$ an \textbf{entropy tuple} if whenever $K_1, . . . , K_l$ (with $l\geq2$) are closed pairwise disjoint neighborhoods of some elements in $\{x_1, . . . , x_k\}$, then
$$
\htop(\act,\mathcal{U})>0,
$$
where $\U=\{K_1^c,\dots, K^c_l\}$. 
\end{definition}

		Let $\Gamma\curvearrowright X$ be an action and $\mathbf{A}=(A_{1},\dots,A_{k})$ a tuple of subsets of $X$. We say $J\subset \Gamma$ is an \textbf{independence set
			for} $\mathbf{A}$ if for every non-empty finite $I\subset J$ and every $\phi \colon I\to\{1,\dots,k\}$ we have
		\[
		\bigcap_{s\in I}s^{-1}A_{\phi(s)}\neq\varnothing\text{.}%
		\]
		We define the \textbf{independence density of} $\mathbf{A}$ to be the largest $q\geq0$ such that every set $F\Subset \Gamma$ has a
		subset of cardinality at least $q\left\vert F\right\vert $ which is an
		independence set for $\mathbf{A}$.

	\begin{definition}
 \label{def:IE pairs}
		Let $\Gamma$ be an amenable group, $\act$ an action, and $k \geq 1$. We say that $\ox=(x_{1},\dots,x_{k})\in X^{k}$ is an \textbf{independence entropy tuple (IE-tuple)} if for every product neighborhood
		$U_{1}\times \dots \times U_{k}$ of $\mathbf{x}$ we have that $(U_{1},\dots,U_{k})$ has
		positive independence density. We denote the set of IE-tuples of length $k$ by
		$\IE_{k}(X,\Gamma)$.
	\end{definition}

 \begin{remark}
 The first definition of IE-tuples of an action of an amenable group (stated in \cite{KerrLi2007}) is not exactly the previous definition, but it is an equivalent notion \cite[Theorem 4.8]{KerrLi2013}. 
 \end{remark}

The following result is proved in \cite[Theorem 3.16]{KerrLi2007}. 
	\begin{theorem}[Kerr and Li \cite{KerrLi2007}]
 \label{thm:IE-tuples}
Let $\Gamma$ be an amenable group, $\act$ an action, $k\geq 2$ and  $(x_1,\dots, x_k) \in X^k \setminus \Delta_k(X)$. Then $(x_1,\dots, x_k)$ is an entropy tuple if and only if it is an IE-tuple. 
     
	\end{theorem}

IE-pairs are witnesses of topological entropy \cite[Proposition 3.9]{KerrLi2007} \cite[Theorem 4.8, Theorem~8.1]{KerrLi2013}. 

\begin{proposition}[Blanchard~\cite{Blanchard1993}; Kerr and Li~\cite{KerrLi2007,KerrLi2013}]
\label{prop:IE pair}
Let $\Gamma$ be an amenable group and $\act$ an action. The following statements are equivalent:
\begin{enumerate}
    \item $\htop(\act)>0$.
    \item $\IE_2(X,\Gamma) \setminus \Delta_2(X) \neq \varnothing$.
    \item $\IE_k(X,\Gamma) \setminus \Delta^2_k(X) \neq \varnothing$ for every $k \geq 2$.
\end{enumerate}
\end{proposition}

The first connection between combinatorial independence and entropy of actions of $\Z$ was explored by Weiss in \cite[Chapter 8]{weiss2000single}. Following this path, Huang and Ye \cite{HuangYe2006} proved a result that is equivalent to \cref{thm:IE-tuples} for actions of $\Z$. The proof relies on combinatorial arguments using the Sauer-Perles-Shelah Lemma and measure-theoretic tools. In order to generalize these results to actions of amenable groups, a different approach was taken in \cite{KerrLi2007}. 

Let $Z$ be a non-empty finite set, $k \geq 2$, and $\U$ the cover of 
$$
\{0, 1,\dots,k\}^Z =\Pi_{z\in Z} \{0, 1,\dots,k\}
$$ consisting of subsets of the form $\Pi_{z\in Z} \{i_z\}^c$, where $1 \leq i_z \leq k$ for each
$z\in Z$. For $S \subset \{0,1,\dots,k\}^Z$ we write $N_S(\U)$ to denote the minimal number of sets in $\U$ one needs to cover $S$.

The following result appeared in~\cite[Lemma 3.3]{KerrLi2007} and is a key combinatorial result in the proof of \cref{prop:IE pair} and \cref{thm:IE-tuples}.

\begin{lemma} [Kerr and Li \cite{KerrLi2007}]
\label{lem:combinatorial} Let $k \geq 2$ and $b > 0$. There exists a constant $c > 0$,
depending only on $k$ and $b$, such that for every non-empty finite set $Z$ and $S \subset \{0, 1,\dots, k\}^Z$ with
$N_S(\U) \geq k^{b|Z|}$, there exists $W \subset Z$ such that $|W | \geq c|Z|$ and $S|_W \supset \{1,\dots, k\}^W$.
\end{lemma}

Here are some of the basic properties that IE-tuples satisfy (proofs can be found in \cite[Proposition 3.9 and Theorem 3.15]{KerrLi2007}). 

\begin{theorem}[Blanchard \cite{Blanchard1993}, Glasner \cite{glasner1997simple}, Kerr and Li \cite{KerrLi2007}]
\label{thm:IE properties} Let $\Gamma$ be an amenable group, $\act$, $\Gamma\curvearrowright Y$ actions, and $k\in \N$. We have that 
    \begin{enumerate}
     
        \item $\IE_{k}(X,\Gamma)$ is a closed subset of $X^k$ that is invariant under the product action. 
        \item $\IE_{k}(X\times Y,G)=\IE_{k}(X,\Gamma)\times \IE_{k}(Y,\Gamma)$.
        \item If $\pi\colon \act \rightarrow \Gamma\curvearrowright Y$ is a factor map, then $$(\pi\times\dots\times \pi)(\IE_{k}(X,\Gamma))=\IE_{k}(Y,\Gamma).$$
        
     \end{enumerate} 
\end{theorem}

Let $k\in \N$. Denote by $\mathcal{F}^k(X)$ the space of all pairwise disjoint tuples of closed subsets of $X$ with length $k$.

\begin{question}Let $\Gamma$ be an amenable group and $\act$ an action. Can $\htop(\act)$ be computed with the use of combinatorial independence? 

\end{question}
In an earlier version of the survey, we asked whether
\[
\htop(\act)=\sup_{k\in\N}\ \sup_{\mathbf{A}\in\mathcal{F}^k(X)}
\bigl(\text{independence density of }\mathbf{A}\bigr)\,\log k
\]
holds. In view of \cref{P-BS} and the existence of uniquely ergodic $\Zb$-actions with arbitrarily large entropy \cite{HK}, the answer is negative.

\begin{proposition}[Burguet and Shi \cite{BScommunication}] \label{P-BS}
Let $\Gamma$ be amenable and let $\Gamma\curvearrowright X$ be a uniquely ergodic action. Then for any $k\in \Nb$ and any tuple $\oA=(A_1, \dots, A_k)$ of pairwise disjoint closed subsets of $X$, one has
$$ (\mbox{independence density of }\oA)\log k\le \frac{1}{3}\log 3. $$
\end{proposition}
\begin{proof}
We write $q$ for the independence density of $\oA$.

Denote by $\mu$ the unique $\Gamma$-invariant Borel probability measure on $X$. Since $\Gamma\curvearrowright X$ is uniquely ergodic, for any continuous $f: X\rightarrow \Rb$ and any $\varepsilon>0$, there are $\delta>0$ and $K\Subset G$ such that for any $(K, \delta)$-invariant $F\Subset \Gamma$ and any $x\in X$, one has \cite[Theorem 3.2.7]{Ollagnier1985book}
$$ \left|\frac{1}{|F|}\sum_{s\in F}f(sx)-\int_Xf\, d\mu\right|<\varepsilon.$$

Let $1\le j\le k$ and $\varepsilon>0$. Since $\mu$ is regular, we can find an open subset $U_j$ of $X$ such that $A_j\subset U_j$ and $\mu(U_j\setminus A_j)<\varepsilon$. By the Urysohn lemma there is a continuous $f: X\rightarrow [0, 1]$ such that $f=0$ on $X\setminus U_j$ and $f=1$ on $A_j$. Then we have $\delta$ and $K$ as above. Take a $(K, \delta)$-invariant $F\Subset G$. Then there is an independence set $F'\subset F$ for $\oA$ such that $|F'|\ge q|F|$. In particular, we can find some $x\in \bigcap_{s\in F'}s^{-1}A_j$. Then
\begin{align*}
2\varepsilon+\mu(A_j)\ge \varepsilon+\mu(U_j)\ge \varepsilon+\int_X f\, d\mu\ge \frac{1}{|F|}\sum_{s\in F}f(sx)\ge \frac{|F'|}{|F|}\ge q.
\end{align*}
Letting $\varepsilon\to 0^+$, we get $\mu(A_j)\ge q$.

As the sets $A_1, \dots, A_k$ are pairwise disjoint, we have
$$1\ge \sum_{j=1}^k\mu(A_j)\ge kq.$$
Thus $q\log k\le \frac{1}{k}\log k$. The function $[1, \infty)\rightarrow \Rb$ that sends $t$ to $\frac{1}{t}\log t$ is increasing on $[1, e]$ and decreasing on $[e, \infty)$. Thus
\begin{equation*}
q\log k\le \frac{1}{k}\log k\le \max\left(\frac{1}{2}\log 2, \frac{1}{3}\log 3\right)=\frac{1}{3}\log 3. \qedhere
\end{equation*}
\end{proof}

\subsection{Sofic groups}
For $n\in \NN$ we write $[n]=\{1,\ldots,n\}$, and we denote by $\Sym(n)$ the group of permutations of $[n]$. A group $\Gamma$ is \textbf{sofic} if there exists a sequence $\left\{ n_{i}\right\}_{i\in \NN}$ of positive integers that tends to infinity and a sequence $\Sigma=\{\sigma_i \colon \Gamma\rightarrow \Sym(n_i) \}_{i\in \N}$ such that
\begin{align*}
\lim_{i\to\infty} \frac{1}{n_i} \left\vert \left\{ v\in [n_i]
: \sigma_{i}(st)v=\sigma_{i}(s)\sigma_{i}(t)v\right\}  \right\vert  &
=1 \mbox{ for every } s,t\in \Gamma \text{, and}\\
\lim_{i\to\infty} \frac{1}{n_i}\left\vert \left\{  v\in [n_i]
: \sigma_{i}(s)v\neq\sigma_{i}(t)v\right\}  \right\vert  & =1 \mbox{ for every } s\neq t\in \Gamma.
\end{align*}
In this case we say that $\Sigma$ is a \textbf{sofic approximation sequence} for $\Gamma$. Residually finite groups and amenable groups are sofic. 
We refer the reader to \cite{CapraroLupini2015, Ceccherini-SilbersteinCoornaert2010, Pestov2008} for general information about sofic groups.	Very recently OpenAI showed that there are groups which are not sofic \cite{OpenAI2026} (see also \cite{Fournier-Facio2026, KunThom2026}). 

	The following notion of topological entropy for sofic group actions was introduced by Kerr and Li~\cite{KerrLi2011}, following Bowen's breakthrough for measure-preserving actions~\cite{Bowen2010_2}.
	
    Let $\Gamma\curvearrowright X$ be an action, $F\Subset \Gamma$, $\delta>0,$ $n\in\mathbb{N}$, and $\sigma\colon \Gamma\rightarrow
	\Sym(n)$.
 \begin{definition}
 \label{def:Map}    

Let $\rho$ be a compatible metric on $X$. We define $\Map(\rho,F,\delta,\sigma)$ as the set of all maps
	$\varphi\colon [n]  \rightarrow  X$ such that
	\[
	\left(  \frac{1}{n}\sum_{v=1}^{n}\rho(\varphi(\sigma(s)v),s\varphi
	(v))^{2}\right)  ^{1/2}\leq\delta \ \mbox{for every }s \in F.
	\]
	 \end{definition}

We write $N_{\varepsilon}(Y, \rho_{\infty})$ for the maximum cardinality of a subset $Y'$ of $Y \subset X^{[n]}$ such that whenever $\varphi_1,\varphi_2$ are distinct in $Y'$ then $$\max_{v \in [n]} \rho(\varphi_1(v),\varphi_2(v)) \geq \varepsilon.$$

	Let $\Gamma$ be a sofic group and $\Sigma = \{\sigma_i \colon G \rightarrow \Sym(n_i)\}_{i\in \N}$ a sofic approximation sequence for $\Gamma$. The \textbf{topological sofic entropy of }$\Gamma\curvearrowright X$
	\textbf{(with respect to }$\Sigma$) is defined by
	\[
	\hsof(\Gamma\curvearrowright X)=\sup_{\varepsilon>0}\inf_{F\Subset \Gamma}\inf_{\delta>0}\limsup_{i\to\infty}\frac{1}{n_{i}}\log N_{\varepsilon}%
	(\Map(\rho,F,\delta,\sigma_i),\rho_{\infty}) \in \{-\infty\}\cup [0,\infty].
	\]	
The value of $\hsof(\Gamma\curvearrowright X)$ does not depend on the choice of $\rho$ \cite[Proposition 10.25]{KerrLiBook2016}, and hence, it is invariant under conjugacy.

\begin{remark}
\label{rem:amenable}
	If $\Gamma$ is an amenable group, then the topological entropy of any action $\act$ coincides with the sofic entropy with respect to any sofic approximation sequence~\cite[Theorem 5.3]{kerrli2013soficity} \cite[Theorem 10.37]{KerrLiBook2016}.
\end{remark}

	Let $\Gamma$ be a sofic group, $\Gamma\curvearrowright X$ an action and $\mathbf{A}=(A_{1},\dots,A_{k})$ a tuple of
	subsets of $X$. Given $F\Subset \Gamma$, $\delta>0$, $n\in\mathbb{N}$ and
	$\sigma\colon G\rightarrow \Sym(n)$, we say $J\subset[n]$ is a
	$(\rho,F,\delta,\sigma)$-\textbf{independence set for $\mathbf{A}$} if for every
	$\omega\colon J\rightarrow [k]  $ there exists $\varphi\in
	\Map(\rho,F,\delta,\sigma)$ such that $\varphi(v)\in A_{\omega(v)}$ for every $v\in
	J$.

\begin{definition} 
	Let $\Gamma$ be a sofic group, $\Gamma\curvearrowright X$ an action and $\Sigma=\left\{  \sigma_{i} \colon G \rightarrow \Sym(n_{i})\right\}
	_{i\in\mathbb{N}}$ a sofic approximation sequence for $\Gamma$. We say that a tuple of subsets of $X$, $\mathbf{A}=(A_{1},\dots,A_{k})$, 
has \textbf{positive upper independence density over $\Sigma$}, if there exists $q>0$ such that for every $F\Subset \Gamma$ and
	$\delta>0$, $\mathbf{A}$ has a $(\rho,F,\delta,\sigma_{i})$-independence set of cardinality at least $qn_{i}$ for infinitely many $i$'s.

We say $\mathbf{x}=(x_{1},\dots,x_{k})\in X^{k}$ is a \textbf{sofic independence entropy tuple with respect to $\Sigma$ ($\Sigma$-IE-tuple)} if for every product
	neighborhood $U_{1}\times\cdot\cdot\cdot\times U_{k}$ of
	$\mathbf{x}$ the tuple $\mathbf{U}=(U_1,\dots,U_k)$ has positive upper independence density over $\Sigma$.
We denote
	the set of $\Sigma$-\textbf{IE-tuples} of length $k$ by $\IE_{k}^{\Sigma}(X,\Gamma)$.
\end{definition}
 The previous property does not depend on the choice of the metric $\rho$ \cite[Lemma 10.24]{KerrLiBook2016}.

\begin{remark}
If $\Gamma$ is an amenable group, the set of $\Sigma$-IE-tuples for any sofic approximation sequence $\Sigma$ is the same as the set of IE-tuples \cite[Theorem 4.8]{KerrLi2013}.     
\end{remark}

\begin{remark}
 Zhang \cite{zhang2012local} defined entropy tuples with respect to a sofic approximation sequence $\Sigma$, using the topological sofic entropy with respect to a finite open cover, analogous to \cref{def:entropy tuple}. As it was noted in \cite[Remark 4.4]{KerrLi2013}, a non-diagonal tuple is an entropy tuple with respect to $\Sigma$ if and only if it is a $\Sigma$-IE-tuple.
\end{remark}

For the proof of the next result see \cite[Proposition 4.16 and Theorem 8.1]{KerrLi2013}. 
	\begin{theorem}
 [Kerr and Li \cite{KerrLi2013}]
 \label{thm: sofic IE pairs}
		 Let $\Gamma$ be sofic, $\Sigma$ a sofic approximation sequence for $\Gamma$,  and
		$\Gamma\curvearrowright X$ an action. The following statements are equivalent.
\begin{enumerate}
    \item  $\hsof(\act)>0$.
    \item $\IE^{\Sigma}_2(X,\Gamma)\setminus \Delta_2(X) \neq \varnothing$.
    \item $\IE^{\Sigma}_k(X,\Gamma)\setminus \Delta^2_k(X) \neq \varnothing$ for every $k\geq 2$.
\end{enumerate}


	\end{theorem}

The next result is obtained from \cite[Proposition 4.16]{KerrLi2013}. Note that in the second point we do not get equality as in \cref{thm:IE properties}. This comes from the fact that factors of actions with zero sofic entropy may have positive sofic entropy.

\begin{theorem}[Kerr and Li \cite{KerrLi2013}]
\label{thm:sofic IE properties} Let $\Gamma$ be a sofic group, $\Sigma$ a sofic approximation sequence for $\Gamma$, $\act$, $\Gamma\curvearrowright Y$ actions and $k\in \N$. We have that 
    \begin{enumerate}
        \item $\IE^{\Sigma}_{k}(X,\Gamma)$ is a closed subset of $X^k$ that is invariant under the product action. 
        \item If $\pi\colon \act \rightarrow \Gamma\curvearrowright Y$ is a factor map, then $$(\pi\times\dots\times \pi)(\IE^{\Sigma}_{k}(X,\Gamma))\subset\IE^{\Sigma}_{k}(Y,\Gamma).$$
        
    \end{enumerate}
    
\end{theorem}

A product formula for $\Sigma$-IE-tuples is possible under extra assumptions on $\Sigma$ \cite[Theorem 5.2]{KerrLi2013}.

\subsection{Countable groups}
Let $\Gamma \curvearrowright X$ be an action and $\mathcal{U}$ a finite open cover of $X$. We define the
		\textbf{naive topological entropy of $\Gamma\curvearrowright X$ with respect to $\mathcal{U}$} as%
		\[
		\hnaive(\act,\mathcal{U})=\inf_{F\Subset \Gamma}\frac{1}{\left\vert
			F\right\vert }\log N(\mathcal{U}^{F})\in [0,\infty) .
		\]
The \textbf{naive topological
		entropy} of $\Gamma\curvearrowright X$ is defined as
	\[
	\hnaive(\act)=\sup_{\mathcal{U}}\hnaive(\act,\mathcal{U}) \in [0,\infty],
	\]
where $\U$ ranges over the set of finite open covers. 
	The notion of naive topological entropy was introduced by Burton~\cite{burton2017naive}. He showed that if $\Gamma$ is not amenable then $\hnaive(\act)$ can only take the values in the set $\{0, \infty\}$.

\begin{remark}
	If $\Gamma$ is amenable, then the topological entropy of any action $\act$ coincides with the naive topological entropy~\cite[Theorem 6.8]{DownarowiczFrejRomagnoli2016}.
\end{remark}

The next definition uses notions from Section 2.1, and it is very similar to \cref{def:IE pairs}, the only difference being that there is no condition on the group. For this reason we use the same notation. 

	\begin{definition}
		Let $\act$ be an action and $k \geq 1$ an integer. We say a tuple $\mathbf{x}=(x_{1},\dots,x_{k})\in X^{k}$ is an \textbf{orbit independence entropy tuple (orbit IE-tuple)} if for every product neighborhood
		$U_{1}\times \dots \times U_{k}$ of $\mathbf{x}$, we have that $(U_{1},\dots,U_{k})$ has
		positive independence density. We denote the set of orbit IE-tuples of length $k$ by
		$\IE_{k}(X,\Gamma)$.
	\end{definition}

\begin{remark}
If $\Gamma$ is a sofic group and $\Sigma$ a sofic approximation sequence for $\Gamma$, then every $\Sigma$-IE-tuple is an orbit IE-tuple~\cite[Proposition 4.6]{KerrLi2013}, but it is not known if there exist orbit IE-tuples that are not $\Sigma$-IE-tuples. 
\end{remark}

Positive naive entropy can be characterized with combinatorial independence but not at the level of orbit IE-pairs.

We now define a notion of multiple independence. 

\begin{definition}
For any finite subset $\mathcal{A} \subset \mathcal{F}^2(X)$, we define the {\bf independence density} of $\mathcal{A}$ to be the largest $q\ge 0$ such that for every $F\Subset \Gamma$ there are some $J\subset F$ with $|J|\ge q|F|$ and some $ \mathbf{A}\in \mathcal{A}$ so that $J$ is an independence set for $\mathbf{A}$.
\end{definition}

The proof of the following result appears in \cite[Theorem 2.5]{LiRong2019}.

		\begin{theorem}
  [Li and Rong \cite{LiRong2019}]
  \label{thm:naiveentropy}
		Let $\Gamma\curvearrowright X$ be an action.
	 We have that $\hnaive(\act)>0$ if and only if there is a finite $\mathcal{A}\subset \mathcal{F}^2(X)$ with positive independence density. 
   In particular, if $\IE_2(X,\Gamma)\setminus \Delta_2(X) \neq \varnothing$, then $\hnaive(\act)>0$.


	\end{theorem}
	
	There exists an action of the free group $\mathbb{F}_2$ with positive naive topological entropy and no non-trivial orbit IE-pairs \cite[Proposition 5.2]{LiRong2019}.

\subsection{CPE and UPE}

As mentioned in the introduction, one of the original motivations of studying entropy pairs was to obtain a better comprehension of uniform positive entropy and completely positive entropy. 

We say a finite open cover is \textbf{non-dense} if every element of the cover is not dense. 
\begin{definition}
Let $\act$ be an action. We say that $\act$ has \textbf{naive uniform positive entropy (naive UPE)} if for every non-dense open cover, $\U=\{U_1,U_2\}$, of $X$, we have
$$
\hnaive(\act,\U)>0.
$$

If the group is amenable, we simply call this property \textbf{UPE}. 
\end{definition}

The following result was proved in \cite{Blanchard1993} for actions of $\Z$, in \cite{KerrLi2007} for actions of amenable groups (with the use of \cref{lem:combinatorial}). The same proof works for actions of countable groups. 

\begin{proposition}
[Blanchard \cite{Blanchard1993}, Kerr and Li \cite{KerrLi2007}]
  Let $\act$ be an action. We have $\IE_2(X,\Gamma)=X^{2}$ if and only if $\Gamma\curvearrowright X$ has naive UPE.
\end{proposition}

  Given a sofic approximation sequence for $\Gamma$, we say $\Gamma\curvearrowright X$ has \textbf{$\Sigma$-UPE} if $\IE_2^{\Sigma}(X,\Gamma)=X^{2}$. While not used in the current literature, $\Sigma$-UPE is equivalent to having positive topological sofic entropy for every non-dense open cover consisting of two sets \cite[Remark 4.4]{KerrLi2013}.

\begin{definition}
    Let $k\geq 2$. We say that an action $\act$ has \textbf{naive UPE of order $k$} if $\IE_k(X,\Gamma)=X^{k}$. If the group is amenable, we simply call this property \textbf{UPE of order $k$}. 
\end{definition}

UPE behaves similarly when considering different forms of entropy; with CPE, one has to be more careful because it deals with factors. 

\begin{definition}
Let $\Gamma$ be an amenable group and $\act$ an action. We say that $\act$ has \textbf{completely positive entropy (CPE)} if every non-trivial factor has positive topological entropy. 
\end{definition}

CPE can also be characterized using the \textbf{topological Pinsker factor} for actions of amenable groups, that is, the largest factor with zero topological entropy in the sense that every factor with zero topological entropy is also a factor of the topological Pinsker factor. An action of an amenable group has CPE if and only if the topological Pinsker factor is trivial. 
 Using \cref{prop:IE pair} and \cref{thm:IE properties} (3) we obtain the following result. 
\begin{corollary}
[Blanchard and Lacroix \cite{blanchard1993zero}]
      Let $\Gamma$ be an amenable group and $\act$ an action. The topological Pinsker factor of $\act$ is obtained as the quotient of the smallest closed $\Gamma$-invariant equivalence relation on $X$ containing the set of IE-pairs.
  
In particular, $\act$ has CPE if and only if the smallest closed $\Gamma$-invariant equivalence relation on $X$ containing the set of IE-pairs is $X^2$.
      
\end{corollary}
    
As mentioned before, there exists an action with positive naive topological entropy and no non-trivial orbit IE-pairs \cite[Proposition 5.2]{LiRong2019}. This implies that the previous corollary does not hold for actions of countable groups in the naive entropy context.  

Given a sofic group $\Gamma$ and  $\Sigma$ a sofic approximation sequence for $\Gamma$, we may also define the \textit{topological $\Sigma$-Pinsker factor} for actions with $\htop^{\Sigma}(\act)\geq 0$. The existence of these factors can be proved using \cite[Lemma 2.13]{li2013soficmean}. A topological $\Sigma$-Pinsker factor may still have factors with positive sofic topological entropy. If every non-trivial factor has positive sofic topological entropy with respect to $\Sigma$ we say it has \textbf{$\Sigma$-CPE}. Using \cref{thm: sofic IE pairs} and \cref{thm:sofic IE properties} (2), one can see that if the smallest closed $\Gamma$-invariant equivalence relation on $X$ containing the set of $\Sigma$-IE-pairs is $X^2$ then $\act$ has $\Sigma$-CPE. It remains unknown whether the converse holds.

\begin{question}
      Let $\Gamma$ be a sofic group, $\Sigma$ a sofic approximation sequence for $\Gamma$, and $\act$ an action with $\Sigma$-CPE. Is $X^2$ the smallest closed $\Gamma$-invariant equivalence relation on $X$ containing $\IE_2^{\Sigma}(X,\Gamma)$?
\end{question}

\section{Variations}

The definition of entropy tuples is remarkably flexible: a number of natural variations have arised and connect local entropy theory with other parts of dynamics, including equicontinuity, sequence entropy, the Ellis semigroup, nilsystems, and mean dimension.

\subsection{Equicontinuity and regional proximality}

The notions of equicontinuity and regional proximality predate local entropy theory and play a central role in the structural theory of minimal systems. 
As we will see, they are also closely related, both technically and conceptually, to natural variants of IE-tuples.

In this subsection, we survey the definitions and main results of such topics, and we will also present a new result concerning regionally proximal tuples.

A structure theorem in topological dynamics usually involves the description of its natural factors or extensions. 
Factor descriptions often involve quantifying how close a factor map is to being bijective. 
Let $k\in\N$ and $\pi\colon (X,\Gamma)\rightarrow (Y,\Gamma)$ be a factor map. We say that $\pi$ is \textbf{almost $k$-to-$1$} if $k$ is the smallest integer such that the $G_\delta$ set
\[
\bigl\{x\in X:\ |\pi^{-1}(\pi(x))|\le k\bigr\}
\]
is dense in $X$.

One of the most natural classes of factors is the class of equicontinuous factors. 
An action $\act$ is \textbf{equicontinuous} if for every $\varepsilon>0$ there exists $\delta>0$, such that if $\rho(x,y)\leq \delta$ then $\rho(gx,gy)\leq \varepsilon$ for all $g\in \Gamma$. Equicontinuity is invariant under conjugacy. An action $\act$ is \textbf{minimal} if $X$ is the only non-empty closed $\Gamma$-invariant set.  
Minimal equicontinuous actions are uniquely ergodic.

The \textbf{maximal equicontinuous factor}, $\Gamma\curvearrowright X_{\eq}$, is the largest equicontinuous factor, that is, every equicontinuous factor is also a factor of $\Gamma\curvearrowright X_{\eq}$. 
The factor map to the maximal equicontinuous factor will be denoted by $\pi_{\eq}$. Given a minimal action, we denote the $\Gamma$-invariant Borel probability measure of the maximal equicontinuous factor by $\nu_{\eq}$. An action is \textbf{almost automorphic} if $\pi_{\eq}$ is almost 1-to-1.

For an overview of minimal systems, structure theorems, and maximal equicontinuous factors, see \cite{auslander1988minimal}. 

\begin{definition}
Let $\act$ be an action and $k\geq 2 $ an integer. We say $(x_1,\dots,x_k)\in X^k$ is an \textbf{equicontinuous structure tuple (or Q-tuple)} if $\pi_{\eq}(x_i)=\pi_{\eq}(x_j)$ for every $i,j\in\{1,\dots,k\}$. We denote the set of Q-tuples of length $k$ by $\Q_k(X,\Gamma)$.  
\end{definition}

\begin{definition}
    Let $\act$ be an action. We define the \textbf{minimal Q-rank} of $\act$ as
    \[
    r_{\Q}(X,\Gamma)=\inf\bigl\{\,|\pi_{\eq}^{-1}(y)|:y\in X_{\eq}\,\bigr\},
    \]
    and the \textbf{maximal Q-rank} of $\act$ as
    \[
    R_{\Q}(X,\Gamma)=\sup\bigl\{\,|\pi_{\eq}^{-1}(y)|:y\in X_{\eq}\,\bigr\}.
    \]
\end{definition}
Since $\Delta^2_1(X)=\emptyset$, then $R_Q(X,\Gamma)=1$ if and only if $\act$ is equicontinuous. If $\act$ is minimal, then $\pi_{\eq}$ is almost $k$-to-1 if and only if $r_Q(X,\Gamma)=k$; in particular,
$\act$ is {almost automorphic} if and only if $r_Q(X,\Gamma)=1$. Non-periodic Toeplitz $\Z$-subshifts are minimal, almost automorphic but not equicontinuous \cite{downarowicz2005survey}. 

The relation given by $\Q_2(X,\Gamma)$ can also be studied without the use of factors, using regional proximality. Regionally proximal pairs were introduced by Ellis and Gottschalk \cite{EllisGottschalk1960Homomorphisms}. Regionally proximal tuples were introduced by Auslander \cite{auslander2004group}.

\begin{definition}
Let $\act$ be an action and let $\rho$ be a compatible metric on $X$. For any integer $k\ge 2$, we say $(x_1, \dots, x_k)\in X^k$ is {\bf regionally proximal} if for any $\varepsilon>0$, there are $(y_1, \dots, y_k)\in X^k$ and $s\in \Gamma$ such that $\rho(x_j, y_j)<\varepsilon$ for all $1\le j\le k$ and $\rho(sy_i, sy_j)< \varepsilon$ for all $1\le i, j\le k$. We denote by $\RP_k(X,\Gamma)$ the set of all regionally proximal tuples of length $k$. We set $\RP_1(X,\Gamma)=X$.
\end{definition}

    If $\Gamma\curvearrowright X$ is minimal, then $(x_1, \dots, x_k)\in \RP_k(X,\Gamma)$ if and only if, for any neighborhoods $U_1, \dots, U_k$ of $x_1, \dots, x_k$ respectively, there is some $s\in \Gamma$ such that $U_j\cap s^{-1}U_k\neq \emptyset$ for all $1\le j\le k$.

In general $\RP_k(X,\Gamma)$ is a closed $\Gamma$-invariant subset of $X^k$, but $\RP_2(X,\Gamma)$ may not be an equivalence relation. We define $\RP_2^*(X,\Gamma)$ as the smallest closed equivalence relation that contains $\RP_2(X,\Gamma)$.

For a proof of \cref{veech2}, see~\cite[Theorem 9.3 and Theorem 9.8]{auslander1988minimal}.
\begin{theorem}
[Ellis and Gottschalk \cite{EllisGottschalk1960Homomorphisms}, McMahon~\cite{McMahon1978RelativizedWeakDisjointness}]
\label{veech2}
    Let $\act$ be an action. We have that $$\Q_2(X,\Gamma)=\RP_2^*(X,\Gamma).$$

    Furthermore, if $\mathcal{M}(\act)\neq \emptyset$ and $\act$ is minimal, then $$\Q_2(X,\Gamma)=\RP_2(X,\Gamma),$$ and hence $\RP_2(X,\Gamma)$ is an equivalence relation.
\end{theorem}

Note that every action of an amenable group admits an invariant Borel probability measure (see, e.g., \cite[Theorem 4.4]{KerrLiBook2016}).

One may wonder what the relation between regionally proximal pairs and tuples is. A version of \cref{RP-pair to tuple} was proven by Auslander, where the assumption $\mathcal{M}(\act)\neq \emptyset$ is replaced by the local Bronstein condition~\cite[Theorem~2]{Auslander2001MinimalFlowsClosedProximalCell}~\cite[Theorem 8]{auslander2004group}.

\begin{theorem} \label{RP-pair to tuple}
Let $\Gamma\curvearrowright X$ be a minimal action with $\mathcal{M}(\act)\neq \emptyset$. Let $k\ge 2$ be an integer, and let $x_1, \dots, x_k\in X$. The following are equivalent:
\begin{enumerate}
\item $(x_1, x_j)\in \RP_2(X,\Gamma)$ for all $2\le j\le k$;
\item $(x_1, \dots, x_k)\in \RP_k(X,\Gamma)$;
\item for any open neighborhoods $U_2, \dots, U_k$ of $x_2, \dots, x_k$ respectively, there is some $s\in \Gamma$ such that $sx_1\in U_k$ and $sU_j\cap U_k\neq \emptyset$ for all $2\le j\le k$.
\end{enumerate}
\end{theorem}

Before proving \cref{RP-pair to tuple}, we provide a corollary and a couple of lemmas.

\begin{corollary}
Let $\Gamma\curvearrowright X$ be a minimal action with $\mathcal{M}(\act)\neq \emptyset$. Then
\[
\begin{aligned}
R_{\Q}(X,\Gamma) &= \sup\bigl\{\, k\in\N : \RP_k(X,\Gamma)\not\subset \Delta_k^2(X) \,\bigr\},\\
r_{\Q}(X,\Gamma) &= \inf\bigl\{\, k\in\N : \exists (x_1,\dots,x_k)\in \RP_k(X,\Gamma)\setminus \Delta_k^2(X)\ \text{such that}\\
&\hspace{3.6em} (x_1,y)\in \RP_2(X,\Gamma)\Rightarrow (x_1,\dots,x_k,y)\in \Delta_{k+1}^2(X)\,\bigr\}.
\end{aligned}
\]

\end{corollary}

The following lemma is proved implicitly in the proof of \cite[Theorem 9.8]{auslander1988minimal}. We provide a proof for completeness.

Given $\mu\in\mathcal{M}(X)$, we denote the support of $\mu$ by $\supp(\mu)$.

\begin{lemma}
\label{L-RP and measure to intersection}
Let $\Gamma\curvearrowright X$ and $\Gamma\curvearrowright Y$ be actions. Assume that $\Gamma\curvearrowright X$ is minimal and that there exists $\mu\in\mathcal{M}(\Gamma\curvearrowright Y)$ with $\supp(\mu)=Y$. Let $(x, z)\in \RP_2(X,\Gamma)$ and let $U$ be an open neighborhood of $z$. Let $V$ be a nonempty open subset of $Y$. Then there is some $s\in \Gamma$ such that $sV\cap V\neq \emptyset$ and $sx\in U$.
\end{lemma}
\begin{proof} Let $Z$ be the closure of $\bigcup_{s\in \Gamma} s(\{x\}\times V)$ in $X\times Y$. For each $x'\in X$, set
$$Z_{x'}=\{y\in Y: (x', y)\in Z\}\subset Y.$$
By hypothesis, $Y$ has a $\Gamma$-invariant Borel probability measure $\mu$ with full support.
For $x', x''\in X$, set
$$\rho_{\mu}(x', x'')=\mu(Z_{x'}\Delta Z_{x''}).$$
By \cite[Lemma 9.5]{auslander1988minimal}, $\rho_{\mu}$ is a $\Gamma$-invariant continuous pseudometric on $X$. Let $E$ be the set of $(x', x'')\in X^2$ satisfying $\rho_{\mu}(x', x'')=0$.
By \cite[Lemma 9.6]{auslander1988minimal}, $E$ is a $\Gamma$-invariant equivalence relation containing $\Q_2(X,\Gamma)$. Since
$$(x, z)\in \RP_2(X,\Gamma)\subset \Q_2(X,\Gamma)\subset E,$$
we have $\rho_{\mu}(x, z)=0$. As $V\subset Z_x$, we have
$$ \mu(V\setminus Z_z)\le \mu(Z_x\Delta Z_z)=\rho_{\mu}(x, z)=0.$$
Since $\mu$ has full support, the open set $V\setminus Z_z$ must be empty. That is, $V\subset Z_z$. Then $(U\times V)\cap Z\supset \{z\}\times V$, and hence the open set $U\times V$ intersects with $Z$. Consequently, we can find some $s\in \Gamma$ such that $s(\{x\}\times V)$ intersects with $U\times V$. This means $sx\in U$ and $sV\cap V\neq \emptyset$.
\end{proof}

\begin{lemma} \label{L-pair to tuple}
Let $\Gamma\curvearrowright X$ be a minimal action with $\mathcal{M}(\act)\neq\emptyset$. Let $k\ge 2$ be an integer, and let $x_1, \dots, x_k\in X$ such that $(x_1, x_j)\in \RP_2(X,\Gamma)$ for all $2\le j\le k$. Let $U_2, \dots, U_k$ be open neighborhoods of $x_2, \dots, x_k$ respectively. Then there is some $s\in \Gamma$ such that $sx_1\in U_k$ and $sU_j\cap U_k\neq \emptyset$ for all $2\le j\le k$.
\end{lemma}
\begin{proof} We argue by induction on $k$. For $k=2$ this is \cite[Corollary 9.9]{auslander1988minimal}. Assume that the assertion holds for some $k\ge 2$.

Let $x_1, \dots, x_{k+1}\in X$ such that $(x_1, x_j)\in \RP_2(X,\Gamma)$ for all $2\le j\le k+1$, and let $U_2, \dots, U_{k+1}$ be open neighborhoods of $x_2, \dots, x_{k+1}$ respectively. From \cref{veech2} we know that $(x_2, x_j)\in \RP_2(X,\Gamma)$ for all $3\le j\le k+1$. By the induction hypothesis we can find some $s\in \Gamma$ such that $sx_2\in U_{k+1}$ and $sU_j\cap U_{k+1}\neq \emptyset$ for all $3\le j\le k+1$.
Then $V_j=U_j\cap s^{-1}U_{k+1}$ is a nonempty open subset of $X$ for each $2\le j\le k+1$.

Consider the diagonal action of $\Gamma$ on $Y=X^k$. By assumption $X$ has a $\Gamma$-invariant Borel probability measure $\mu$. Since the action $\Gamma\curvearrowright X$ is minimal, $\mu$ has full support. Then $\mu^k$ is a $\Gamma$-invariant Borel probability measure on $Y$ with full support. Note that $V=V_2\times \cdots \times V_{k+1}$ is a nonempty open subset of $Y$.
Applying \cref{L-RP and measure to intersection} to $(x, z)=(x_1, x_2)$ and $U=V_2$, we find some $t\in \Gamma$ such that $tV\cap V\neq \emptyset$ and $tx_1\in V_2$. Then for each $2\le j\le k+1$, one has
$$ tV_j\cap V_j\neq \emptyset,$$
whence
$$ stU_j\cap U_{k+1}\supset stV_j\cap sV_j\neq \emptyset.$$
Also,
$$stx_1\in sV_2\subset U_{k+1}.$$
This finishes the proof for the induction step and completes the proof of the lemma.
\end{proof}

\begin{proof}[Proof of \cref{RP-pair to tuple}] (1)$\Rightarrow$(3) is \cref{L-pair to tuple}, and (3)$\Rightarrow$(2)$\Rightarrow$(1) is trivial. 
\end{proof}

\subsection{Sequence entropy and IN-tuples}

Let $\act$ be an action. Following Goodman \cite{SE}, for a sequence $S = \{s_n\}_{n\in\Nb}$ in $\Gamma$ we define the \textbf{topological sequence entropy of $\act$ along $S$ with respect to a finite open cover $\cU$} by
\[ \htop^S (\act,\cU ) = \limsup_{n\to\infty} \frac1n \log
N\bigg( \bigvee_{i=1}^n s_i^{-1} \cU \bigg) \in[0,\infty). \]
We define
\[\htop ^{S}(\act) = \sup_{\cU} \htop^S (\act,\cU ) \text{, and} \]
\[\htop ^{*}(\act) = \sup_{\U,S} \htop^S (\act,\cU ), \]
where the suprema are taken over all finite open covers $\U$ and all sequences $S$ in $\Gamma$.
Huang, Li, Shao, and Ye \cite{huang2003null} and Huang, Maass, and Ye \cite{huang2004sequence} emulated the entropy pair/tuple definition for sequence entropy.
\begin{definition}
Let $\act$ be an action and $k\geq 2$. We say that $\ox = (x_1, \dots, x_k)\in X^k\setminus \Delta_k(X)$ is a \textbf{sequence
entropy tuple} if whenever $K_1, \dots, K_l$ (with $l\geq 2$) are closed pairwise
disjoint neighborhoods of some elements of
$\{x_1, \dots, x_k\}$, the open cover
$\{K_1^\comp, \dots, K_l^\comp\}$ has positive topological sequence
entropy with respect to some sequence in $\Gamma$.
\end{definition}

\begin{definition}\label{D-IN-pair}
Let $\act$ be an action and $k\geq 1$. We say that $\ox = (x_1, \dots, x_k)\in X^k$ is an \textbf{IN-tuple} if for any product
neighborhood $U_1 \times\cdots\times U_k$ of $\ox$, we have that $(U_1,
\dots, U_k)$ has arbitrarily large finite independence sets. We
denote the set of IN-tuples of length $k$ by $\IN_k(X,\Gamma)$. 
\end{definition}

For the proof of the next result, see \cite[Theorem 5.9]{KerrLi2007}.

\begin{theorem}
[Kerr and Li \cite{KerrLi2007}]
\label{P-SE vs IN}
Let $\act$ be an action, $k\geq 2$ and $(x_1, \dots, x_k)\in X^k\setminus \Delta_k(X)$. Then $(x_1, \dots, x_k)$ is a sequence entropy tuple
if and only if it is an IN-tuple.
\end{theorem}

IN-tuples and IE-tuples have similar basic properties \cite[Proposition 5.4]{KerrLi2007}. 
\begin{theorem}[Huang, Li, Shao and Ye \cite{huang2003null}, Kerr and Li \cite{KerrLi2007}]
\label{thm:IN properties}  Let $\act$ and $\Gamma\curvearrowright Y$ be actions and $k\in \N$. We have that 
    \begin{enumerate}
        \item $\IN_{k}(X,\Gamma)$ is a closed subset of $X^k$ that is invariant under the product action. 
        \item If $\pi\colon \act \rightarrow \Gamma\curvearrowright Y$ is a factor map, then $$(\pi\times\dots\times \pi)(\IN^{}_{k}(X,\Gamma))=\IN^{}_{k}(Y,\Gamma).$$
        
    \end{enumerate}
    
\end{theorem}

{
\begin{definition}
    Let $\Gamma\curvearrowright X$ be an action. We define the \textbf{IN-rank} of $\act$ as
\[
\begin{aligned}
R_{\IN}(X,\Gamma) &= \sup\bigl\{\, k\in\N : \IN_k(X,\Gamma)\not\subset \Delta_k^2(X) \,\bigr\}.
\end{aligned}
\]

\end{definition}

Topological sequence entropy is always the logarithm of an integer or infinity \cite[Theorem 4.4]{HuangYe2009}.
\begin{theorem} 
[Huang and Ye \cite{HuangYe2009}]
\label{thm:seqentr}
Let $\act$ be an action. We have that 
\[
\htop^{*}(\act)=\log R_{\IN}(X,\Gamma).
\]
\end{theorem}
We say that $(X,\Gamma)$ is {\bf null} if $\htop^* (\act)=0$. 
\begin{corollary}
[Huang, Li, Shao and Ye \cite{huang2003null}, Kerr and Li \cite{KerrLi2007}]
\label{prop:INpair}
Let $\act$ be an action. 
    Then $\act$ is null if and only if $R_{\IN}(X,\Gamma)=1$. 
\end{corollary}
For a direct proof of \cref{prop:INpair}, see \cite[Proposition 5.4]{KerrLi2007}.

Let $k\geq 2$ be an integer. A minimal action $\act$ is \textbf{$k$-regular} if the $G_{\delta}$ set
$$Z=\{z\in X_{\eq}:|\pi_{\eq}^{-1}(z)|\leq k-1\}$$
satisfies $\nu_{\eq}(Z) = 1$.

When $\act$ is $2$-regular, we simply say it is \textbf{regular}. Every minimal regular action is uniquely ergodic and almost automorphic.

Almost automorphy and unique ergodicity were proven for minimal null actions of $\Z$ by Huang, Li, Shao, and Ye \cite{huang2003null}. Regularity was proved for minimal null subshifts on amenable groups by García-Ramos \cite{garcia2017weak} and for minimal null actions of $\Z$ by García-Ramos, Jäger, and Ye \cite{garcia2021mean}.

Structure theorems for minimal actions with bounded sequence entropy have been studied by Maass and Shao \cite{maass2007structure}, Huang, Lian, Shao, and Ye \cite{huang2021minimal}, and Liu, Wang, and Xu \cite{liu2025sequenceIT}. We present a new generalization.
\begin{theorem}
\label{boundedseq}
Let $\act$ be a minimal action with $\mathcal{M}(\act)\neq \emptyset$, and let $k \geq 2$ be an integer with $\htop^{*}(\act)\leq \log (k-1)$. Then $\act$ is $k$-regular; consequently, $r_{\Q}(X,\Gamma)\leq k-1$.
\end{theorem}
\begin{proof}
This follows from \cref{thm:seqentr} and \cref{chumin1}.
\end{proof}

On the other hand, null actions can have infinite maximal Q-rank.

\begin{proposition}
There exists a minimal null action $\Zb\curvearrowright X$ with $R_{\Q}(X,\Zb)=\infty.$ 
\end{proposition}
\begin{proof}
Let $p$ be a prime number. 
Let $Z$ be the projective limit of $\Zb/p^n\Zb$ under the natural map $$q_n: \Zb/p^{n}\Zb\rightarrow \Zb/p^{n-1}\Zb.$$ Then $Z$ is a compact abelian group, and we have a natural injective map $\Zb\hookrightarrow Z$. Take a $b_n\in \Zb$ for each $n\in \Nb$ such that the following hold:
\begin{enumerate}
\item[(i)] $b_{n+1}= b_n \mod p^n$ for each $n$;
\item[(ii)] there is no $m\in \Zb$ satisfying $m= b_n \mod p^n$ for all $n$.
\end{enumerate}
It is easy to find such $b_n$. We set $b_0=0$. 

Let $A$ be the one-point compactification $\Nb\cup \{\infty\}$ of $\Nb$. 
For each $n\in \Nb$, we take an injective map 
$$\varphi_n: q_n^{-1}(b_{n-1}+p^{n-1}\Zb)\setminus \{b_n+p^n\Zb\}\rightarrow \Nb.$$ 
We construct an $f: \Zb\rightarrow A$ as follows. For $m\in \Zb$, by (ii), there is a smallest $n\in \Nb$ such that $$m+p^n\Zb\neq b_n+p^n\Zb,$$. Writing this $n$ as $n_m$, we set 
$$f(m)=\varphi_{n_m}(m+p^{n_m}\Zb)\in \Nb.$$
Then $f(m+kp^{n_m})=f(m)$ for all $k\in \Zb$. Thus, $f: \Zb\rightarrow A$ is continuous when $\Zb$ is equipped with the relative topology $\Zb\subset Z$. That is, $f$ is a semicocycle on $Z$ in the sense of \cite[Definition 6.1]{downarowicz2005survey}. Denote by $X$ the orbit closure of $f$ in $A^\Zb$. By \cite[Theorems 5.1 and 6.2]{downarowicz2005survey} $\Zb\curvearrowright X$ is minimal and almost automorphic.

We shall show that $\Zb\curvearrowright X$ is null. For this, we need some preparation. Let $m\in \Zb$ be nonzero. Then there is the largest $N\ge 0$ such that $m\in p^N\Zb$. We write this $N$ as $\ord_p(m)$. Then $\ord_p(m)=\ord_p(-m)$.

Let $k\neq l$ in $\Zb$ and $m\in \Zb$ such that $f(k+m)\neq f(l+m)$. If $n_{k+m}<n_{l+m}$, then $\ord_p(k-l)=n_{k+m}-1$. If $n_{k+m}>n_{l+m}$, then $\ord_p(k-l)=n_{l+m}-1$. If $n_{k+m}=n_{l+m}$, since $f(k+m)\neq f(l+m)$, we also have $\ord_p(k-l)=n_{k+m}-1$. We obtain the following formula, which will be used in the proof
\begin{align} \label{E-infinite fiber1}
 \ord_p(k-l)+1=\min(n_{k+m}, n_{l+m}).
 \end{align}

For $a\in \Rb$ we write $\lceil a\rceil$ for the smallest $m\in \Zb$ satisfying $m\ge a$.

Let $N\ge 0$ and let $F$ be a finite subset of $\Zb$ such that 
$$\ord_p(k-l)\ge N \mbox{ for all distinct } k, l\in F.$$ 
We claim that there is a subset $F'$ of $F$ with $|F'|\ge \lceil |F|/p\rceil$ such that 
$$\ord_p(k-l)\ge N+1 \mbox{ for all distinct } k, l\in F'.$$ 
We first find some $a\in \Zb$ such that $k-a\in p^N\Zb$ for all $k\in F$. Then $(k-a)/p^N\in \Zb$ for all $k\in F$. Thus we can find a subset $F'$ of $F$ with $|F'|\ge \lceil |F|/p\rceil$ such that $(k-a)/p^N+p\Zb$ is the same for all $k\in F'$. Then $\ord_p(k-l)\ge N+1$ for all distinct $k, l\in F'$. This proves our claim. 

Let $F$ be a finite subset of $\Zb$ with $|F|\ge p+2$. We claim that there are distinct $k, l, m\in F$ such that 
$$\ord_p(l-k)<\ord_p(m-l).$$ 
We first find distinct $k_1, k_2\in F$ such that $\ord_p(k_2-k_1)$ is the smallest value among $\ord_p(l-k)$ for distinct $l, k\in F$. If there is some $k_3\in F\setminus \{k_1, k_2\}$ with $\ord_p(k_3-k_1)>\ord_p(k_2-k_1)$, then we may take $(k, l, m)=(k_2, k_1, k_3)$. Thus we may assume that $\ord_p(k_3-k_1)=\ord_p(k_2-k_1)$ for all $k_3\in F\setminus \{k_1, k_2\}$. Set $N=\ord_p(k_2-k_1)$. Then $\ord_p(k-k_1)=N$ for all $k\in F\setminus \{k_1\}$. Consequently, $\ord_p(k-l)\ge N$ for all distinct $k, l\in F\setminus \{k_1\}$. Since $|F\setminus \{k_1\}|\ge p+1$, from the above paragraph we can find distinct $k_3, k_4\in  F\setminus \{k_1\}$ such that $\ord_p(k_4-k_3)>N$. Now we can set $(k, l, m)=(k_1, k_3, k_4)$. This proves our claim. 

Let  $F$ be a finite subset of $\Zb$ with $|F|\ge p^2+2p+3$. We claim that there are distinct $k, l, m, n\in F$ such that 
$$\ord_p(l-k)<\ord_p(m-l)<\ord_p(n-m).$$
We first find distinct $k_1, k_2\in F$ such that $\ord_p(k_2-k_1)$ is the smallest value among $\ord_p(l-k)$ for distinct $l, k\in F$. Denote by $F'$ the set of $k\in F\setminus \{k_1\}$ satisfying $\ord_p(k-k_1)=\ord_p(k_2-k_1)$. Note that $\ord_p(k-l)\ge \ord_p(k_2-k_1)$ for all distinct $k, l\in F'$. 
If $|F'|\ge p(p+1)+1$, then we can find some $F''\subset F'$ with $|F''|\ge p+2$ such that $\ord_p(k-l)\ge \ord_p(k_2-k_1)+1$ for all distinct $k, l\in F''$, and by the above paragraph we can find distinct $k_3, k_4, k_5\in F''$ such that $\ord_p(k_4-k_3)<\ord_p(k_5-k_4)$. In such case we can take $(k, l, m, n)=(k_1, k_3, k_4, k_5)$. Thus we may assume that $|F'|\le p(p+1)$. Then $|F\setminus (F'\cup \{k_1\})|\ge p+2$. By the above paragraph we can find distinct $k_3, k_4, k_5\in F\setminus (F'\cup \{k_1\})$ such that $\ord_p(k_4-k_3)<\ord_p(k_5-k_4)$. If $\ord_p(k_4-k_1)\le \ord_p(k_4-k_3)$, then we may take $(k, l, m, n)=(k_2, k_1, k_4, k_5)$. Thus we may assume that $\ord_p(k_4-k_1)>\ord_p(k_4-k_3)$. Then $\ord_p(k_4-k_3)=\ord_p(k_3-k_1)>\ord_p(k_2-k_1)=\ord_p(k_3-k_2)$. Thus we may take $(k, l, m, n)=(k_2, k_3, k_1, k_4)$. This proves our claim.   

For each $U\subset A$, we write $[U]$ for the set of $x\in X$ satisfying $x(0)\in U$. To show that $\Zb\curvearrowright X$ is null, it suffices to show that for any $j\in \Nb$, the pair $([\{j\}], [A\setminus \{j\}])$ does not have arbitrarily large finite independence sets. In fact, we shall show that each independence set of $([\{j\}], [A\setminus \{j\}])$ has cardinality at most $p^2+2p+2$. Assume that $F\subset \Zb$ is a finite independence set of $([\{j\}], [A\setminus \{j\}])$ with $|F|\ge p^2+2p+3$ instead. From the above paragraph we can find distinct $k_1, k_2, k_3, k_4\in F$ such that 
$$\ord_p(k_2-k_1)<\ord_p(k_3-k_2)<\ord_p(k_4-k_3).$$
Since $F$ is an independence set of $([\{j\}], [A\setminus \{j\}])$, we can find some $l_1, l_2\in \Zb$ such that
$$ f(l_1+k_1)\in \Nb\setminus \{j\}, f(l_1+k_2)=j, f(l_1+k_3)\in \Nb\setminus \{j\}, f(l_1+k_4)=j,$$
and 
$$f(l_2+k_1)=j, f(l_2+k_2)=j, f(l_2+k_3)\in \Nb\setminus \{j\}, f(l_2+k_4)=j.$$
From \eqref{E-infinite fiber1} we have
$$\ord_p(k_3-k_2)+1=\min(n_{l_1+k_2}, n_{l_1+k_3})=\min(n_{l_2+k_2}, n_{l_2+k_3}),$$
and 
$$ \ord_p(k_4-k_3)+1=\min(n_{l_1+k_3}, n_{l_1+k_4})=\min(n_{l_2+k_3}, n_{l_2+k_4}).$$
Since $\ord_p(k_3-k_2)<\ord_p(k_4-k_3)$, we get 
$$n_{l_1+k_2}=\min(n_{l_1+k_2}, n_{l_1+k_3})=\min(n_{l_2+k_2}, n_{l_2+k_3})=n_{l_2+k_2}.$$
Since $f(l_1+k_2)=f(l_2+k_2)$ and $\varphi_{n_{l_1+k_2}}$ is injective, we have
$$l_1+k_2+p^{n_{l_1+k_2}}\Zb=l_2+k_2+p^{n_{l_1+k_2}}\Zb.$$
Using \eqref{E-infinite fiber1} again we have
$$\ord_p(k_2-k_1)+1=\min(n_{l_1+k_1}, n_{l_1+k_2}).$$ 
Since $\ord_p(k_2-k_1)<\ord_p(k_3-k_2)$, we get 
$$ n_{l_1+k_1}<n_{l_1+k_2}.$$
Then
$$l_1+k_1+p^{n_{l_1+k_1}}\Zb=l_2+k_1+p^{n_{l_1+k_1}}\Zb.$$
Thus $f(l_1+k_1)=f(l_2+k_1)$, a contradiction. Therefore $\Zb\curvearrowright X$ is null. 

As before, $\pi_{\eq}\colon X \to X_{\eq}$ denotes the maximal equicontinuous factor map. By \cite[Theorem 6.5]{downarowicz2005survey} there is a closed subgroup $H_f$ of $Z$ such that $X_{\eq}=Z/H_f$. From the proof of \cite[Theorems 6.4 and 6.5]{downarowicz2005survey} we know that if $x\in X$, ${\bf n}\in Z$, and $\{n_i\}_{i\in \Nb}$ is a sequence in $\Zb$ such that $f(m+n_i)\to x(m)$ in $A$ for each $m\in \Zb$ and $n_i\to {\bf n}$ in $Z$, then $\pi_{\eq}(x)={\bf n}+H_f$.  

Now we require further that each $j\in \Nb$ appears in the range of $\varphi_n$ for infinitely many $n\in \Nb$. Write ${\bf b}$ for the element in $Z$ whose image under $Z\rightarrow \Zb/p^n\Zb$ is $b_n$ for each $n$. We claim that $|\pi_{\eq}^{-1}({\bf b}+H_f)|=\infty$. Let $j\in \Nb$. By assumption there is an increasing sequence $\{n_i\}_{i\in \Nb}$ in $\Nb$ such that $j$ is in the range of $\varphi_{n_i}$ for each $i$. Say, $\varphi_{n_i}(c_i+p^{n_i}\Zb)=j$ for some $c_i\in \Zb$. Then $c_i\in b_{n_i-1}+p^{n_i-1}\Zb$ for each $i$, whence $c_i\to {\bf b}$ in $Z$. We have $f(0+c_i)=f(c_i)=j$ for every $i$. Passing to a subsequence if necessary, we may assume that $f(m+c_i)$ converges in $A$ for every $m\in \Zb$. For each $m\in \Zb$, write $x_j(m)$ for the limit of $f(m+c_i)$ in $A$. Then $x_j\in X$, and $\pi_{\eq}(x_j)={\bf b}+H_f$. Note that $x_j(0)=j$. Thus $|\pi_{\eq}^{-1}({\bf b}+H_f)|=\infty$. 
\end{proof}

The following theorem addresses the realizability of sequence entropy for minimal actions. An action $\act$ is said to be \textbf{free} if $gx = x$ implies that $g$ is the identity. If $|X|$ is infinite, then every minimal action $\Z\curvearrowright X$ is free; this is not true in general for arbitrary groups. Constructing minimal free actions is generally more delicate and requires consideration of the group's geometric structure. 

For the construction of such examples, see
\cite[Theorems~1.1 and~1.2]{gomezsequence}.

\begin{theorem}
[Gómez, León-Torres and Muñoz-López \cite{gomezsequence}]
\label{jaime}
Let $k\geq2$ and $d\in\mathbb N$.
\begin{enumerate}
    \item Let $\Gamma$ be a countable group admitting $\mathbb Z$ as a
    homomorphic image. There exists a minimal $\Gamma$-subshift $X$ such
    that
    \[
    \htop^*(\act)=\log k.
    \]


    \item Let $\Gamma$ be an infinite countable group containing a
    finite-index normal subgroup isomorphic to $\mathbb Z^d$. There
    exists a free minimal $\Gamma$-subshift $X$ such that
    \[
    \log k\leq\htop^*(\act)<\infty.
    \]
\end{enumerate}
\end{theorem}

Every infinite finitely generated nilpotent group admits $\mathbb{Z}$ as a homomorphic image; see \cite[Theorems~5.2.18 and~5.2.20]{Robinson1996Course}. Other standard examples of groups with a $\mathbb{Z}$-quotient include non-trivial free groups. In contrast, there exist finitely generated simple infinite groups (i.e.\ groups whose only normal subgroups are trivial) \cite{Higman1951}. 
In particular, such a group admits no nontrivial abelian quotient; hence it has no $\mathbb{Z}$-quotient, and therefore it does not satisfy the hypotheses of \cref{jaime} (1).

\begin{question}
    Let $k\in \N$ and $\Gamma$ be an infinite finitely generated residually finite group. Is there a minimal action $\act$ such that $$\htop^{*}(\act) = \log k?$$
\end{question}

Note that the first Grigorchuk group is finitely generated, residually finite, and just infinite (i.e., every proper quotient is finite); see \cite[Proposition~1.25 and the discussion following Lemma~5.3]{BartholdiGrigorchukSunic2003BranchGroups}. Therefore, it does not satisfy the hypotheses of \cref{jaime} (1).

\begin{question}
Let $\act$ be an action. We could define the \textbf{minimal IN-rank} of $\act$ as
\[
\begin{aligned}
r_{\IN}(X,\Gamma) &= \inf\bigl\{\, k\in\N : \exists (x_1,\dots,x_k)\in \IN_k(X,\Gamma)\setminus \Delta_k^2(X)\ \text{such that}\\
&\hspace{3.6em} (x_1,\dots,x_k,y)\in \IN_{k+1}(X,\Gamma)\Rightarrow (x_1,\dots,x_k,y)\in \Delta_{k+1}^2(X)\,\bigr\}.
\end{aligned}
\]

\begin{enumerate}
\item Is there any general relationship between $r_{\IN}(X,\Gamma)$ and $r_{\Q}(X,\Gamma)$?
\item In particular, does $r_{\IN}(X,\Gamma)<\infty$ imply $r_{\Q}(X,\Gamma)<\infty$?
\end{enumerate}
\end{question}

\subsection{Tameness and IT-tuples}
\label{sectionIT}
Given an action $\act$, the \textbf{Ellis semigroup} $E(\act)$ is defined as the closure of 
$$\{x\mapsto gx:g\in \Gamma\}\subset X^X$$ 
in the product topology. Using the composition as the operation one can check that $E(\act)$ is indeed a semigroup. We
say that $\act$ is {\bf tame} if the cardinality of $E(\act)$ is at most $2^{\aleph_0}$ (the cardinality of the real numbers). 

Tameness was originally defined by Köhler \cite{kohler1995enveloping} by requiring that no continuous function $f\colon X \rightarrow \mathbb{R}$ has an
infinite $\ell_1$-isomorphism set. We use the equivalent definition as stated by Glasner in \cite{glasner2006tame}. 

In contrast to the density conditions in the context of entropy, here we are interested in the existence of infinite independence sets. This is related to Rosenthal's $\ell_1$ Theorem \cite{rosenthal1978some} and the Ramsey method involved in its proof (see \cite{gowers2003ramsey}). 

\begin{definition}\label{D-IT-pair}
Let $\act$ be an action and $k\geq 1$. 
We say that $\ox = (x_1, \dots, x_k)\in X^k$ is an {\it IT-tuple}
if for any product
neighborhood $U_1 \times\cdots\times U_k$ of $\ox$ the tuple $(U_1,
\dots, U_k)$ has an infinite independence set. We denote the set of
$IT$-tuples of length $k$ by $\IT_k(X,\Gamma)$. 
\end{definition}

Every orbit IE-tuple is an IT-tuple \cite[Corollary 7.2]{KerrLi2013}. Every IT-tuple is an IN-tuple. 

For a proof of the following result see \cite[Proposition 6.4]{KerrLi2007}. 
\begin{theorem}[Kerr and Li \cite{KerrLi2007}]
\label{thm:IT properties}  Let $\act$, $\Gamma\curvearrowright Y$ be actions and $k\in \N$. We have that 
    \begin{enumerate}
        \item $\IT_{k}(X,\Gamma)$ is a closed subset of $X^k$ that is invariant under the product action. 
        \item If $\pi\colon \act \rightarrow \Gamma\curvearrowright Y$ is a factor map, then $$(\pi\times\dots\times \pi)(\IT^{}_{k}(X,\Gamma))=\IT^{}_{k}(Y,\Gamma).$$
        
    \end{enumerate}
    
\end{theorem}

The following result was proven in \cite[Proposition 6.4]{KerrLi2007}.

\begin{theorem}
[Kerr and Li \cite{KerrLi2007}]
\label{thm:charactame}
Let $\act$ be an action. 
    Then $\act$ is tame if and only if $\IT_2(X,\Gamma)\subset \Delta_2(X)$. 
\end{theorem}

\begin{corollary}
[Kerr and Li \cite{KerrLi2007}]
 If an action $\act$ is null, then it must be tame.   
\end{corollary}
There exist minimal actions of $\Z$ that are tame but not null (see \cite[Section 11]{KerrLi2007} and \cite{fuhrmann2020tameness}).

The following result was proven for minimal actions of $\Z$ by Glasner \cite{glasner2006tame}. Then, Kerr and Li proved it for actions of $\Z$ in \cite{kerr2005dynamical} and for actions of amenable groups in \cite{KerrLi2007}. The following generalization was proved by Li and Rong \cite[Theorem 1.3]{LiRong2019}.
\begin{theorem}
[Glasner \cite{glasner2006tame}, Kerr and Li \cite{kerr2005dynamical,KerrLi2007}, Li and Rong \cite{LiRong2019}]
If an action $\act$ is tame, then it has zero naive topological entropy.
\end{theorem}

{
\begin{definition}
    Let $\Gamma\curvearrowright X$ be an action. We define the \textbf{IT-rank} of $\act$ as
\[
\begin{aligned}
R_{\IT}(X,\Gamma) &= \sup\bigl\{\, k\in\N : \IT_k(X,\Gamma)\not\subset \Delta_k^2(X) \,\bigr\}.
\end{aligned}
\]

\end{definition}

The following remark is easy to check. 
\begin{remark}
    Let $\act$ be an action. 
    We have that
    \[
      R_{\IT}(X,\Gamma)\leq R_{\IN}(X,\Gamma)\leq R_{\Q}(X,\Gamma).
    \]
\end{remark}

Almost automorphy was proven for minimal tame actions of abelian groups by Huang \cite{huang2006tame} and Kerr and Li \cite{KerrLi2007} independently. 
Structure theorems for minimal actions with bounded IT-tuples were studied bu Huang, Lian, Shao and Ye \cite{huang2021minimal}, Liu, Wang and Xu \cite{liu2025sequenceIT} and  Liu, Xu and Zhang \cite{liu2025independence}. 

Given an action $\act$, a subset $A\subset X$ is called an \textbf{IT-set} if it is non-empty and every tuple of points in $A$ is an IT-tuple.

\cref{chumin1} further extends these results, resolving Conjecture~1 in \cite{huang2021minimal}.
This result draws substantially on the ideas and arguments developed in \cite{liu2025independence}.

\begin{theorem}
    
\label{chumin1}
Let $\act$ be a minimal action with $\mathcal{M}(\act)\neq \emptyset$. There exists a dense $G_{\delta}$ subset $Z\subset X_{\eq}$, with $\nu_{\eq}(Z)=1$, such that for every $z\in Z$, $\pi_{\eq}^{-1}(z)$ is an IT-set.
\end{theorem}
\begin{proof}
    Use \cref{RP-pair to tuple} and \cite[Remark 2.7]{liu2025independence}. 
\end{proof}

The following three corollaries are straight-forward consequences of \cref{chumin1}.

\begin{corollary}
\label{q-it}
  Let $\act$ be a minimal action with $\mathcal{M}(\act)\neq \emptyset$ and $k \geq 2$ an integer. If $R_{\IT}(X,\Gamma)\leq k-1$ then $\act$ is $k$-regular.
\end{corollary}
\cref{q-it} answers affirmatively \cite[Question 1.8]{liu2025sequenceIT}.

\begin{corollary}
\label{ineqs}

  Let $\act$ be a minimal action with $\mathcal{M}(\act)\neq \emptyset$. We have that
    \[
      r_{\Q}(X,\Gamma)\leq   R_{\IT}(X,\Gamma)\leq R_{\IN}(X,\Gamma)\leq R_{\Q}(X,\Gamma)
     .
    \]
\end{corollary}
\begin{corollary}
[Glasner \cite{glasner2018structure}]\label{cor:glasner} Every minimal tame action which admits an invariant Borel probability measure, is almost automorphic.
\end{corollary}

Let $\act$ be an action. We say that $x\in X$ is a \textbf{minimal point} if the sub-action generated by $x$ is minimal, that is, if $\Gamma\curvearrowright \overline{orb(x)}$
is a minimal action. We say that $\act$ is \textbf{incontractible} if the set of minimal points of $\Gamma\curvearrowright X^n$ is dense for every $n\in \N$.

The following result was proved in \cite[Theorem 1.3]{liu2025independence}.
\begin{theorem}
[Liu, Xu and Zhang \cite{liu2025independence}]
\label{chumin+}
Let $\act$ be an incontractible minimal action. There exists a dense $G_{\delta}$ subset $Z\subset X_{\eq}$ with $\nu_{\eq}(Z)=1$ such that, for every $z\in Z$, the fiber $\pi_{\eq}^{-1}(z)$ is an IT-set. In particular, if $R_{\IT}(X,\Gamma)\leq k-1$ for some integer $k\geq 2$, then $\act$ is $k$-regular.
\end{theorem}

For the proof of \cref{incont} see \cite[Page 5, (c)]{liu2025independence}.
\begin{proposition}
    \label{incont}
Let $\act$ be a minimal action. If $r_{\Q}(X,\Gamma)<\infty$ then $\act$ is incontractible. 
\end{proposition}
\cref{thm:tameregular} was proved by Fuhrmann, Glasner, Jäger and Oertel \cite{fuhrmann2021irregular} with the hypothesis that $\act$ is almost automorphic instead of $r_{\Q}(X,\Gamma)<\infty$. It follows from \cref{incont}, \cref{chumin+} and \cref{cor:glasner}.
\begin{corollary}
[Fuhrmann, Glasner, Jäger and Oertel \cite{fuhrmann2021irregular}, Liu, Xu and Zhang \cite{liu2025independence}]
\label{thm:tameregular}
   Let $\act$ be a minimal tame action. If either of the following conditions holds:
\begin{itemize}
\item $r_{\Q}(X,\Gamma)<\infty$,
\item $\mathcal{M}(\act)\neq \emptyset$,
\end{itemize}
then $\act$ is regular.

\end{corollary}

Unique ergodicity for minimal tame actions of abelian groups was proven by Huang \cite{huang2006tame} and Kerr and Li \cite{KerrLi2007} independently.
Bounded IT-tuples imply finitely many ergodic probability measures for minimal actions. Liu, Xu, and Zhang proved the statement of \cref{chumin2} under an additional hypothesis (local Bronstein condition) \cite{liu2025independence}. Previous weaker bounds were found in \cite{huang2021minimal, liu2025sequenceIT}. 

\begin{theorem}
\label{chumin2}
Let $\act$ be a minimal action. We have that
\[
|\mathcal{M}_e(\act)|\leq R_{\IT}(X,\Gamma).
\]

\end{theorem}
\begin{proof}
   Follows from \cref{chido}, which will be proved later. 
\end{proof}
We recover the following result.
\begin{corollary}
[Glasner \cite{glasner2018structure}]
\label{thm:glasner} Every minimal tame action has at most one Borel invariant probability measure.
\end{corollary}

Minimal systems may present arbitrarily large IT-ranks for a large family of acting groups; see the proofs of \cite[Theorems 1.1 and 1.2]{gomezsequence}.

\begin{theorem}
    [Gómez, León-Torres and Muñoz-López \cite{gomezsequence}]
    Let $k,d\in \N$ with $k\geq 2$.
    \leavevmode
    \begin{enumerate}
        \item Let $\Gamma$ be a group admitting $\mathbb{Z}$ as a homomorphic image. There exists a minimal $\Gamma$-subshift $X$ with $R_{\IT}(X,\Gamma)=k$.
        \item Let $\Gamma$ be a group that contains a copy of $\mathbb{Z}^d$ as a finite-index subgroup. There exists a free minimal $\Gamma$-subshift $X$ such that $k\leq R_{\IT}(X,\Gamma)<\infty $.
    \end{enumerate}
\end{theorem}

In view of \cref{RP-pair to tuple}, it is natural to ask the following question.

\begin{question}
Let $\Gamma\curvearrowright X$ be a minimal action with $\mathcal{M}(\act)\neq\emptyset$. Let $k\geq 2$, and let $x_1,\dots,x_k\in X$. If
$$
(x_i,x_j)\in \IT_2(X,\Gamma)
\qquad\text{for every }1\leq i,j\leq k,
$$
does it follow that
$$
(x_1,\dots,x_k)\in \IT_k(X,\Gamma)?
$$
\end{question}

We also do not know whether the analogous statement holds for IE or IN tuples; see also \cite[Page~307]{GW952}.

\subsection{Nilsystems and cubic independence tuples}

Nilsystems are natural generalizations of equicontinuous systems. They provide connections with other areas of mathematics such as geometry and number theory. For a thorough exploration on these topics see \cite{host2018nilpotent}. In this subsection, we define IN-tuples of order $d$ and we will see their relation to nilsystems. 

A \textbf{topological dynamical system (TDS)} is a pair $(X,T)$, where $X$ is a compact metrizable space and $T\colon X\rightarrow X$ is a homeomorphism. 

Let $G$ be a group. For $g, h\in G$ and $A,B \subset G$, we write $[g, h] =
ghg^{-1}h^{-1}$ for the commutator of $g$ and $h$ and
$[A,B]$ for the subgroup generated by $\{[a, b] : a \in A, b\in B\}$.
The commutator subgroups $G_j$, $j\ge 1$, are defined inductively by
setting $G_1 = G$ and $G_{j+1} = [G_j, G]$. Let $d \ge 1$ be an
integer. We say that $G$ is {\bf $d$-step nilpotent} if $G_{d+1}$ is
the trivial subgroup.

Let $G$ be a $d$-step nilpotent Lie group and $H$ a discrete
cocompact subgroup of $G$. The compact manifold $X = G/H$ is
called a {\bf $d$-step nilmanifold}. The group $G$ acts on $X$ by
left translations and we write this action as $(g, x)\mapsto gx$.
For any $\tau\in G$, the transformation $x\mapsto \tau x$ of $X$ is called
 a {\bf $d$-step nilsystem}. A \textbf{nilsystem} is a $d$-step nilsystem for some $d\geq 1$. A 1-step nilsystem is simply a rotation on a compact abelian Lie group, and hence it must be equicontinuous.

 We may also define nilsystems of order $d$ using inverse limits. We say that a TDS $(X,T)$ is an \textbf{inverse limit of} the TDSs $(X_i,T_i)_{i\in \N}$ if there exist factor maps 
$\pi_i\colon X_{i+1}\rightarrow X_i$, such that 
\begin{itemize}
    \item $X$ is the subset of
$\prod_{i\in \N}X_i$ given by $\{ (x_i)_{i\in \N }: \pi_i(x_{i+1}) =
x_i \text{ }\forall i\in \N\}$, 
\item the topology on $X$ is induced by the product topology on $\prod_{i\in \N}X_i$, and
\item $T$ acts as $T_i$ on each coordinate.
\end{itemize}
A minimal system $(X,T)$ is called a \textbf{nilsystem of order $d$} if it is an inverse limit of minimal $d$-step nilsystems. A minimal system $(X,T)$ is called a \textbf{nilsystem of order $\infty$} if it is an inverse limit of minimal nilsystems. 

Let $(X,T)$ be a minimal TDS and $d\in \N \cup \{\infty\}$. The \textbf{maximal nilfactor of order $d$} is the largest factor that is a nilsystem of order $d$. The factor map to the maximal nilfactor of order $d$ will be denoted by $\pi_d\colon X\rightarrow X_d$. Using results from \cite{leibman2005pointwise} one can show the existence of maximal nilfactors. Maximal nilfactors can also be constructed using regionally proximal pairs of finite order \cite[Theorem 3.9]{shao2012regionally} and infinite order \cite[Remark 3.5]{dong2013infinite}. The maximal nilfactor of order 1 is the maximal equicontinuous factor.

\begin{definition}
Let $(X, T)$ be a TDS and $d,k \in \N$. We say that $\ox=(x_1,\ldots,x_k) \in X^k$  is an \textbf{$IN^{[d]}$-tuple} if for all $ \ell \in \N$ and product neighborhood $U_1\times\dots \times U_k$ of $\ox$, there exist vectors $\vec{p}_{j}\in \N^d$ for $j=1,\ldots,\ell$ (with $\vec{p}_{j}\neq \vec{p}_{j'}$ for every $j\neq j' \in \{1,...,\ell\}$), such that
\[ \bigcup_{j=1}^\ell \{ \vec{p}_j \cdot \vec{\epsilon} : \vec{\epsilon}\in \{0,1\}^d  \}\setminus \{0\}   \] 
{(where $\cdot$ is the usual dot product on $\mathbb{R}^d$)} is an independence set for $(U_1,\ldots,U_k)$.
We denote the set of $IN^{[d]}$-tuples of length $k$ by $\IN_k^{[d]}(X,T)$. 
\end{definition}
\begin{remark}
\label{rem:IN1}
 Note that $\IN_k^{[1]}(X,T)$ is defined similar to $\IN_k^{}(X,\mathbb{Z})$ but the independence sets are required to be in $\mathbb{N}$ instead of $\mathbb{Z}$. Nonetheless one can check that $\IN_k^{[1]}(X,T)=\IN_k^{}(X,\mathbb{Z})$.
 \end{remark}

A finite subset $F$ of $\N$ is called a {\bf finite IP-set}  if
there exists a finite subset $\{p_1,p_2,\ldots,p_m\}$ of $\N$ such
that
$$F=FS(\{p_i\}_{i=1}^m)=\{p_{i_1}+\cdots+p_{i_k}:1\leq i_1<\cdots<i_k\leq m \text{ and } k\in \N\}.$$

\begin{definition}
Let $(X, T)$ be a TDS and $k\in \N$. We say that $\ox=(x_1,\ldots,x_k) \in X^k$  is an \textbf{$IN^{[\infty]}$-tuple} if for any product neighborhood $U_1\times \cdots
\times U_k$ of $\ox$, there exist independence sets for $(U_1,...,U_k)$ that contain
arbitrarily long finite IP-sets. We denote the set of these tuples by $\IN_k^{[\infty]}(X,T)$.
\end{definition}

Every minimal nilsystem of order $d\in \N \cup \{\infty\}$ is uniquely ergodic \cite[Section 11.2, Theorem 11]{host2018nilpotent}. We denote the unique invariant Borel probability measure on $X_{d}$ by $\nu_{d}$.

The following result was recently proved in \cite{qiuyu}, building on \cite[Theorem 4.5]{dong2013infinite}.

\begin{theorem}
[Dong, Donoso, Maass, Shao and Ye \cite{dong2013infinite}, Qiu and Yu \cite{qiuyu}]
\label{thm:nil infty}
Let $(X, T)$ be a minimal TDS. If $\IN_2^{[\infty]}(X,T)\subset \Delta_2(X)$, then $(X,T)$ is uniquely ergodic and $\pi_{\infty}$ is regular, that is, $$\nu_{\infty}(\{ y\in X_{\infty}:|\pi_{\infty}^{-1}(y)|=1\} )=1.$$
\end{theorem}

\begin{theorem}
[Qiu and Yu \cite{qiuyu}]
 Let $(X,T)$ be a minimal TDS. If there is some integer
$k \geq 2$ such that $\IT^{\infty}_k(X,T)\subset \Delta^2_k(X)$, then $(X,T)$ has only finitely many ergodic probability measures and $\pi_{\infty}$ is almost finite-to-one.  
\end{theorem}

We also have the following finite order structure theorem. For the proof of the following result see \cite[Theorem 5.7]{qiu2020independence}.
\begin{theorem}
[Qiu \cite{qiu2020independence}]
\label{thm:nil}
Let $(X, T)$ be a minimal TDS and $d\in \N$. If $\IN_2^{[d]}(X,T)\subset \Delta_2(X)$, then $\pi_d$ is almost 1-to-1.
\end{theorem}

\subsection{Mean dimension pairs}

Mean dimension is a dynamical invariant introduced for actions of amenable groups by Gromov \cite{gromov1999topological} and studied by Lindenstrauss and Weiss \cite{lindenstrauss2000mean}. It is particularly useful for studying infinite entropy systems and deciding when a system can be embedded in another \cite{Lindenstrauss99,gutman2020embedding}.  Mean dimension was extended for actions of sofic groups by Li \cite{li2013soficmean}. 

We say a finite open cover $\mathcal{V}$ \textbf{refines} $\U$, $\mathcal{V} \succ \U$, if every member of $\mathcal{V}$ is contained in a member of $\U$.
\begin{definition}
Let $\U$ be a finite open cover of $X$. We denote
$$
\ord(\U)=(\max_{x\in X}\sum_{U\in \U}1_{U} (x))-1 \text{, and}
$$
$$
\D(\U)=\min_{\mathcal{V} \succ \U}\ord(\mathcal{V}).
$$
\end{definition}

\begin{definition}Let $\Gamma$ be an amenable group, $\act$ an action, $\{F_n\}_{n\in \N}$ a F\o lner sequence for $\Gamma$ and $\U$ a finite open cover of $X$. We define the \textbf{mean dimension of $\act$ with respect to $\U$} as
$$ \mdim (\act,\U)=\lim_{n\rightarrow \infty}\frac{\D( \U^{F_n})}{|F_n|}.$$ The \textbf{mean dimension of }$\act$ is the supremum of $\mdim (\act,\U)$ over all finite open covers and is denoted by $\mdim(\act)$. 
\end{definition}

Mean dimension is invariant under conjugacy and does not depend on the choice of the F\o lner sequence \cite[Theorem 9.4.1]{coornaert2015topological}. For a more thorough exposition on the theory of mean dimension for actions of amenable groups see \cite[Chapter 10]{coornaert2015topological}.

We will now define mean dimension for actions of sofic groups. 
Let $\U$ be a finite open cover and $n\in \N$. We denote by $\U^{[n]}$ the finite open cover of $X^{[n]}$ consisting of $U_1\times U_2\times \dots \times U_n$ for $U_1, \dots, U_n\in \U$.

Let $F\Subset \Gamma$, $\delta>0$, $n\in \N$ and $\sigma\colon G\rightarrow \Sym(n)$. We have that $\Map(\rho, F, \delta, \sigma)$ is a closed subset of $X^{[n]}$. Consider the restriction
$\U^{[n]}|_{\Map(\rho, F, \delta, \sigma)}$ of $\U^{[n]}$ to $\Map(\rho, F, \delta, \sigma)$.
Denote $\mathcal{D}(\U^{[n]}|_{\Map(\rho, F, \delta, \sigma)})$ by $\mathcal{D}(\U, \rho, F, \delta, \sigma)$.

Let $F\Subset \Gamma$ and
$\delta > 0$. For a finite open cover $\U$ of $X$ and $\Sigma=\{\sigma_i \colon \Gamma\rightarrow \Sym(n_i) \}_{i\in \N}$ a sofic approximation sequence, we define
\begin{align*}
\mathcal{D}_\Sigma(\U, \rho ,F, \delta ) &=
\varlimsup_{i\to\infty} \frac{\mathcal{D}(\U, \rho, F, \delta, \sigma_i)}{n_i},\\
\mathcal{D}_\Sigma(\U) &=\inf_{F\Subset \Gamma} \inf_{\delta > 0} \mathcal{D}_\Sigma(\U, \rho ,F, \delta ).
\end{align*}
If $\Map (\rho ,F,\delta ,\sigma_i )$ is empty for all sufficiently large $i$, we set
$\mathcal{D}_\Sigma(\U, \rho ,F, \delta ) = -\infty$.
We define the {\bf sofic mean dimension (with respect to $\Sigma$)} as
$$ \mdim_\Sigma (\act) = \sup_{\U} \mathcal{D}_\Sigma(\U)$$
for $\U$ ranging over  finite open covers of $X$.

\begin{remark} \label{R-mean top dim1}
The quantities
$\mathcal{D}_\Sigma(\U)$ and $\mdim_\Sigma(\act)$ do not depend on the choice of $\rho$ \cite{li2013soficmean}. In particular, $\mdim_\Sigma(\act)$ is invariant under conjugcay.
If $\Gamma$ is an amenable group then $\mdim_{\Sigma}$ does not depend on the choice of a sofic approximation sequence, and the value coincides with the mean dimension as introduced by Gromov \cite[Theorem 3.1]{li2013soficmean}.  
\end{remark}

 Let $\Gamma$ be a sofic group and $\Sigma$ a sofic approximation sequence for $\Gamma$. We say that $\act$ has \textbf{uniformly positive mean dimension (with respect to $\Sigma$)}  if
\[
\mathcal{D}_\Sigma(\U)>0
\]
for every standard open cover $\mathcal U$ of $X$. 

 Generalizing a result of Lindenstrauss and Weiss \cite[Theorem 3.6]{lindenstrauss2000mean}, Li proved that the shift action $\Gamma\curvearrowright I^{G}$ of a sofic group $\Gamma$ (where $I$ is the unit interval) has completely positive sofic mean dimension \cite[Theorem 7.5]{li2013soficmean}. 

An action $\act$ has \textbf{completely positive sofic mean dimension (with respect to $\Sigma$)} if every non-trivial factor has positive sofic mean dimension. 

 García-Ramos and Gutman produced the first explicit examples of actions with completely positive mean dimension that do not have uniform positive mean dimension. More recently, Barak and Lindenstrauss \cite{BarakLindenstrauss26} constructed more dynamically natural examples of systems with such properties. More broadly, they developed a useful machinery relating symbolic dynamics to mean dimension: a geometric thickening of a symbolic system, called the hub-and-spoke construction, that converts symbolic independence into mean-dimensional complexity. 

\begin{question}
Given $\Gamma$ amenable, does there exist a non-trivial minimal action $
\Gamma \curvearrowright X$ with uniform positive mean dimension?
\end{question}

These properties can also be studied locally.
 
\begin{definition}
\label{def:mdpairs}
  Let $\Gamma$ be a sofic group, $\Sigma$ a sofic approximation sequence for $\Gamma$ and $\act$ an action. We say that $(x,y)\in X^2\setminus \Delta_2(X)$ is a \textbf{$\Sigma$-mean dimension pair} if for any closed disjoint neighborhoods, $K_x$ and $K_y$, of $x$ and $y$ respectively, we have that $\D_{\Sigma}(\{K_x^c,K_y^c\})>0$. We denote the set of $\Sigma$-mean dimension pairs by $D^{\Sigma}_m(X,\Gamma)$.
\end{definition}
The following result was proved in \cite[Corollary 4.8]{garcia2021mean}.
  \begin{theorem}
  [García-Ramos and Gutman \cite{garciameandimesnion}]
\label{cor:zmd}
 Let $\Gamma$ be a sofic group, $\Sigma$ a sofic approximation sequence for $\Gamma$ and $\act$ an action. Then $\mdim_\Sigma (\act)>0$ if and only if $D^{\Sigma}_m(X,\Gamma)\neq \emptyset.$
\end{theorem}

 An action with zero mean dimension may still have a factor with positive mean dimension. For the following result, see \cite[Theorem 4.16]{garciameandimesnion}.

\begin{theorem}
[García-Ramos and Gutman \cite{garciameandimesnion}]
 Let $\Sigma$ be a sofic approximation sequence for $\Gamma$ with and $\act$ an action with $\mdim_\Sigma (\act)>0$. If the smallest closed $\Gamma$-invariant equivalence relation containing $\Sigma$-mean dimension pairs is $X^2$, then $\act$ has completely positive sofic mean dimension.   
\end{theorem}

\begin{question}
Let $\Sigma$ be a sofic approximation sequence for $\Gamma$, and $\act$ an action. Is every $\Sigma$-mean dimension pair a $\Sigma$-IE-pair?
\end{question}
We do not know the answer to the previous question even when $G=\Z$.

Considering the characterization of positive entropy using combinatorial independence, it is natural to ask the following question.

\begin{question}
Can positive mean dimension be characterized with some form of combinatorial independence? 
\end{question}

Kerr and Li introduced the notion of SBP pairs (small boundary property pairs) and proved that, for actions of amenable groups, every mean dimension pair is an SBP pair. They asked whether the converse holds for free minimal actions of $\mathbb Z^d$, that is, whether every SBP pair must be a mean dimension pair \cite{KerrLiSBPProducts}.

In \cite{garciameandimesnion} the descriptive complexity of uniform and completely postive mean dimension was also studied. 

\section{Applications and connections} 
\subsection{Homoclinic points of algebraic actions}

In this subsection we are interested in how homoclinic points relate to entropy. This line of research began with the work of Lind and Schmidt \cite{LindSchmidt1999} on algebraic actions of $\Z^d$.  

We say an action $\act$ is \textbf{algebraic} if $X$ is a compact metrizable abelian group and $\Gamma$ acts by automorphisms of $X$.

For algebraic actions it is more convenient to work with points rather than pairs or tuples because we have a reference fixed point, $e_X$, the identity of $X$. 

Let $\act$ be an algebraic action. We say $x\in X$ is an \textbf{IE-point} 
if $(x,e_X)\in \IE_2(X,\Gamma)$.
The set of IE-points is a $\Gamma$-invariant closed subgroup of $X$ \cite[Theorem 6.4]{KerrLi2013}, called the \textbf{IE-group} and will be denoted  by $\IE(X,\Gamma)$.
Furthermore, for each $k\in \NN$, we have \cite[Theorem 6.4]{KerrLi2013}
\begin{align*}
\IE_k(X,\Gamma)&=\{(yx_1, \dots, yx_k): x_1, \dots, x_k\in \IE(X,\Gamma), y\in X\}.
\end{align*}

If $\Gamma$ is sofic and $\Sigma$ is a sofic approximation sequence for $\Gamma$, we say $x\in X$ is an \textbf{IE$^{\Sigma}$-point} 
if $(x,e_X)\in \IE^{\Sigma} _2(X,\Gamma)$.
The set of IE$^{\Sigma}$-points is a $\Gamma$-invariant closed subset of $X$, and will be denoted by $\IE^{\Sigma}(X,\Gamma)$.

Given an algebraic action $\Gamma\curvearrowright X$, we say $x\in X$ is a \textbf{homoclinic point} \cite{LindSchmidt1999} if $sx\to e_X$ as $G\ni s\to \infty$. The homoclinic points form a $\Gamma$-invariant subgroup of $X$,  called the \textbf{homoclinic group}, and will be denoted  by $\A(X,\Gamma)$.

	We say $\act$ is \textbf{expansive} if there exists $c >0$ such that whenever $x,y \in X$, if $x \neq y$ then there exists $g \in \Gamma$ such that $\rho(gx,gy)> c$. 

The following result was proved for actions of amenable groups in \cite{ChungLi2015}. The following generalization appeared in \cite[Corollary 7.2]{barbieri2022markovian}.

\begin{theorem}
[Chung and Li \cite{ChungLi2015}, Barbieri, García-Ramos and Li \cite{barbieri2022markovian}]
\label{P_homoclinic_points}

	Let $\Gamma\curvearrowright X$ be an expansive algebraic action. We have:
	\begin{enumerate}
		\item $\A(X,\Gamma)\subset \IE(X,\Gamma)$.
		\item If $\Gamma$ is sofic and $\Sigma$ is any sofic approximation sequence for $\Gamma$, then $A(X,\Gamma)\subset \IE^{\Sigma}(X,\Gamma)$ .
	\end{enumerate}
\end{theorem}

The family of \textbf{elementary amenable groups} is the smallest class of groups containing all finite groups and all abelian groups, and is closed under taking subgroups, quotient groups, extensions, and inductive limits~\cite{Day1957,chou1980elementary}. Elementary amenable groups are amenable, but there exist amenable groups that are not elementary amenable \cite{Grigorchuk1986}. 

For each locally compact abelian group $Y$, we denote by $\widehat{Y}$ its \textbf{Pontryagin dual}. It consists of all continuous group homomorphisms $Y\rightarrow \R/\ZZ$, and becomes a locally compact abelian group under pointwise addition and the topology of uniform convergence on compact subsets.
Then $Y$ is compact metrizable if and only if $\widehat{Y}$ is countable and discrete.

For a compact metrizable abelian group $X$, there is a natural one-to-one correspondence between algebraic actions of $\Gamma$ on $X$ and actions of $\Gamma$ on $\widehat{X}$ by automorphisms. There is also a natural one-to-one correspondence between the latter and the left $\ZZ \Gamma$-module structure on $\widehat{X}$ ($\ZZ \Gamma$ is the group ring which consists of all functions $f\colon \Gamma\rightarrow \ZZ$ of finite support; for more information, see \cite{Passman1985}). Thus, up to isomorphism, there is a natural one-to-one correspondence between algebraic actions of $\Gamma$ and countable left $\ZZ \Gamma$-modules.

We say an algebraic action $\Gamma\curvearrowright X$ is \textbf{finitely presented} if $\widehat{X}$ is a finitely presented left $\ZZ \Gamma$-module. 
Expansive algebraic actions are always finitely generated ($\widehat{X}$  is finitely generated) \cite[Proposition 2.2 and Corollary 2.16]{Schmidt1995}.

\begin{theorem}
[Chung and Li \cite{ChungLi2015}, Barbieri, García-Ramos and Li \cite{barbieri2022markovian}]
\label{thm:IE group}
	Let $\Gamma$ be an elementary amenable group with an upper bound on the orders of finite subgroups of $\Gamma$, and $\Gamma\curvearrowright X$ a finitely presented expansive algebraic action.
	Then $\overline{\A(X,\Gamma)} = \IE(X,\Gamma)$.
\end{theorem}	

The previous theorem was proved for polycyclic-by-finite groups by Chung and Li \cite[Theorem 1.2]{ChungLi2015}, and generalized in \cite[Corollary 7.5]{barbieri2022markovian}. The statement of the result holds more generally, for amenable groups such that the strong Atiyah conjecture is satisfied, and there is an upper bound on the orders of finite subgroups (see \cite{barbieri2022markovian} for definitions and details).

\begin{question}
Does the statement of \cref{thm:IE group} hold if we relax the hypothesis of elementary amenability for $\Gamma$ and simply ask for amenability?

\end{question}
The statement of \cref{thm:IE group} is not satisfied for actions of the free group  $\mathbb{F}_2$ in the context of naive topological entropy \cite[Example 7.6]{barbieri2022markovian}. 

\begin{question}
Is there a finitely presented expansive algebraic action of $\mathbb{F}_2$ with positive topological sofic entropy for some sofic approximation sequence and trivial homoclinic group?
\end{question}
As a corollary of \cref{thm:IE group} we obtain the following result. 
\begin{corollary}
    Let $\Gamma$ be an elementary amenable group  with an upper bound on the orders of finite subgroups of $\Gamma$, and $\Gamma\curvearrowright X$ an expansive finitely presented algebraic action. Then, 
    \begin{enumerate}
        \item $\act$ has positive topological entropy if and only if ${\A(X,\Gamma)}\neq \{e_X\}$, and 
        \item $\act$ has UPE if and only if $\overline{\A(X,\Gamma)}=X.$
    \end{enumerate}
\end{corollary}

\begin{remark}
    Meyerovitch proved that there exist an abelian group $\Gamma$ and an expansive algebraic action $\act$ with positive topological entropy and trivial homoclinic group \cite[Theorem 1.2]{Meyerovitch2017}.
\end{remark}

An algebraic action can be seen as a measure-preserving action with respect to the Haar measure of $X$. For algebraic actions of amenable groups, the Haar measure is a measure of maximal entropy, thus one has positive topological entropy if and only if one has positive measure-theoretical entropy \cite[Theorem 2.2]{deninger2006fuglede}. Furthermore, an algebraic action of an amenable group has UPE if and only if it has completely positive entropy in the measure-theoretic sense (every non-trivial measurable factor has positive measure-theoretic entropy) \cite{ChungLi2015}.

\subsection{Continuum theory}

 In this subsection we will talk about a couple of interactions between topological dynamics and continuum theory. A topological space $X$ is a \textbf{continuum} if it is a compact connected metrizable space. For an introduction to continuum theory see \cite{nadler1992continuum}.

 In the first topic that we will expose, the underlying theme is that
complicated dynamics should imply existence of complicated continua.  This line of investigation began with the work of Barge and Martin \cite{barge1985chaos}, one of their results implies that if $f$ is a continuous self-map of the interval $I = [0,1]$ with positive topological entropy, then the inverse limit space $\varprojlim(I,f)$ (the inverse limit by infinite copies of the same system) contains an indecomposable continuum. Indecomposable continua were introduced by Brouwer. A continuum is \textbf{indecomposable} if it contains more than one point and it cannot be expressed as the union of any two of its proper subcontinua. The dyadic solenoid and the pseudo-arc are indecomposable continua; the cube and the torus are not indecomposable.

Let $Y$ be a compact metrizable space and $\varepsilon>0$. A continuous mapping $g\colon X\rightarrow Y$ is an \textbf{$\varepsilon$-mapping}
if for every $y \in Y$, the diameter of $g^{-1}(y)$ is less than $\varepsilon$. A continuum $X$ is \textbf{$Y$-like} if for every $\varepsilon>0$ there is an $\varepsilon$-mapping from $X$ onto $Y$. 

Using local entropy theory and introducing a new dynamical pair called zigzag pair the following result was proved in \cite[Corollary 5.4 and Corollary 5.6]{darji2017chaos}. 
\begin{theorem}
[Darji and Kato \cite{darji2017chaos}] Let $Y$ be a finite graph treated as a topological space, and $X$ a $Y$-like continuum. If $X$ admits a homeomorphism $T$ such that $(X,T)$ has positive topological entropy, then $X$ contains an indecomposable continuum. Furthermore, if $X$ admits a homeomorphism $T$ such that $(X,T)$ has UPE then $X$ is an indecomposable continuum. 
 
\end{theorem}

Now we will present a result where continuum theory and local entropy theory are applied to obtain a result regarding the possible values of sequence entropy. 

 Let $S_{\Gamma}(X)$ be the set of the values
$\htop ^{*}(\act)$ for all possible actions $\act$. \cref{thm:seqentr} implies that $\{0\}\subset S_{\Gamma}(X) \subset \{0,\log 2, \log 3,...\}\cup \{\infty\}$. The following result appeared in \cite[Theorem 8.11]{snoha2020topology}. 

\begin{theorem}
[Snoha, Ye and Zhang \cite{snoha2020topology}]

\label{thm:snoha}
Let $\Gamma$ be a finitely generated group such that there is a surjective homomorphism $\Gamma \rightarrow \ZZ$, and $\{0\}\subset A \subset \{0,\log 2, \log 3,...\}\cup \{\infty\}$. There exists a one-dimensional
continuum $X$ such that $S_{\Gamma}(X)=A$.
\end{theorem}

\begin{question}
For which groups does the conclusion of \cref{thm:snoha} hold?
\end{question}

\begin{question}
Is it true that for every set $\{0\} \subset A \subset [0, \infty]$ such that $\N a\subset A$ for every $a\in A$, there exists a compact metrizable space $X$ such that the possible values of topological entropy for all actions of $\ZZ$ over $X$ is $A$?
\end{question}

\subsection{Descriptive complexity}

Given a compact metrizable space, one may study the properties of the set of dynamical systems (or actions) with a certain property on that. One way to approach this is with the viewpoint of Baire category theory and descriptive set theory, which can give us information of how complicated a set is. This approach can allow us to discard possible characterizations of the studied properties. For example, Halmos and Rohlin proved that the properties of the sets of measure-preserving mixing and weak-mixing systems are different and thus concluded the existence of weak-mixing systems that are not mixing without explicitly constructing one. 

While the existence of actions with CPE and not with UPE is known, in this subsection we will see a striking difference between these families.

This line of research began with the work of Pavlov when he found a sufficient condition for completely positive entropy on subshifts of finite type of $\Z^d$, and asked if the condition was also necessary \cite{pavlov2018topologically}. This question was answered in the negative by Barbieri and García-Ramos by defining a rank (or hierarchy) for actions with completely positive entropy \cite{barbieri2021hierarchy}. Westrick proved that this rank is actually a $\Pi_1^1$-rank \cite{westrick2019topological} (see also \cite{darji2021note}), which implies that this rank can be used to prove that certain sets are complicated (not Borel) in the sense of descriptive set theory.  

We will give a quick overview of classical descriptive complexity. For a more thorough exposition of these topics see \cite{Kechris1995, moschovakis2009descriptive}. 

Given a compact metrizable space $X$, we define $\text{TDS(X)}$ as the set of all continuous functions from $X$ into $X$ endowed with the  uniform convergence topology. Note that $\text{TDS(X)}$ is a \textbf{Polish space}, i.e., a separable space whose topology can be induced by a complete metric. 
In previous sections we defined UPE, CPE and expansiveness for group actions. The definition for semigroup actions and in particular continuous maps $T\colon X\rightarrow X$ is the natural adaptation. 
Let $T\colon X\rightarrow X$ be a continuous function. We say that $T$ is \textbf{mixing} if for every pair of non-empty open sets $U$ and $V$ there exists $N\in \NN$ such that $T^{-n}U\cap V \neq \emptyset$ for all $n\geq N$.  

We also define the following subspaces
\begin{gather*}
\textup{UPE}(X)=\{T\in \textup{TDS(X)}: T\textup{ has UPE}\},
\\
\textup{CPE}(X)=\{T\in \textup{TDS(X)}: T\textup{ has CPE}\},
\\
\textup{Mix}(X)=\{T\in \textup{TDS(X)}: T\textup{ is mixing}\},\textup{ and}
\\
\textup{Exp}(X)=\{T\in \textup{TDS(X)}: T\textup{ is expansive}\}.
\end{gather*}

We say that a subset of a Polish space is \textbf{analytic} if it is the continuous image of a Polish space and \textbf{coanalytic} (or $\Pi_1^1$) if it is the complement of an analytic set.
All Borel subsets of a Polish space are both analytic and coanalytic. Moreover, if a set is both analytic and coanalytic, then it must be Borel. However, in every uncountable Polish space, there are analytic, and hence coanalytic, sets that are not Borel. Loosely speaking, if a set is analytic or coanalytic but not Borel, it means that the set cannot be described with countable information.

A standard way to prove that a coanalytic set is not Borel is to reduce it to a known combinatorial set which is not Borel. More specifically, if $\mathcal{B}$ is a known non-Borel subset of some Polish space $\mathcal{Y}$, $\mathcal{A} \subset  \mathcal{X}$ and $f\colon \mathcal{Y}\rightarrow \mathcal{X}$ is a Borel function such that $f^{-1}( \mathcal{A}) = \mathcal{B}$, then $\mathcal{A}$ is not Borel. In this case, we say that $\mathcal{B}$ is \textbf{Borel reducible} to $\mathcal{A}$. This inspires the following definition. 
\begin{definition}
A coanalytic subset $\mathcal{A}$ of a Polish space $\mathcal{X}$ is \textbf{complete coanalytic} (or $\Pi_1^1$-complete) if for every coanalytic subset $\mathcal{B}$ of a Polish space $\mathcal{Y}$ there exists a Borel function $f\colon \mathcal{Y}\rightarrow \mathcal{X}$ such that $f^{-1}( \mathcal{A}) = \mathcal{B}$. 
\end{definition}
Complete coanalytic sets are the most complicated coanalytic sets. 

For a proof of the following result see \cite[Corollary 3.2 and Proposition 2.12]{darji2021local}.
\begin{theorem}
[Darji and García-Ramos \cite{darji2021local}]
Let $X$ be a compact metrizable space. Then $\textup{UPE}(X)$ is a Borel subset of $\textup{TDS}(X)$, and $\textup{CPE}(X)$ is a conalytic subset of $\textup{TDS}(X)$.
\end{theorem}

Whether $\textup{CPE}(X)$ is Borel or not will depend on $X$. There exist infinite compact metrizable spaces that do not admit dynamics with positive topological entropy, thus for these spaces $\textup{CPE}(X)$ is empty (and hence Borel) \cite{cook1967continua}. 

There are two techniques to prove that $\textup{CPE}(X)$ or subfamilies, are not Borel. One can use Borel reductions or the entropy rank introduced by Barbieri and García-Ramos.  

The following result is a combination of \cite[Theorem 3.7]{darji2021note}, \cite[Theorem 35.23]{Kechris1995}, and the proof of \cite[Theorem 3.3]{barbieri2021hierarchy},

\begin{theorem}
[Barbieri and García-Ramos \cite{barbieri2021hierarchy}] Let $X$ be a Cantor space. We have that $\textup{CPE}(X)$ is a coanalytic but not Borel subset of $\textup{TDS}(X)$.
\end{theorem}

For the proof of the following result see \cite[Theorem 4.1]{darji2021local}.
\begin{theorem}
[Darji and García-Ramos \cite{darji2021local}]
\label{thm:d-g}
For every $d\geq 1$ we have that $\textup{CPE}(I^d)$ is a complete coanalytic (thus not Borel) subset of $\textup{TDS}(I^d)$.
\end{theorem}

The previous result also holds for the circle, the torus and some other compact orientable manifolds. 

\begin{remark}
    If one restricts to the space of Lipschitz functions in $\textup{TDS}(I)$ one may still obtain a result similar to \cref{thm:d-g}. We do not know what happens if we restrict to smooth functions. 
\end{remark}

One may wonder what happens if we consider dynamical systems with natural dynamical properties.

The following result is a combination of \cite[Theorem 3]{salo2019entropy}, \cite[Theorem 35.23]{Kechris1995} and \cite[Theorem 3.7]{darji2021note}. 

\begin{theorem}
[Salo \cite{salo2019entropy}] 
\label{thm:salo}
Let $X$ be a Cantor space. We have that $\textup{Exp}(X) \cap \textup{CPE}(X)$ is a coanalytic but not Borel subset of $\textup{TDS}(X)$.
\end{theorem}

\begin{remark}
    For any compact metrizable space $X$, the family of maps with shadowing and CPE is a Borel subset of $\textup{TDS}(X)$ \cite[Proposition 3.5 and Corollary 3.4]{darji2021local}.
\end{remark}

The following result can be obtained with techniques used by Blokh \cite{blokh1984transitive}  (see \cite[Corollary 3.9]{darji2021local}). 
\begin{proposition}
    Let $X$ be a finite graph treated as a topological space. Then $\textup{Mix}(X)\cap \textup{CPE}(X)$ is a Borel subset of $\textup{TDS}(X)$. 
\end{proposition}

In particular the previous result applies to the interval and the circle.

 In contrast to this, the following result was proved in \cite[Theorem 5.25]{darji2021local}. 

\begin{theorem}
[Darji and García-Ramos \cite{darji2021local}]
Let $X$ be a Cantor space. Then $\textup{Mix}(X) \cap \textup{CPE}(X)$ is a coanalytic but not Borel subset of $\textup{TDS}(X)$. 
\end{theorem}

 Proving that $\textup{CPE}(X)$ and $\textup{Exp}(X) \cap \textup{CPE}(X)$ are complete coanalytic for a Cantor space $X$ might be accessible, but some extra arguments are missing. The following question seems harder to approach.
 
 \begin{question}
     Let $X$ be a Cantor space. Is $\textup{Mix}(X) \cap \textup{CPE}(X)$ a complete coanalytic subset of $\textup{TDS}(X)$?
 \end{question}

 \begin{question}
     Let $X$ be a torus of dimension at least $2$. Is $\textup{Mix}(X) \cap \textup{CPE}(X)$ a Borel subset of $\textup{TDS}(X)$?
 \end{question}

A subshift is of finite type if it can be determined using a finite number of forbidden patterns. It is not difficult to see that for every group $\Gamma$ and every finite alphabet $A$ the family of $\Gamma$-subshifts of finite type in $A^{G}$ is countable and hence any subset of this family will be a Borel ($F_{\sigma}$) set. Nonetheless, the complexity of these families can be studied using another branch of descriptive set theory that uses computable functions instead of Borel functions. We will not give the precise definitions of effective Borel and coanalytic sets, since they are more involved than their topological counterparts. The interested reader can check \cite{moschovakis2009descriptive}.
For the following theorem see \cite[Theorem 1]{westrick2019topological}.
\begin{theorem}
[Westrick \cite{westrick2019topological}]
The property of completely positive entropy is effective complete coanalytic in the set of ${\ZZ^2}$-subshifts of finite type in $A^{\ZZ^{2}}$ for a finite set $A$.   
\end{theorem}

Completely positive mean dimension can also be studied from this point of view. In this situation a different model is more natural.  

 We now recall the basics of the Hausdorff metric. Given a metric space $(Y,\rho)$ we denote with $K(Y)$ the space of all non-empty compact subsets of $Y$. We equip $K(Y)$ with the Vietoris topology, which is generated by the \textbf{Hausdorff metric} defined as
$$
\rho_H(A,B)=\inf \{\delta>0: A\subset B_{\delta}\text{ and } B\subset A_{\delta}\},
$$
where $C_{\delta}=\{x\in Y:\rho(x,C)<\delta\}$.

If $Y$ is compact, then so is $K(Y)$. 

Let $A$ be a compact metric space. We equip $A^{\Z}$ with the product topology, and we denote the shift action by $\sigma\colon A^{\Z}\rightarrow A^{\Z}$. We define 
$$
\S(A)=\{X\subset A^{\Z}:X\text{ is closed, non-empty and } \sigma(X)= X\}.
$$ Since $\S(A)$ is a closed subset of $K (A^{\Z})$, we also equip it with the induced topology. Note that $\S(A)$ is also a compact metrizable space, and hence it is Polish.  We also define the following subspace
\[
\S_{c+}(A)=
\{
X\in \S(A): (X,\sigma)\textup{ has completely positive mean dimension}
\}.
\]

\begin{theorem}
[García-Ramos and Gutman \cite{garciameandimesnion}]
    We have that $\S_{c+}(I)$ is a complete coanalytic subset of $\S(I)$.
\end{theorem}

\subsection{Li-Yorke chaos}

In the seminal work of Li and Yorke it was proved that the existence of a periodic point of order 3 of an interval map implies some form of disorder \cite{liperiod}. This notion is now known as Li-Yorke chaos. An active line of research has been to understand under which conditions positive topological entropy implies Li-Yorke chaos. 

Let $S$ be an infinite subset of $\Gamma$.
A pair $(x,y)\in X\times X$ is called \textbf{$S$-scrambled}
if
\[\limsup_{S\ni s\to\infty}\rho(sx,sy)>0 \quad\text{ and }\quad
\liminf_{S\ni s\to\infty}\rho(sx,sy)=0.\]
A subset $K\subset X$ with at least two points is called \textbf{$S$-scrambled} if
any two distinct points $x,y\in K$ form an $S$-scrambled pair.
\begin{definition}
  Let $S$ be an infinite subset of $\Gamma$. We say that $\act$ is \textbf{Li-Yorke chaotic along $S$}
if there exists an uncountable $S$-scrambled subset of $X$.  
We say that $\act$ is \textbf{Li-Yorke chaotic} if it is {Li-Yorke chaotic along $\Gamma$}.
\end{definition}

Blanchard, Glasner, Kolyada, and Maass \cite{BlanchardGlasnerKolyadaMaass2002} proved, using local entropy theory and measure-theoretic methods, that any TDS with positive topological entropy must be Li-Yorke chaotic. This result was then proved 
for actions of amenable groups by Kerr and Li \cite{KerrLi2007} using local entropy theory and purely topological arguments. The following generalization appears in \cite[Theorem 1.1]{huang2021positive}.

\begin{theorem} 
[Huang, Li and Ye \cite{huang2021positive}]
\label{thm:liyorkesequence}
Let $\Gamma$ be an amenable group and $\act$ an action. If $\act$ has positive topological entropy,
then for any infinite subset $S\subset \Gamma$, $\act$ is Li-Yorke chaotic along $S$. 
\end{theorem}

The following result was obtained using combinatorial independence in the sense of \cref{thm:naiveentropy} in \cite[Theorem 1.1]{LiRong2019}. 

\begin{theorem}
[Blanchard et al \cite{BlanchardGlasnerKolyadaMaass2002}, Kerr and Li \cite{KerrLi2007,KerrLi2013}, Li and Rong \cite{LiRong2019}]
\label{thm:naiveLiYorke}
    Every action $\act$ with positive naive topological entropy is Li-Yorke chaotic. 
\end{theorem}

An action $\act$ is \textbf{distal} if, for every pair of distinct points $x,y\in X$, we have $\inf _{g\in \Gamma}\rho(gx,gy)>0$. Every nilsystem (See Section 3.3) is distal. 
The following result is a direct corollary of \cref{thm:naiveLiYorke} and answers a question of Bowen \cite[Question 8]{bowen2020examples}. 

\begin{corollary}
[Li and Rong \cite{LiRong2019}]
    Every distal action $\act$ has  zero topological naive entropy. 
\end{corollary}

Liu recently proved that every action with positive sofic topological entropy has Li-Yorke chaos along any infinite sequence \cite{Liuchaos}. The following question remains open. 

\begin{question}
 Is every action $\act$ with positive naive topological entropy, Li-Yorke chaotic along every infinite $S\subset \Gamma$?
\end{question}

\subsection{Induced actions}
\label{induced}
Bauer and Sigmund initiated the study of induced dynamical systems on the space of Borel probability measures \cite{bauer1975topological}. In some sense, this is like studying dynamics with random initial conditions. Local entropy theory has been very useful for proving several results that compare the dynamics of an action with the induced dynamics.   

	Let $X$ be a compact metrizable space, $\mathcal{B}_X$ the collection of Borel subsets of $X$, and $C(X)$ the space of continuous maps from $X$ to $\mathbb{R}$ endowed with the uniform convergence topology generated by the sup metric, and $\mathcal{M}(X)$ the set of Borel 
	probability measures on $X$ endowed with the \textbf{weak$^\ast$-topology}, which is the smallest topology making the map
	$\mu\mapsto\int_Xgd\mu$
	continuous for every $g\in C(X)$. The space $\mathcal{M}(X)$ is compact and metrizable. Furthermore, given an action $\act$, we naturally induce an action $\Gamma\curvearrowright \mathcal{M}(X)$ with the push-forwards.  { We view $X$ as a closed $\Gamma$-invariant subset of $\mathcal{M}(X)$ by identifying
each point $x \in X$ with the corresponding Dirac measure $\delta_x$.
With this identification, for any $k\in\N$ we have
$\mathrm{IE}_k(X) \subset X^k \subset \mathcal{M}(X)^k$ and
$\mathrm{IE}_k(\mathcal{M}(X)) \subset \mathcal{M}(X)^k$.
Given any $\mu \in \mathcal{M}(X^k)$ and any index $1 \le j \le k$, we write $\mu^{(j)}$
for the image measure of $\mu$ in $\mathcal{M}(X)$ induced by the projection
onto the $j$-th coordinate.

The following result completely describes the relation between $\IE_k(\mathcal{M}(X),\Gamma)$ and $\IE_k(X,\Gamma)$ \cite[Theorem 1.1]{liliu}.

 \begin{theorem}
 [Li and Liu \cite{liliu}]
 \label{thm:inducedIE-tuples}
Let $\Gamma$ be amenable and $\act$ an action. For each $k\in\mathbb{N}$, the following hold:
\begin{enumerate}
    \item $\IE_k(\mathcal{M}(X),\Gamma)$ is the closed convex hull of $\IE_k(X,\Gamma)$ in $\mathcal{M}(X)^k$. 
    
    \item $\IE_k(X,\Gamma) = \IE_k(\mathcal{M}(X),\Gamma) \cap X^k$ is the set of extreme points of $\IE_k(\mathcal{M}(X),\Gamma)$. 
    
    \item $\IE_k(\mathcal{M}(X),\Gamma) = \{(\mu^{(1)},\dots,\mu^{(k)}) : \mu \in \mathcal{M}(\IE_k(X,\Gamma))\}$. 
\end{enumerate}
\end{theorem}

The following result was proved in \cite[Theorem A]{glasner1995quasi}.
It also follows from the previous result using $k=2$ and \cref{prop:IE pair}.

 \begin{theorem}
 [Glasner and Weiss \cite{glasner1995quasi}]
\label{thm:induced}
 Let $\Gamma$ be amenable and $\act$ an action.  Then $\htop(\act)>0$ if and only if $\htop(\Gamma \curvearrowright\mathcal{M}(X))>0$. 
\end{theorem}

The following result was proved for $k=2$ and $\Gamma=\Z$ in \cite{bernardes2022uniformly}, and for amenable groups and $k=2$ in \cite{liu2022relative}. For a proof of the result, see \cite[Corollary 3.31]{liliu}.

 \begin{theorem}
 [Bernardes Jr, Darji and Vermersch \cite{bernardes2022uniformly}, Liu and Wei \cite{liu2022relative}, Li and Liu \cite{liliu}]
\label{thm:UPEinduced}
 Let $\Gamma$ be amenable, $\act$ an action and $k\in \N$. Then $\act$ has UPE of order $k$ if and only if $\Gamma \curvearrowright\mathcal{M}$ has UPE of order $k$. 
\end{theorem}

Topological entropy is connected to the mean dimension of the induced action. The following result was proved for $G=\Z$ in \cite[Main Theorem]{burguet2022topological} and for amenable groups in \cite[Main Theorem]{shi2023mean}.
\begin{theorem}
 [Burguet and Shi \cite{burguet2022topological}, Shi and Zhang \cite{shi2023mean}]
 Let $\Gamma$ be an amenable group and $\act$ an action. Then $\act$ has positive topological entropy if and only if $\Gamma \curvearrowright \mathcal{M}(X)$ has positive mean dimension if and only if $\Gamma \curvearrowright \mathcal{M}(X)$  has infinite mean dimension.     
 \end{theorem}

 Topological sequence entropy was first studied for induced actions of $\ZZ$ in \cite[Theorem 5.10]{kerr2005dynamical}. The following result describes the relation between $\IN_k(\mathcal{M}(X),\Gamma)$ and $\IN_k(X,\Gamma)$ \cite[Theorem 1.4]{liliu}. We first need the following definition.

\begin{definition}\label{def:IN-measure}
Let $\act$ be an action and let $k\in \N$. We say that $\mu\in \mathcal{M}(X^k)$ is an \textbf{IN-measure} if for any $n\in \N$ and any $(x_1^{(i)},\dots,x_k^{(i)})\in \supp(\mu)$ for $i\in \{1,\dots,n\}$, one has
$$
\big((x_1^{(1)},\dots,x_1^{(n)}),(x_2^{(1)},\dots,x_2^{(n)}),\dots,(x_k^{(1)},\dots,x_k^{(n)})\big)\in \IN_k(X^n,\Gamma),
$$
where $\Gamma$ acts diagonally on $X^n$.
We denote by $\mathcal{M}_{\mathrm{IN}}(X^k)$ the set of all IN-measures in $\mathcal{M}(X^k)$. 
\end{definition}

 \begin{theorem}
 [Li and Liu \cite{liliu}]
 \label{thm:inducedIN-tuples}
Let $\act$ be an action. For each $k\in\mathbb{N}$, the following hold:
\begin{enumerate}
    \item $\IN_k(\mathcal{M}(X),\Gamma)$ is contained in the closed convex hull of $\IN_k(X,\Gamma)$ in $\mathcal{M}(X)^k$. 
    
    \item $\IN_k(X,\Gamma) = \IN_k(\mathcal{M}(X),\Gamma) \cap X^k$.
    
    \item $\IN_k(\mathcal{M}(X),\Gamma) = \{(\mu^{(1)},\dots,\mu^{(k)}) : \mu \in \mathcal{M}_{\mathrm{IN}}(X^k)\}$. 
\end{enumerate}
\end{theorem}

The following result appeared in \cite[Theorem 6.5]{huang2023multi} generalizing \cite[Remark 2]{qiao2017zero} and \cite[Proposition 5.8]{KerrLi2007}. 

\begin{theorem}[Kerr and Li \cite{KerrLi2007}, Qiao and Zhou \cite{qiao2017zero}, Huang and Zhu \cite{huang2023multi}]\label{thm:seqinduced}Let $\act$ be an action and $S$ a sequence in $\Gamma$. Then $\htop^S(\act)>0$ if and only if $\htop^S(\Gamma \curvearrowright\mathcal{M}(X))=\infty$.     \end{theorem}

\begin{definition}\label{def:IT-measure}
Let $\act$ be an action and let $k\in \N$. We say that $\mu\in \mathcal{M}(X^k)$ is an \textbf{IT-measure} if for any $n\in \N$ and any $(x_1^{(i)},\dots,x_k^{(i)})\in \supp(\mu)$ for $i\in \{1,\dots,n\}$, one has
$$
\big((x_1^{(1)},\dots,x_1^{(n)}),(x_2^{(1)},\dots,x_2^{(n)}),\dots,(x_k^{(1)},\dots,x_k^{(n)})\big)\in \IT_k(X^n,\Gamma),
$$
where $\Gamma$ acts diagonally on $X^n$.
We denote by $\mathcal{M}_{\mathrm{IT}}(X^k)$ the set of all IT-measures in $\mathcal{M}(X^k)$. 
\end{definition}
 
The following result describes the relation between $\IT_k(\mathcal{M}(X),\Gamma)$ and $\IT_k(X,\Gamma)$ \cite[Theorem 1.7]{liliu}.

 \begin{theorem}
 [Li and Liu \cite{liliu}]
 \label{thm:inducedIT-tuples}
Let $\act$ be an action. For each $k\in\mathbb{N}$, the following hold:
\begin{enumerate}
    \item $\IT_k(\mathcal{M}(X),\Gamma)$ is contained in the closed convex hull of $\IT_k(X,\Gamma)$ in $\mathcal{M}(X)^k$. 
    
    \item $\IT_k(X,\Gamma) = \IT_k(\mathcal{M}(X),\Gamma) \cap X^k$. 
    
    \item $\IT_k(\mathcal{M}(X),\Gamma) = \{(\mu^{(1)},\dots,\mu^{(k)}) : \mu \in \mathcal{M}_{\IT}(X^k)\}$. 
\end{enumerate}
\end{theorem}

The next result follows from \cite[Proposition 6.6]{KerrLi2007}. It also follows from \cref{thm:inducedIT-tuples}.
 \begin{theorem}
 [Kerr and Li \cite{KerrLi2007}]
 \label{thm:indtame}
 Let $\act$ be an action. Then $\act$ is not tame if and only if $\Gamma \curvearrowright\mathcal{M}(X)$ is not tame.   
 \end{theorem}

Finally, we present a new result along the same lines. 
  \begin{theorem}
 \label{thm:infty}
 Let $\act$ be an action. Then $\act$ is not tame if and only if $\Gamma \curvearrowright\mathcal{M}(X)$ is not tame if and only if $\IT_k(\mathcal{M}(X),\Gamma)\setminus \Delta_k^2(\mathcal{M}(X))\neq \emptyset$ for every $k\geq 2$.   
 \end{theorem}

 \begin{proof}
The first equivalence follows from \cref{thm:indtame}. 

Assume that $\act$ is not tame. There exists $(x,y)\in \IT_2(X,\Gamma)\setminus \Delta_2(X)$ (\cref{thm:charactame}). Let $k\in \N$. We define $\mu_i=\frac{i\delta_x+(k-i)\delta_y}{k}$ for every $0\leq i \leq k$ (where $\delta_x$ is the Dirac measure supported on $\{x\}$). 

Let $U_0\times \cdots \times U_k$ be an open product neighborhood of $(\mu_0,\mu_1,\dots,\mu_k)$. Pick $U_x,U_y$ disjoint open neighborhoods of $x,y$ sufficiently small, so the following condition is satisfied for every $0\leq i\leq k$: if $p_1,...,p_i\in U_x$ and $p_{i+1},...,p_k\in U_y$ then the measure $\frac{\sum_{j=1}^{k}\delta_{p_j}}{k}\in U_i$.

Since $(x,y)\in \IT_2(X,\Gamma)\setminus \Delta_2(X)$, there exists an infinite independence set, $J$, for $(U_x,U_y)$. Let $\phi\colon J\rightarrow \{0,\dots,k\}$. There exists a set of functions $\phi_{i}\colon J\rightarrow \{x,y\}$ for $1\leq i\leq k$ such that $$|\{1\leq i\leq k: \phi_{i}(g)=x\}|=\phi(g)$$ for every $g\in J$.
Since $J$ is an independence set for $(U_x,U_y)$, for every $1\leq i\leq k$, there exists $z_i\in X$ such that $gz_i\in U_{\phi_i(g)}$ for every $g\in J$. 

Let $\mathbf{z}=\{z_1,\dots,z_k\}$. We have that 
\begin{gather*}
    |\{i\in\{1,...,k\}: gz_i\in U_x\}|= \phi (g) 
\text{, and} \\
     |\{i\in\{1,...,k\}: gz_i\in U_y\}|= k-\phi(g),
\end{gather*}
for every $g\in J$.

Let $\mu_{\mathbf{z}}=\frac{\sum_{i=1}^{k}\delta_{z_i}}{k}$. Using the condition for $U_x$ and $U_y$ we have that $g\mu_{\mathbf{z}}\in U_{\phi (g)}$ for every $g\in J$. Thus $J$ is an independence set for $(U_0,\dots,U_k)$ and $(\mu_0,\mu_1,\dots,\mu_k)\in \IT_{k+1}(\mathcal{M}(X),\Gamma)$. 

We conclude the result. 
 \end{proof}

\section{Measure tuples}
Let $\Gamma \curvearrowright X$ be an action. A Borel probability measure $\mu$ on $X$ is \textbf{$\Gamma$-invariant} if, for every Borel subset $A \subset X$ and $g \in \Gamma$, we have $\mu(A) = \mu(g^{-1}A)$. In this case, we say that $\Gamma \curvearrowright (X, \mu)$ is a \textbf{probability measure-preserving (p.m.p.) action}.
We denote the space of all $\Gamma$-invariant Borel probability measures by $\mathcal{M}(\act)$.

Let $\Gamma \curvearrowright (X,\mu)$ be a p.m.p. action. For a finite partition $\mathcal{P}$ of $X$ consisting of Borel sets, the \textbf{Shannon entropy} of $\mathcal{P}$ with respect to $\mu$ is given by \[ H_{\mu}(\mathcal{P}) = \sum_{A \in \mathcal{P}} -\mu(A)\log \mu(A). \]
For an amenable group $\Gamma$, the \textbf{measure-theoretical entropy of $\Gamma \curvearrowright (X,\mu)$ with respect to $\mathcal{P}$} is given by \[ h_{\mu}(\Gamma\curvearrowright X, \mathcal{P}) = \lim_{n\to\infty}\frac{1}{\left\vert
	F_{n}\right\vert }\log H_{\mu}(\bigvee_{g \in F_n} g^{-1}\mathcal{P}), \]
where $\{F_n\}_{n \in \NN}$ is a F\o lner sequence. The \textbf{measure-theoretical entropy} of $\Gamma\curvearrowright (X,\mu)$ is the supremum of $h_{\mu}(\Gamma\curvearrowright X, \mathcal{P})$ taken over all finite Borel partitions: \[h_{\mu}(\Gamma\curvearrowright X) = \sup_{\mathcal{P} \mbox{ finite}}h_{\mu}(\Gamma\curvearrowright X, \mathcal{P}).\]

The variational principle relates the topological entropy with the measure-theoretical entropy through the following formula \cite[Theorem 5.2.7]{Ollagnier1985book}.

\begin{theorem}
	Let $\Gamma$ be an amenable group and $\act$ an action. We have \[
	\htop(\act) = \sup_{\mu \in \mathcal{M}(\Gamma\curvearrowright X)}h_{\mu}(\Gamma\curvearrowright X).
	\]
\end{theorem}

\begin{remark}
 The theory of measure-theoretical entropy for actions of sofic groups is well developed, but we will not use it in this subsection. See  \cite{Bowen2010_2,KerrLi2011,kerr2013sofic} for more information.  
\end{remark}

Let $\Gamma\curvearrowright X$ be an action, $\mu$ a $\Gamma$-invariant Borel probability measure,
$S=\{s_i\}_{i\in \N}$ a sequence in $\Gamma$
and $\mathcal{P}$ a finite Borel partition of $X$.
The \textbf{sequence entropy of $\Gamma\curvearrowright (X,\mu)$ along $S$ with respect to $\mathcal{P}$} is defined by
\[h_\mu^S(\act,\mathcal{P})=\limsup \frac{1}{n}
H_\mu\left( \bigvee_{i=0}^{n-1} T^{-s_i}\mathcal{P} \right).\]
The \textbf{sequence entropy of $\Gamma\curvearrowright (X,\mu)$ along $S$} is
\[h_\mu^S(\act)=\sup_{\mathcal{P}} h_\mu^S(\act,\mathcal{P}),\]
where the supremum is taken over all finite Borel partitions $\mathcal{P}$. The \textbf{sequence entropy} of an action is defined as the supremum over all sequences:
\[
h^{*}_{\mu}(\act)=\sup _{\mbox{sequence } S}h^{S}_{\mu}(\act).
\]

The topological sequence entropy is always larger than or equal to the sequence entropy with respect to any invariant measure.

\subsection{Measure entropy tuples}

Measure entropy pairs were introduced by Blanchard, Host, Maass, Martinez and Rudolph \cite{BlanchardHostMaassMartinezRudolph1995}. Measure entropy tuples are a natural generalization.
\begin{definition}
    Let $\Gamma$ be an amenable group, $\act$ an action, $\mu\in \mathcal{M}(\act)$, $k\geq 2$ and $(x_1,\dots,x_k)\in X^k\setminus \Delta_k(X)$. We say $(x_1,\dots,x_k)$ is a \textbf{$\mu$-entropy tuple} if for every finite Borel partition $\mathcal{P}$ such that $\{x_1,\dots,x_k\}\not\subset \overline{P}$ for every $P\in \mathcal{P}$, we have that 
    \[
    h_{\mu}(\act,\mathcal{P})>0.
    \]
\end{definition}

Combinatorial independence can also be used to study $\mu$-entropy pairs.

We denote the family of Borel subsets of $X$ by $\mathcal{B} (X)$. For $\delta > 0$ denote by $\mathcal{B} (\mu , \delta )$ the collection of all
Borel subsets $D$ of $X$ such that $\mu (D) \geq 1 - \delta$. Given an action $\act$, we denote
by $\mathcal{B}' (\mu , \delta )$ the collection of all maps
$D\colon \Gamma\rightarrow \mathcal{B} (X)$ such that $\inf_{s\in \Gamma} \mu (D_s ) \geq 1-\delta$.

\begin{definition}\label{D-indset}
Let $\act$ be an action, $\oA = (A_1, \dots, A_k)$ a
tuple of subsets of $X$, and let $D$ be a map from $\Gamma$ to the power set $2^X$
of $X$, with the image of $s\in \Gamma$ written as $D_s$. We say that a set $J\subset \Gamma$ is
an {\bf independence set for $\oA$ relative to $D$} if for every $F\Subset J$ and map $\phi\colon F\rightarrow \{1,\dots, k\}$ we have
$\bigcap_{s\in F} (D_s \cap s^{-1} A_{\phi (s)} ) \neq\emptyset$.
For a subset $D\subset X$, we say that $J$ is an
{\bf independence set for $\oA$ relative to $D$} if for every $F\Subset J$ and map $\phi\colon F\rightarrow \{1,\dots, k\}$ we have
$D\cap \bigcap_{s\in F} s^{-1} A_{\sigma (s)} \neq\emptyset$, i.e., if $J$
is an independence set for $\oA$ relative to the map $\Gamma\rightarrow 2^X$
with constant value $D$.
\end{definition}

Let $\oA = (A_1, \dots, A_k)$ be a tuple of subsets of $X$ and $\delta > 0$.
For every $F\Subset \Gamma$ we define
\begin{align*}
\varphi_{\oA ,\delta} (F) &= \min_{D\in \mathcal{B} (\mu , \delta )}
\max \big\{ |F\cap J| : J\text{ is an independence set for }\oA
\text{ relative to } D \big\} .
\end{align*}
The limit of $\frac{1}{|F|} \varphi_{\oA ,\delta} (F)$ 
as $F$ becomes more and more invariant
might not exist. We define $\upind_\mu (\oA , \delta )$ to be the limit supremum
of $\frac{1}{|F|} \varphi_{\oA ,\delta} (F)$ as $F$ becomes more and more invariant. We refer to the quantity
$\sup_{\delta > 0} \upind_\mu (\oA , \delta)$ as the {\bf upper $\mu$-independence density} of $\mathbf{A}$.

\begin{definition}
\label{def:muIEtuple}
Let $\Gamma\curvearrowright X$ be an action, $\mu\in \mathcal{M}(\act)$ and $k\geq 1$. We say that $\ox=(x_1,\dots,x_k)\in X^k$ is a \textbf{$\mu$-IE-tuple}
if for every product neighborhood $U_1\times \cdots \times U_k$ of $\ox$, we have that $(U_1,\dots,U_k)$ has positive upper $\mu$-independence density.
We denote the set of $\mu$-IE-tuples of length $k$ by $\IE^{\mu}_k(X,\Gamma)$.
\end{definition}
 
The following result appeared in \cite[Theorem 2.27]{KerrLi2009}. 
\begin{theorem}
[Kerr and Li \cite{KerrLi2009}]
  Let $\Gamma$ be an amenable group, $\act$ an action, $\mu\in \mathcal{M}(\act)$, $k\geq 2$ and  $(x_1,\dots, x_k) \in X^k \setminus \Delta_k(X)$. Then $(x_1,\dots, x_k)$ is a $\mu$-entropy tuple if and only if it is a ${\mu}$-IE-tuple. 
\end{theorem}

For a proof of the following result, see \cite[Theorem 2.23]{KerrLi2009}. 
\begin{theorem}
    [Blanchard et al \cite{BlanchardHostMaassMartinezRudolph1995}, Blanchard et al \cite{blanchard1997variation}, Kerr and Li \cite{KerrLi2009}]
    Let $\Gamma$ be an amenable group, $\act$ an action and $k\geq 2$. 
    We have that 
    \begin{enumerate}
        \item $\IE_k^{\mu}(X,\Gamma)\subset \IE_k(X,\Gamma)$ for every $\mu \in \mathcal{M}(\act)$ and
        \item there exists $\mu \in \mathcal{M}(\act)$ such that $\IE_k^{\mu}(X,\Gamma)= \IE_k(X,\Gamma)$.
    \end{enumerate}
\end{theorem}

For results regarding the measure IE-tuples of the induced action on $\mathcal{M}(X)$ see \cite{vermersch2024measure,liliu}.

\subsection{Measure sequence entropy tuples}
Measure sequence entropy pairs were introduced by Huang, Maass and Ye \cite{huang2004sequence}. 
{They are connected to sequence entropy, combinatorial independence, and mean sensitivity.}

\begin{definition}
Let $\Gamma\curvearrowright X$ be an action, $\mu\in \mathcal{M}(\act)$ and $k\geq 2$. 
We say that $(x_1,..,x_k)\in X^k\setminus \Delta_k(X)$ is a $\mu$-\textbf{sequence entropy tuple}
if for any finite Borel partition $\mathcal{P}$ such that $\{x_1,\dots,x_k\}\not\subset \overline{P}$ for every $P\in \mathcal{P}$,
there exists a sequence $S$ in $\Gamma$ with $h_\mu^S(\act,\mathcal{P})>0$.

The set of $\mu$-sequence entropy tuples will be denoted by $\mathrm{SE}_k(X,\Gamma,\mu)$.
\end{definition}

Every $\mu$-sequence entropy tuple of length $k$ is also a (topological) sequence entropy tuple \cite[Corollary 3.6]{huang2004sequence}, but the converse does not necessarily hold. For instance, any topologically weakly mixing system $(X,T)$ that admits a unique
invariant measure $\mu$ supported on a fixed point provides an example in which
no $\mu$-sequence entropy pairs exist, even though
$\mathrm{IN}_2(X,\Z) = X^{2}$ \cite[Remark 4.2]{huang2004sequence}.

For a proof of the following result, see \cite[Proposition 4.7]{KerrLi2009}.
\begin{theorem}
    [Huang, Maass and Ye \cite{huang2004sequence}, Kerr and Li \cite{KerrLi2009}]
    Let $\act$ be an action and $\mu\in \mathcal{M}(\act)$. Then $h_\mu^*(\act)>0$ if and only if there exists a $\mu$-sequence entropy pair $(x_1,x_2)\in X^2\setminus \Delta_2(X)$.
    \end{theorem}

\begin{definition}
    Let $\Gamma\curvearrowright X$ be an action and $\mu\in \mathcal{M}(\act)$. We define the \textbf{SE-rank} of $\act$ as
\[
\begin{aligned}
R_{\SE}(X,\Gamma,\mu) &= \sup\bigl\{\, k\in\N : \SE_k(X,\Gamma,\mu)\not\subset \Delta_k^2(X) \,\bigr\}
\end{aligned}
\]
\end{definition}

The SE rank can also be used to estimate sequence entropy. The following result was proved for $G=\Z$ in \cite[Theorem 3.8]{HuangYe2009}. It follows from \cite[Theorem 1.3]{liu2025independenceMS} and \cite[Theorem 2.9]{liu2025sequenceIT}.

\begin{theorem}
    [Huang and Ye \cite{HuangYe2009}, Liu, Xu and Zhang \cite{liu2025independenceMS}, Liu, Wang and Xu \cite{liu2025sequenceIT}]
    \label{thm:seqformula}
    Let $\act$ be an action, $\mu\in \mathcal{M}(\act)$ and $k\in \N$ with $k\geq 2$. We have that
    \[
    h_\mu^*(\act)\leq \log R_{\SE}(X,\Gamma,\mu).
    \]
    
    Furthermore, if $\mu$ is ergodic, then 
    \[
     h_\mu^*(\act)= \log R_{\SE}(X,\Gamma,\mu).
    \]
\end{theorem}

For the the proof \cref{indms} see \cite[Theorem 1.1]{liu2025independenceMS}.
\begin{theorem}
[Liu, Xu and Zhang \cite{liu2025independenceMS}]
\label{indms}
   Let $\act$ be an action and $\mu\in \mathcal{M}(\act)$. If $(x_1,\dots,x_k)\in X^k $ is a $\mu$-sequence entropy tuple then $(x_1,\dots,x_k)\in \IT_k(X,\Gamma)$.
\end{theorem}

This implies that the ergodic measure of minimal actions with finite IT-rank have bounded sequence entropy.
\begin{corollary}
    Let $\act$ be a minimal action and $\mu\in \cM_e(\act)$. We have \[h_\mu^*(\act)
     \leq \log R_{\IT}(X,\Gamma).
    \]
\end{corollary}

\begin{proof}
    Use \cref{indms,thm:seqformula}.
\end{proof}

We will prove a stronger version of the following result later (see \cref{T-length}).
\begin{theorem}
\label{chido}
    Let $\act$ be a minimal action. We have that 
    \[
     \sum_{\mu \in \mathcal{M}_e(\act)} R_{\SE}(X,\Gamma,\mu)\leq R_{\IT}(X,\Gamma).
    \]
\end{theorem}

Using \cref{chido}, \cref{thm:seqformula} and \cref{thm:seqentr} we recover \cite[Theorem 1.1]{liu2025sequenceIT}.

\begin{corollary}
[Liu, Wang and Xu \cite{liu2025sequenceIT}]
    Let $\act$ be a minimal action. We have that 
    \[
    \log \sum_{\mu \in \mathcal{M}_e(\act)} e^{h^{*}_{\mu}(\act)}\leq \htop^{*}(\act).
    \]
\end{corollary}
When the acting group is abelian a stronger version of \cref{chido} may be written in terms of the continuous and measurable spectra (see \cref{cor:mainineq}).

For results regarding the measure IN-tuples of the induced action on $\mathcal{M}(X)$, see \cite{liliu}.

\begin{question}
Let $\Gamma \curvearrowright (X,\mu)$ be an ergodic action with
\[
h_\mu^*(\act)<\infty.
\]
Does $\Gamma \curvearrowright (X,\mu)$ admit a minimal topological model of finite IT-rank?
\end{question}

\subsection{Characterizations of measure sequence entropy tuples}
Let $\Gamma$ be an amenable group, $\{F_n\}_{n\in \N}$ a Følner sequence for $\Gamma$ and $S\subset \Gamma$.
We define the upper density of $S$ as
\[\overline{D}(S)=\limsup_{n\to\infty} \frac{|S\cap F_n|}{|F_n|}.\]

Measure mean sensitivity is a strong form of sensitivity of initial conditions introduced in \cite{garcia2017weak}. It is invariant under (measure) isomorphisms and is related to discrete spectrum. This notion can be localized as follows.

\begin{definition}
Let $\Gamma$ be an amenable group, $\{F_n\}_{n\in \N}$ a Følner sequence for $\Gamma$, $\Gamma\curvearrowright X$ an action and $\mu\in \mathcal{M}(\act)$.
We say that $\ox=(x_1,\dots,x_k) \in X^k\setminus \Delta_k(X)$ is a $\mu$-\textbf{mean sensitivity tuple} if for every product neighborhood $U_1\times \cdots \times U_k$ of $\ox$,
there exists $\varepsilon>0$ such that for every Borel subset $A$ of positive measure,
there exist $p_1,\dots,p_k\in A$ such that
$$\overline{D}(\{s\in \Gamma: sp_i\in U_{i}\text{ } \forall i \in \{1,\dots,k\} \})>\varepsilon.$$
\end{definition}

Let $\delta>0$ and $K\Subset \Gamma$. A set $F\Subset \Gamma$ is said to be
\textbf{right $(K,\delta)$-invariant} if
\[
|FK\triangle F|\leq \delta |F|.
\]
A Følner sequence $\{F_n\}_{n\in\mathbb N}$ of $\Gamma$
is said to be \textbf{two-sided} if it is eventually right
$(K,\delta)$-invariant for every set $K\Subset \Gamma$ and every
$\delta>0$.

A Følner sequence $\{F_n\}_{n \in \mathbb{N}}$ is called \textbf{tempered} if there exists a constant $C > 0$ such that for all $n \in \mathbb{N}$,
\[
\left| \bigcup_{k < n} F_k^{-1} F_n \right| \leq C |F_n|.
\]
\cref{thm:mspair} was proved for actions of abelian groups in \cite[Corollary 39 and Section 5.2]{garcia2017weak}. This generalization follows from \cite[Theorem 26 and Theorem 32]{garcia2017weak}, \cite[Theorem 1.3]{Yu2021} and \cite[Theorem 4]{Kushnirenko1967MetricInvariantsEntropy}.
\begin{theorem}
[García-Ramos \cite{garcia2017weak}, Yu, Zhang and Zhang \cite{Yu2021}]
\label{thm:mspair}
Let $\Gamma$ be an amenable group, $\{F_n\}_{n\in \N}$ a tempered two-sided Følner sequence for $\Gamma$, $\Gamma\curvearrowright X$ an action, and $\mu$ an ergodic measure. We have $h_\mu^*(\act)>0$ if and only if there exists a $\mu$-mean sensitivity pair with respect to $\{F_n\}_{n\in \N}$.
\end{theorem}

One direction of the following result was studied (for $k=2$ and abelian groups) in \cite{garcia2022measure}\footnote{The characterization is also presented in that paper, but the proof contains a gap.}. Independently, the problem was studied for actions of $\mathbb{Z}$ in \cite[Theorem 1.3]{li2022sequence}. For the most general result, see \cite[Theorem 1.6]{liu2025independenceMS}.

\begin{theorem}
[Li, Liu, Tu and Yu \cite{li2022sequence}, García-Ramos and Muñoz-López \cite{garcia2022measure}, Liu, Xu and Zhang \cite{liu2025independenceMS}]
\label{thm:ms}

Let $\Gamma$ be an amenable group, $\{F_n\}_{n\in \N}$ a tempered Følner sequence for $\Gamma$, $\act$ an action, $\mu$ an ergodic measure and $k\geq 2$.
We have that $(x_1,\dots,x_k)\in X^k\setminus \Delta_k(X)$ is a $\mu$-sequence entropy tuple if and only if $(x_1,\dots,x_k)$ is a $\mu$-mean sensitivity tuple with respect to $\{F_n\}_{n\in \N}$.
\end{theorem}

Measure sequence entropy tuples can also be characterized combinatorially.  
\begin{definition}
\label{def:indfiniteset}
Let $\Gamma\curvearrowright X$ be an action and $\mu\in \mathcal{M}(\act)$.
For $A_1,\dots,A_k\subset X$ and $\varepsilon>0$
we say that $(A_1,\dots,A_k)$ has $(\varepsilon,\mu)$-\textbf{independence
over arbitrarily large finite sets}
if there exists $c>0$ such that for every $N>0$
there is a finite set $F\subset \Gamma$ with $|F|>N$
such that for every $E\in\mathcal{B}'({\mu},\varepsilon)$
there is an independence set $I\subset F$ for $(A_1,\dots,A_k)$ relative to $E$
with $|I|\geq c|F|$.

\end{definition}
\begin{definition}
\label{def:INpair}
Let $\Gamma\curvearrowright X$ be an action, $\mu\in \mathcal{M}(\act)$ and $k\geq 1$. We say that $\ox=(x_1,\dots,x_k)\in X^k$ is a \textbf{$\mu$-IN-tuple}
if for every product neighborhood $U_1\times \cdots \times U_k$ of $\ox$
there exists $\varepsilon>0$ such that
$(U_{1},\dots,U_{k})$ has $(\varepsilon,\mu)$-independence
over arbitrarily large finite sets.
\end{definition}

For a proof of the following theorem see \cite[Theorem 4.9]{KerrLi2009}.
\begin{theorem}[Kerr and Li \cite{KerrLi2009}]
\label{thm:INpair}
Let $\Gamma\curvearrowright X$ be an action, $\mu\in \mathcal{M}(\act)$ and $k\geq 2$. A tuple $(x_1,\dots,x_k)\in X^k\setminus \Delta_k(X)$
is a ${\mu}$-IN-tuple if and only if it is a $\mu$-sequence entropy tuple.
\end{theorem}

\section{ Variational inequality for IT-rank}
\label{schido}
The objective of this section is to prove \cref{T-length}.

Some of the notations used in this section have already been introduced. Since some readers may wish to consult this section independently, we repeat the relevant definitions where appropriate.

Let $X$ be a compact metrizable space.
We denote by $\sB_X$ the Borel $\sigma$-algebra of $X$. For simplicity, we may denote a Borel measure space $(X,\mathcal{B}_X,\mu)$ simply as $(X,\mu)$.

Let $\mu\in \cM(X)$. Denote by $L^2(X, \mu)$ the space of all square-integrable Borel functions $X\rightarrow \Cb$, modulo the equivalence relation of being equal $\mu$-a.e.

For a Hilbert space $\cH$, we denote by $\cU(\cH)$ the group of unitary operators of $\cH$ and by $B(\cH)$ the space of bounded linear operators $\cH\rightarrow \cH$. The strong operator topology on $B(\cH)$ is the topology of pointwise convergence in norm: a net $(T_j)_{j\in J}$ in $B(\cH)$ converges to some $T\in B(\cH)$ in the strong operator topology if and only if, for every $u\in\cH$, we have $\|T_ju-Tu\|\to 0 $ as $j\to \infty$ in $J$.

We denote by $C(X)$ the space of continuous functions $X\rightarrow \Cb$. Equipped with the supremum norm $\|f\|_{\infty}=\sup_{x\in X}|f(x)|$, the space $C(X)$ is a Banach space.

Given an action $\Gamma\curvearrowright X$, we use the same notation for the
\textbf{induced actions} on functions $f\colon X \to \Cb$ and Borel probability measures $\mu$ on $X$. For
$\gamma\in\Gamma$, these are defined by
\[
(\gamma f)(x)=f(\gamma^{-1}x)
\]
and
\[
(\gamma\mu)(A)=\mu(\gamma^{-1}A).
\]
 We have 
\[
\int_X f\,d(\gamma\mu)
=
\int_X \gamma^{-1}f\,d\mu.
\]

Let $Z$ be a compact metrizable space and let $\pi\colon X\rightarrow Z$ be a Borel map.
We define the \textbf{pushforward} of measures as
\[
\pi_*\colon \cM(X)\to \cM(Z),\qquad
\pi_*\mu(A)=\mu\bigl(\pi^{-1}(A)\bigr)\ \ \text{for every }A\in\mathcal{B}_Z,
\]
{
and then define the \textbf{pullback} $\pi^*\colon \mathcal F(Z)\to \mathcal F(X)$ as
\[
\pi^*(f)=f\circ \pi,
\]
where $\mathcal F(W)$ denotes the space of complex-valued functions on $W$.}
Let $\mu\in \cM(X)$. Then $\pi^*$ restricts to a unital algebra homomorphism
\[
\pi^*\colon C(Z)\to L^\infty(X,\mu).
\]

If $\pi_*\mu$ has full support on $Z$, then $\pi^*\colon C(Z)\to L^\infty(X,\mu)$ is injective,
and hence a unital algebra isomorphism from $C(Z)$ onto its range $\pi^*(C(Z))\subset L^\infty(X,\mu)$.

Let $\mu\in \cM(X)$. By an action $\Gamma\curvearrowright (X,\mu)$ we mean a Borel action $\act$ (not necessarily continuous) that preserves $\mu$; that is, $
\mu(g^{-1}A)=\mu(A)$ for every $g\in\Gamma$ and $A\in \sB_X$. We say $\Gamma \curvearrowright (X,\mu)$ is \textbf{ergodic} if $\mu\in\cM_e(\act)$. 

We say the action $\Gamma\curvearrowright (Y,\nu)$ is a \textbf{factor} of $\Gamma \curvearrowright (X,\mu)$ if there is a measurable (modulo null sets) map $\pi\colon X\to Y$ such that $\pi_\ast\mu=\nu$ and is $\Gamma$-equivariant, that is, $\pi(\gamma x)=\gamma\pi(x)$ for $\mu$-a.e.\ $x\in X$ and every $\gamma\in\Gamma$. In this case we say $\pi$ is a \textbf{factor map}. If $\pi$ is bijective (modulo null sets) then we say $\Gamma\curvearrowright (X,\mu)$ and $\Gamma\curvearrowright (Y,\nu)$ are \textbf{isomorphic}, $\pi$ is an \textbf{isomorphism} and we write $\Gamma\curvearrowright (X,\mu)\cong \Gamma\curvearrowright (Y,\nu)$.

Let $G$ be a topological group and $\mathcal H$ a Hilbert space. A \textbf{unitary
representation} of $G$ on $\mathcal H$ is a group homomorphism
\[
U\colon G\to\mathcal U(\mathcal H).
\]
The representation is \textbf{strongly continuous} if,
for every $f\in\mathcal H$, the map
\[
G\to\mathcal H,
\qquad
s \mapsto U(s)f,
\]
is continuous. 

Let $\Gamma \curvearrowright (X,\mu)$ be an action. The \textbf{Koopman representation}
associated with $\Gamma \curvearrowright (X,\mu)$ is the homomorphism
\[
U_\mu\colon\Gamma\to
\mathcal U\bigl(L^2(X,\mu)\bigr)
\]
defined by
\[
U_\mu(\gamma)(f)=\gamma f
\]
for every $\gamma\in\Gamma$ and $f\in L^2(X,\mu)$.

Let $\mathcal K_\mu\subset L^2(X,\mu)$ denote the closed linear span of
all finite-dimensional $\Gamma$-invariant linear subspaces. We call
$\mathcal K_\mu$ the \textbf{Kronecker subspace} of $L^2(X,\mu)$ associated to $\Gamma \curvearrowright (X,\mu)$.

It is well known that 
\[
\mathcal K_\mu
=
\left\{
f\in L^2(X,\mu):
\{\gamma f:\gamma\in\Gamma\}
\text{ is totally bounded in }L^2(X,\mu)
\right\},
\]e.g. see \cite[Theorem~1.2]{BergelsonRosenblatt88}. Thus, $\mathcal K_\mu$ is
precisely the space of almost periodic functions.

Define
\[
\Sigma_\mu
=
\{A\in\sB_X:1_A\in\mathcal K_\mu\}.
\]

\begin{proposition}[Zimmer {\cite[Theorem~7.1]{Zimmer76}}]
\label{prop_kronecker}
Let $\Gamma\curvearrowright(X,\mu)$ be an action. Then
$\Sigma_\mu$ is a $\Gamma$-invariant sub-$\sigma$-algebra of $\sB_X$ and
\[
\mathcal K_\mu=L^2(X,\Sigma_\mu,\mu).
\]
Moreover, there exist a measure-preserving $\Gamma$-system
$(Z_\mu,\nu_\mu)$, a $\Gamma$-invariant conull Borel set
$X'_\mu\subset X$, and a Borel map
\[
\pi_\mu\colon X'_\mu\to Z_\mu
\]
such that
\[
\pi_\mu(\gamma x)=\gamma\pi_\mu(x)
\qquad
\text{for every }\gamma\in\Gamma\text{ and }x\in X'_\mu,
\]
and
\[
\pi_\mu^{-1}(\sB_{Z_\mu})=\Sigma_\mu
\qquad\text{modulo }\mu\text{-null sets}.
\]
\end{proposition}

We call $\Sigma_\mu$ the \textbf{Kronecker $\sigma$-algebra},
$\Gamma\curvearrowright(Z_\mu,\nu_\mu)$ the \textbf{Kronecker factor},
and $\pi_\mu$ the \textbf{Kronecker factor map}. Since factors correspond,
modulo null sets, to invariant sub-$\sigma$-algebras, the Kronecker factor
is unique up to isomorphism.

The pullback
\[
\pi^*_\mu\colon L^2(Z_\mu,\nu_\mu)
\to L^2(X,\Sigma_\mu,\mu)
\]
is a $\Gamma$-equivariant unitary isomorphism. Thus,
\[
L^2(Z_\mu,\nu_\mu)
\cong L^2(X,\Sigma_\mu,\mu)
=\mathcal K_\mu.
\]

We will also consider continuous actions of not necessarily countable groups. The definition is analogous. 

For any compact metrizable group $G$ and any closed subgroup $H\leq G$, there is a unique $G$-invariant Borel probability measure on $G/H$ \cite[Corollary~2.51]{Folland2016ACourseAHA}. We denote it by $\nu_{G/H}$. We have that $\nu_{G/H}$ is fully supported \cite[Proposition 2.58]{Folland2016ACourseAHA}.
The measure $\nu_{G/\{e_G\}}$ is the usual normalized Haar measure and will be denoted by $\nu_G$.

\begin{definition}
Let $G$ be a topological group and $G\curvearrowright X$ an action.
We say that $X$ is a \textbf{homogeneous $G$-space} if the action is transitive on points;
that is, for every $y_1,y_2\in X$, there exists $g\in G$ such that
\[
gy_1=y_2.
\]
Equivalently, $X=Gy_0$ for any $y_0\in X$.
\end{definition}

For ergodic actions of abelian groups, the Kronecker factor is isomorphic to a rotation on a compact abelian group. In general, Mackey showed that $Z_{\mu}$ can be realized as a homogeneous space of a compact group \cite{Mackey64}.
\begin{theorem}[Mackey]
\label{mackey}
Let $\Gamma\curvearrowright (X,\mu)$ be an ergodic action.
Then there exist a compact metrizable group $G_\mu$, a closed subgroup
$H_\mu\leq G_\mu$, and a group homomorphism
\[
\phi_\mu\colon \Gamma\to G_\mu
\]
such that the following properties hold:
\begin{enumerate}[label=(\roman*)]
\item $\phi_\mu(\Gamma)$ is dense in $G_\mu$;

\item 
\[
\bigcap_{g\in G_\mu}gH_\mu g^{-1}
=
\{e_{G_\mu}\};
\]

\item the actions
\[
\Gamma\curvearrowright (Z_\mu,\nu_\mu)
\quad\text{and}\quad
\Gamma\curvearrowright
(G_\mu/H_\mu,\nu_{G_\mu/H_\mu})
\]
are isomorphic, and the latter action is
given by
\[
\gamma\cdot(gH_\mu)
=
\phi_\mu(\gamma)gH_\mu
\]
for every $\gamma\in\Gamma$ and $g\in G_\mu$.
\end{enumerate}
 The realization is unique in the following sense:
if $(G'_\mu,H'_\mu,\phi'_\mu)$ is another triple satisfying
items {\rm(i)--(iii)}, then there exist a topological group
isomorphism
\[
\alpha\colon G_\mu\to G'_\mu
\]
and an element $a\in G'_\mu$ such that
\[
\alpha\circ\phi_\mu=\phi'_\mu
\qquad\text{and}\qquad
\alpha(H_\mu)=aH'_\mu a^{-1}.
\]

\end{theorem}
\begin{proof}
Since $L^2(Z_\mu,\nu_\mu)\cong \cK_{\mu}$ and $\mathcal K_\mu$ is the closed linear span of the
finite-dimensional $\Gamma$-invariant linear subspaces of $L^2(X,\mu)$ (pure point spectrum) by
\cite[Theorem~1]{Mackey64}, there exist a compact metrizable group
$G_\mu$, a closed subgroup $H_\mu\leq G_\mu$, and a homomorphism
\[
\phi_\mu\colon\Gamma\to G_\mu
\]
with dense image such that
\[
\Gamma\curvearrowright(Z_\mu,\nu_\mu)
\cong
\Gamma\curvearrowright
(G_\mu/H_\mu,\nu_{G_\mu/H_\mu}).
\]

We may assume that the action of $G_\mu$ on $G_\mu/H_\mu$ is faithful.
Indeed, let
\[
N_\mu
=
\bigcap_{g\in G_\mu}gH_\mu g^{-1}.
\]
Then $N_\mu$ is the kernel of the action of $G_\mu$ on $G_\mu/H_\mu$.
If
\[
q_\mu\colon G_\mu\to G_\mu/N_\mu
\]
denotes the quotient homomorphism, then there is a natural $\Gamma$-equivariant
homeomorphism that induces
\[
G_\mu/H_\mu
\cong
(G_\mu/N_\mu)/(H_\mu/N_\mu).
\]
Thus, replacing $(G_\mu,H_\mu,\phi_\mu)$ by
\[
\bigl(G_\mu/N_\mu,\,
      H_\mu/N_\mu,\,
      q_\mu\circ\phi_\mu\bigr),
\]
we may assume that
\[
\bigcap_{g\in G_\mu}gH_\mu g^{-1}
=
\{e_{G_\mu}\}.
\]

We now prove uniqueness. 
Recall $U_{\mu}$ is the Koopman representation.

Since
\[
U_\mu(\gamma)(\mathcal K_\mu)=\mathcal K_\mu
\qquad\text{for every }\gamma\in\Gamma,
\]
the Koopman representation restricts to a unitary representation
\[
\rho_\mu\colon\Gamma\to\mathcal U(\mathcal K_\mu),
\qquad
\rho_\mu(\gamma)
=
U_\mu(\gamma)|_{\mathcal K_\mu}.
\]

Suppose that $(G_i,H_i,\phi_i)$, $i\in\{1,2\}$, satisfy items
{\rm(i)--(iii)}. For each $i\in\{1,2\}$, let
\[
\beta_i\colon
(Z_\mu,\nu_\mu)
\to
(G_i/H_i,\nu_{G_i/H_i})
\]
be a $\Gamma$-equivariant isomorphism, and let
\[
U_i\colon
G_i\to
\mathcal U\bigl(L^2(G_i/H_i,\nu_{G_i/H_i})\bigr)
\]
be the Koopman representation associated with the action of $G_i$ on
$G_i/H_i$. Define
\[
W_i
=
\pi^*_\mu\circ {\beta_i}^*
\colon
L^2(G_i/H_i,\nu_{G_i/H_i})
\to
\mathcal K_\mu.
\]
The $\Gamma$-equivariance of $\beta_i$ implies that
\begin{equation}
\label{iota-phi-rho}
W_iU_i\bigl(\phi_i(\gamma)\bigr)W_i^{-1}
=
\rho_\mu(\gamma)
\qquad
\text{for every }\gamma\in\Gamma.
\end{equation}

The representation $U_i$ is strongly continuous and
\[
\ker U_i
=
\bigcap_{g\in G_i}gH_i g^{-1}.
\]
By item {\rm(iii)}, $U_i$ is faithful. Since $G_i$ is compact, $U_i$ is
a topological embedding. Therefore,
\[
\iota_i\colon
G_i\to\mathcal U(\mathcal K_\mu),
\qquad
\iota_i(g)=W_iU_i(g)W_i^{-1},
\]
is also a topological embedding. By \eqref{iota-phi-rho},
\[
\iota_i\circ\phi_i=\rho_\mu.
\]
Since $\phi_i(\Gamma)$ is dense in $G_i$, it follows that
\begin{equation}
\label{Gmu-operator-realization}
\iota_i(G_i)
=
\overline{\rho_\mu(\Gamma)}^{\,\mathrm{SOT}}
\end{equation}
where $\overline{\rho_\mu(\Gamma)}^{\,\mathrm{SOT}}$ denotes the closure of $\rho_\mu(\Gamma)$ in the strong operator topology.
Consequently,
\[
\alpha
=
\iota_2^{-1}\circ\iota_1
\colon G_1\to G_2
\]
is a topological group isomorphism satisfying
\[
\alpha\circ\phi_1=\phi_2.
\]
Moreover, $\alpha$ is unique because $\phi_1(\Gamma)$ is dense in
$G_1$.

It remains to compare $H_1$ and $H_2$. Consider the measure-theoretic isomorphism
\[
\Psi
=
\beta_2\circ\beta_1^{-1}
\colon
G_1/H_1\to G_2/H_2.
\]
We define $G_1\curvearrowright G_2/H_2$ via
\[
g\cdot(g'H_2)=\alpha(g)g'H_2 
\]
for $g\in G_1$ and $g'\in G_2$.
 Moreover, for every
$f\in L^2(G_2/H_2,\nu_{G_2/H_2})$,
\begin{align*}
\Psi^* f
=f\circ\Psi
=f\circ\beta_2\circ\beta_1^{-1}
=W_1^{-1}W_2f.
\end{align*}
Hence
\[
\Psi^* U_2\bigl(\alpha(g)\bigr)
=
U_1(g)\Psi^*
\qquad
\text{for every }g\in G_1.
\]
Thus, for every $s\in G_1$ and every measurable $A\subset G_2/H_2$,
\[
\Psi^{-1}(sA)=s\Psi^{-1}(A)
\quad\text{modulo null sets}.
\] 
By Mackey's point-realization theorem
\cite[Theorem~2]{Mackey62} (see also \cite[Theorem 3.5]{Ramsay71}), the inverse-image map generated by $\Psi$ is induced by a Borel
$G_1$-equivariant isomorphism between $G_1$-invariant conull subsets $A_1\subset G_1/H_1$ and $A_2\subset G_2/H_2$, which we still denote by $\Psi$.
The left-translation action of $G_1$ on $G_1/H_1$ has a single orbit. Since
$\alpha$ is surjective, the action of $G_1$ on $G_2/H_2$ also has a single
orbit. Therefore, since $A_1$ and $A_2$ are nonempty and
$G_1$-invariant,
\[
A_1=G_1/H_1
\qquad\text{and}\qquad
A_2=G_2/H_2.
\]

Write
\[
\Psi(H_1)=aH_2
\]
for some $a\in G_2$. Equivariance and bijectivity imply that the
stabilizers of $H_1$ and $aH_2$ are equal. Therefore,
\[
H_1
=
\alpha^{-1}(aH_2a^{-1}),
\]
and hence
\[
\alpha(H_1)=aH_2a^{-1}.
\]
This proves uniqueness.
\end{proof}

\begin{notation}
\label{rm:fullys}
Given an ergodic action $\Gamma\curvearrowright(X,\mu)$, we fix a faithful
homogeneous realization $(G_\mu,H_\mu,\phi_\mu)$ provided by
\cref{mackey}, and identify
\[
(Z_\mu,\nu_\mu)
=
(G_\mu/H_\mu,\nu_{G_\mu/H_\mu}).
\]
\end{notation}

We define the action $G_\mu\curvearrowright Z_\mu$ as before:
\[
s\cdot(gH_\mu)=(sg)H_\mu.
\]

Using the pullback we define the action $G_{\mu}\curvearrowright \mathcal K_\mu$ 
\[
sf
=
\pi^*_\mu\bigl(s\bigl((\pi^*_\mu)^{-1}f\bigr)\bigr),
\qquad
s\in G_\mu,\quad f\in\mathcal K_\mu.
\]

Let
\[
U_{G_\mu}\colon G_\mu\to
\mathcal U\bigl(L^2(Z_\mu,\nu_\mu)\bigr)
\]
be the Koopman representation associated with the action $G_{\mu}\curvearrowright Z_{\mu}$. Since $\nu_\mu$ is $G_\mu$-invariant,
this is a strongly continuous unitary representation.

By \eqref{Gmu-operator-realization}, we identify $G_\mu$ with
\[
\overline{\rho_\mu(\Gamma)}^{\,\mathrm{SOT}}
\leq\mathcal U(\mathcal K_\mu).
\]
Under this identification,
\[
\phi_\mu(\gamma)
=
\rho_\mu(\gamma)
=
U_\mu(\gamma)|_{\mathcal K_\mu}
\qquad
\text{for every }\gamma\in\Gamma,
\]
and
\[
s
=
\pi^*_\mu\,U_{G_\mu}(s)\,(\pi^*_\mu)^{-1}
\qquad
\text{for every }s\in G_\mu.
\]
Thus, the canonical unitary representation of $G_\mu$ on $\mathcal K_\mu$
is strongly continuous.

Moreover, $\pi^*_\mu$ restricts to an isometric unital algebra
isomorphism
\[
\pi^*_\mu\colon
L^\infty(Z_\mu,\nu_\mu)
\to
L^\infty(X,\Sigma_\mu,\mu)
=
L^\infty(X,\mu)\cap\mathcal K_\mu.
\]
Since $U_{G_\mu}(s)$ preserves $L^\infty(Z_\mu,\nu_\mu)$ for every
$s\in G_\mu$, it follows that
\[
s\bigl(L^\infty(X,\Sigma_\mu,\mu)\bigr)
=
L^\infty(X,\Sigma_\mu,\mu).
\]
Therefore, the canonical action of $G_\mu$ on $\mathcal K_\mu$ restricts
to an action
\[
G_\mu\curvearrowright L^\infty(X,\Sigma_\mu,\mu)
\]
by isometric unital algebra automorphisms.
\subsection{A factor map from the Kronecker factor to the MEF}
\label{sub:factor}

The objective of this subsection is to introduce $\theta_\mu\colon Z_{\mu}\rightarrow X_{\eq}$, a natural factor map between the Kronecker factor and the maximal equicontinuous factor of an action (\cref{theta}).

\begin{lemma}
\label{L-cont to cont group2}
Let $G$ be a compact metrizable group, let $\nu_G$ be its normalized
Haar measure, and let $f\in L^\infty(G,\nu_G)$. For $s\in G$, let
$sf$ denote the element of $L^\infty(G,\nu_G)$ represented by
\[
x\longmapsto \widetilde f(s^{-1}x),
\]
where $\widetilde f$ is any representative of $f$. Then there exists
$h\in C(G)$ such that
\[
f=h\qquad \nu_G\text{-a.e.}
\]
if and only if the orbit map
\[
F_f\colon G\longrightarrow L^\infty(G,\nu_G),
\qquad
F_f(s)=sf,
\]
is continuous in the $L^\infty$-norm topology.
\end{lemma}

\begin{proof}
Since $\nu_G$ has full support, the canonical map
\[
C(G)\longrightarrow L^\infty(G,\nu_G)
\]
is an isometric embedding. We identify $C(G)$ with its image under this
map.

Suppose first that $f=h$ in $L^\infty(G,\nu_G)$ for some $h\in C(G)$.
Then $sf=sh$ for every $s\in G$, and hence
\[
\|sf-f\|_{\infty}
=
\sup_{x\in G}|h(s^{-1}x)-h(x)|
\longrightarrow 0
\qquad\text{as }s\to e_G.
\]
Indeed, the map
\[
G\times G\longrightarrow \mathbb C,
\qquad
(s,x)\longmapsto h(s^{-1}x)-h(x),
\]
is uniformly continuous. Thus $F_f$ is continuous at $e_G$. Since
translations act isometrically on $L^\infty(G,\nu_G)$,
\[
\|F_f(s)-F_f(t)\|_{\infty}
=
\|(t^{-1}s)f-f\|_{\infty},
\]
so $F_f$ is continuous on $G$.

Conversely, suppose that $F_f$ is continuous. Let $(U_n)_{n\in\mathbb N}$
be a decreasing basis of open neighborhoods of $e_G$. For every $n$,
choose $k_n\in C(G)$ such that
\[
k_n\geq 0,\qquad
\int_G k_n\,d\nu_G=1,
\qquad
\operatorname{supp}(k_n)\subset U_n.
\]
Such functions exist by Urysohn's lemma and the fact that $\nu_G$ has
full support.

Choose a bounded Borel representative $\widetilde f$ of $f$ satisfying
\[
|\widetilde f(x)|\leq \|f\|_{\infty}
\qquad\text{for every }x\in G,
\]
and define
\[
f_n(x)
=
\int_G k_n(xu^{-1})\widetilde f(u)\,d\nu_G(u),
\qquad x\in G.
\]
The map
\[
G\times G\longrightarrow \mathbb C,
\qquad
(x,u)\longmapsto k_n(xu^{-1}),
\]
is uniformly continuous. Therefore,
\[
\begin{aligned}
|f_n(x)-f_n(y)|
&\leq
\|f\|_{\infty}
\sup_{u\in G}
|k_n(xu^{-1})-k_n(yu^{-1})|
\longrightarrow 0
\end{aligned}
\]
as $y\to x$. Thus $f_n\in C(G)$.

Haar measure on a compact group is invariant under inversion and
translations. Making the change of variables $s=xu^{-1}$ gives
\[
f_n(x)
=
\int_G k_n(s)\widetilde f(s^{-1}x)\,d\nu_G(s).
\]

Set
\[
m(s)=\|sf-f\|_{\infty},
\qquad s\in G.
\]
By assumption, $m$ is continuous. For each $s\in G$, the function
\[
x\longmapsto
\widetilde f(s^{-1}x)-\widetilde f(x)
\]
represents $sf-f$, and hence
\[
|\widetilde f(s^{-1}x)-\widetilde f(x)|
\leq m(s)
\]
for $\nu_G$-almost every $x\in G$. Since the expression on the left is
jointly Borel in $(s,x)$, Fubini's theorem yields a conull set
$X_0\subset G$ such that, for every $x\in X_0$, this inequality holds
for $\nu_G$-almost every $s\in G$.

It follows that, for every $x\in X_0$,
\[
\begin{aligned}
|f_n(x)-\widetilde f(x)|
\leq
\int_G k_n(s)
|\widetilde f(s^{-1}x)-\widetilde f(x)|
\,d\nu_G(s)\leq
\int_G k_n(s)m(s)\,d\nu_G(s)\leq
\sup_{s\in U_n}\|sf-f\|_{\infty}.
\end{aligned}
\]
Taking the essential supremum over $x$ gives
\[
\|f_n-f\|_{\infty}
\leq
\sup_{s\in U_n}\|sf-f\|_{\infty}.
\]
Since $F_f$ is continuous at $e_G$, the right-hand side converges to
zero. Consequently,
\[
f_n\longrightarrow f
\qquad\text{in }L^\infty(G,\nu_G).
\]

The image of $C(G)$ in $L^\infty(G,\nu_G)$ is closed, since the
embedding is isometric and $C(G)$ is complete. As $f_n\in C(G)$ for
every $n$, it follows that $f$ belongs to this image.
\end{proof}

\begin{lemma}
\label{L-cont to cont homogeneous}
Let $G$ be a compact metrizable group and let $H\leq G$ be a closed
subgroup. Let $f\in L^\infty(G/H, \nu_{G/H})$.
Then the equivalence class of $f$ has a continuous representative if
and only if the orbit map
\[
F_f\colon G\to L^\infty(G/H,\nu_{G/H}),
\qquad
F_f(s)=sf,
\]
is continuous with respect to the $L^\infty$-norm topology.
\end{lemma}
\begin{proof} The ``only if'' part can be proved in the same way as the ``only if'' part of \cref{L-cont to cont group2} .

Denote by $\pi\colon G\to G/H$ the quotient map. Then
\[
\pi_*(\nu_G)=\nu_{G/H}.
\]
Let $\tilde f\colon G/H\to\Cb$ be a bounded Borel function representing
$f$. Assume that $F_f$ is continuous with respect to the
$L^\infty$-norm topology. Since the pullback
$\pi^*\colon
L^\infty(G/H,\nu_{G/H})
\to
L^\infty(G,\nu_G)$
is isometric and $G$-equivariant, the orbit map
$F_{\tilde f\circ\pi}$ is also continuous.

From
\cref{L-cont to cont group2} there exists $g\in C(G)$
such that
\[
\tilde f\circ\pi=g
\qquad \nu_G\text{-a.e.}
\]
For every $h\in H$, the function $\tilde f\circ\pi$ is invariant under
right translation by $h$. Consequently,
\[
g(xh)=g(x)
\qquad \nu_G\text{-a.e. }x\in G.
\]
Both sides are continuous functions of $x$, and $\nu_G$ has full
support. Therefore,
\[
g(xh)=g(x)
\qquad\text{for every }x\in G\text{ and }h\in H.
\]
Thus, $g$ is constant on the fibers of $\pi$ and induces a function
$\bar f\in C(G/H)$ satisfying
\[
g=\bar f\circ\pi.
\]
It follows that
\[
\tilde f=\bar f
\qquad \nu_{G/H}\text{-a.e.}
\]
Therefore, the equivalence class of $f$ has the continuous
representative $\bar f$.
\end{proof}
 Let \[
\cK_c
=
\left\{
f\in L^\infty(X,\mu)\cap\cK_\mu| \mbox{ } F_f(G_\mu) \subset L^{\infty}(X,\mu)\text{ and } 
F_f 
\text{ is $L^{\infty}$-norm continuous} 
\right\}.
\]

\begin{lemma}
\label{L-kronecker}
Let $\Gamma \curvearrowright (X,\mu)$ be an ergodic action with $\mu$ fully supported. Then
\[
\pi^*_\mu\bigl(C(Z_\mu)\bigr)=\cK_c.
\]
\end{lemma}

\begin{proof}
Recall that $\pi^*_\mu$ restricts to an equivariant isometric
isomorphism
\[
\pi^*_\mu\colon
L^\infty(Z_\mu,\nu_\mu)
\to
L^\infty(X,\Sigma_\mu,\mu)
=
L^\infty(X,\mu)\cap\cK_\mu.
\]
Since $\nu_\mu$ has full support, we identify $C(Z_\mu)$ with its image
in $L^\infty(Z_\mu,\nu_\mu)$.

As mentioned, we identify $Z_\mu$ with $G_\mu/H_\mu$. By
\cref{L-cont to cont homogeneous}, for every
$f\in L^\infty(Z_\mu,\nu_\mu)$, the orbit map $F_f\colon G_\mu\to L^\infty(Z_\mu,\nu_\mu)$ is continuous with respect to the $L^\infty$-norm topology if and only if $f\in C(Z_\mu)$.

Since $\pi^*_\mu$ is equivariant,
\[
F_{\pi^*_\mu(f)}(s)
=
\pi^*_\mu\bigl(F_f(s)\bigr)
\qquad
\text{for every }s\in G_\mu.
\]
As $\pi^*_\mu$ is isometric, $F_f$ is continuous if and only if
$F_{\pi^*_\mu(f)}$ is continuous. By the definition of $\cK_c$, this
is equivalent to
\[
\pi^*_\mu(f)\in\cK_c.
\]
Therefore,
\[
\pi^*_\mu\bigl(C(Z_\mu)\bigr)=\cK_c.\qedhere
\]
\end{proof}

\begin{lemma}
\label{L-equi in Kronecker}
    Let $\act$ be a continuous action and $\mu\in\cM_e(\act)$ fully supported. Then \[
{\pi_{\eq}}^*\bigl(C(X_{\eq})\bigr)\subset {\pi_{\mu}}^*\bigl(C(Z_\mu)\bigr)\subset L^\infty(X,\mu).
\] 

\end{lemma}
\begin{proof}
Since $Z_\mu$ is compact
every $f\in C(Z_\mu)$ is bounded and $f\circ\pi_\mu$ is measurable.
Therefore,
\[
\pi^*_\mu\bigl(C(Z_\mu)\bigr)\subset L^\infty(X,\mu).
\]

For every $h\in C(X)$, let $X_h$ be the factor of $X$ defined by the
equivalence relation
\[
x\sim_h y
\quad\Longleftrightarrow\quad
h(\gamma x)=h(\gamma y)
\qquad\text{for every }\gamma\in\Gamma.
\]
 Let
\[
\pi_h\colon X\to X_h
\]
be the factor map and let $\bar h\in C(X_h)$ be the function satisfying
\[
h=\bar h\circ\pi_h.
\]
By the definition of $\sim_h$, the family
\[
\{\gamma\bar h:\gamma\in\Gamma\}
\]
separates the points of $X_h$.

Given a compact metrizable space $Y$ and a continuous action $\Gamma\curvearrowright Y$, we say a function $f\in C(Y)$ is \emph{almost periodic} if the $\Gamma$-orbit of $f$, $
\{f\circ s:s\in\Gamma\}$,
is totally bounded in $C(Y)$.
Recall that $\Gamma\curvearrowright Y$ is equicontinuous if and only if its almost periodic
continuous functions separate points if and only if every element of $C(Y)$ is almost periodic
\cite[Lemma~1(3)]{Ellis1959Equicontinuity}. Also note that the $\Gamma$-orbit of $h$ is totally
bounded in $C(X)$ if and only if the $\Gamma$-orbit of $\bar h$ is totally
bounded in $C(X_h)$. Therefore, the $\Gamma$-orbit of $h$ is totally bounded if and
only if $X_h$ is equicontinuous. Using the maximality of $\pi_{\eq}\colon X\to X_{\eq}$ we conclude that
\[
h\in {\pi_{\eq}}^*\bigl(C(X_{\eq})\bigr)
\]
if and only if its $\Gamma$-orbit is totally bounded in the supremum
norm, or equivalently, has compact closure in $C(X)$.

Let $f\in {\pi_{\eq}}^*(C(X_{\eq}))$. Denote by
\[
\overline{\Gamma f}^{\|\cdot\|_\infty}
\quad\text{and}\quad
\overline{\Gamma f}^{\|\cdot\|_2}
\]
the orbit closures of $f$ in $C(X)$ and $L^2(X,\mu)$, respectively.
Then $\overline{\Gamma f}^{\|\cdot\|_\infty}$ is compact. Let
\[
\iota\colon C(X)\to L^2(X,\mu)
\]
be the natural embedding. Since $\iota$ is continuous, and is
injective, because $\mu$ has full support, we have
\[
\iota\left(\overline{\Gamma f}^{\|\cdot\|_\infty}\right)
=
\overline{\Gamma f}^{\|\cdot\|_2},
\]
and the restriction
\[
\iota\colon
\overline{\Gamma f}^{\|\cdot\|_\infty}
\to
\overline{\Gamma f}^{\|\cdot\|_2}
\]
is a homeomorphism. In particular,
$\overline{\Gamma f}^{\|\cdot\|_2}$ is compact, and therefore
$f\in\cK_{\mu}$.

Recall that, by \eqref{Gmu-operator-realization}, we identify
\[
G_\mu
=
\overline{\rho_\mu(\Gamma)}^{\,\mathrm{SOT}}
\leq
\mathcal U(\mathcal K_\mu),
\]
where

\[
\rho_\mu(\gamma)
=
U_\mu(\gamma)|_{\mathcal K_\mu}
\qquad
\text{for every }\gamma\in\Gamma.
\]
For $f\in\mathcal K_\mu$, the orbit map
\[
F_f\colon G_\mu\longrightarrow\mathcal K_\mu,
\qquad
F_f(g)=gf,
\]
is continuous with respect to the $L^2$-norm. Since
$\phi_\mu(\Gamma)$ is dense in $G_\mu$ and $G_\mu f$ is compact, it follows
that
\[
G_\mu f
=
\overline{\phi_\mu(\Gamma)f}^{\|\cdot\|_2}
=
\overline{\Gamma f}^{\|\cdot\|_2}.
\]

Using the homeomorphism above, the orbit map
\[
F_f\colon
G_\mu\to
\overline{\Gamma f}^{\|\cdot\|_\infty}
\subset C(X)
\]
is therefore continuous in the supremum norm for every $f\in {\pi_{\eq}}^*(C(X_{\eq}))$. Since the supremum norm
and the $L^\infty$-norm agree on $C(X)$, it follows that the orbit map
\[
F_f\colon G_\mu\to L^\infty(X,\mu)
\]
is continuous in the $L^\infty$-norm topology. Hence $f\in\cK_c$.

By \cref{L-kronecker}
\[
\cK_c={\pi_{\mu}}^*\bigl(C(Z_\mu)\bigr),
\]
and consequently
\[
{\pi_{\eq}}^*(C(X_{\eq}))
\subset
\pi^*_\mu\bigl(C(Z_\mu)\bigr).
\qedhere\]
\end{proof}

A \textbf{complex algebra} is a complex vector space $A$ equipped with an associative bilinear multiplication
$(a,b)\mapsto ab$. A \textbf{Banach algebra} is a complex algebra $A$ together with a norm $\|\cdot\|$ making
$A$ a Banach space and satisfying $\|ab\|\le \|a\|\,\|b\|$ for all $a,b\in A$. If $A$ has a multiplicative
identity element $1_A$, we call it \textbf{unital}. The space $C(X)$ equipped with pointwise addition and multiplication, is a commutative unital Banach algebra.

A linear map $\Phi\colon A\rightarrow B$ between unital complex
algebras is a \textbf{unital algebra homomorphism} if $\Phi(ab)=\Phi(a)\Phi(b)$ for all $a,b\in A$ and
$\Phi(1_A)=1_B$; it is a \textbf{unital algebra isomorphism} if, in addition, $\Phi$ is bijective. For a
commutative unital Banach algebra $A$, the \textbf{Gelfand spectrum} ${\rm Spec}(A)$ is the set of all
unital algebra homomorphisms $\chi\colon A\rightarrow \Cb$, equipped with the weak$^*$ (pointwise) topology. For a proof of \cref{thm:gelfand-CX}, see \cite[Chapter VII, Theorem 8.7]{ConwayFA}.
\begin{theorem}[Gelfand theorem for $C(X)$]
\label{thm:gelfand-CX}
For each $x\in X$, define the evaluation homomorphism
$\ev_x\colon C(X)\to\mathbb{C}$ by $\ev_x(f)=f(x)$.
Then every unital algebra homomorphism $\chi\colon C(X)\to\mathbb{C}$
is of the form $\chi=\ev_x$ for a unique $x\in X$.
Consequently, the map
\[
X\to \mbox{\rm Spec}({C(X)}),\qquad x\mapsto \ev_x
\]
is a homeomorphism.
\end{theorem}

 Let $\act$ be a continuous action, $\mu\in\cM_e(\act)$ fully supported and
\[
\mathcal A_{\eq}
=
{\pi_{\eq}}^*\bigl(C(X_{\eq})\bigr).
\]
The proof of \cref{L-equi in Kronecker} shows that $\mathcal A_{\eq}$ is
$G_\mu$-invariant and that $G_\mu$ acts continuously on it by isometric
unital algebra automorphisms. Since
\[
\mathcal A_{\eq}\cong C(X_{\eq}),
\]
the Gelfand spectrum of $\mathcal A_{\eq}$ is canonically identified with
$X_{\eq}$. Hence, by Gelfand duality, the action on $\mathcal A_{\eq}$
induces a continuous action
\[
G_\mu\curvearrowright X_{\eq}.
\]
Since
\[
\rho_\mu(\gamma)f=\gamma f
\qquad
\text{for every }\gamma\in\Gamma
\text{ and }f\in\mathcal A_{\eq},
\]
the \(G_\mu\)-action on \(X_{\eq}\) restricts, through
\(\phi_\mu\colon\Gamma\to G_\mu\), to the original \(\Gamma\)-action; that is,
\[
\phi_\mu(\gamma)x=\gamma x
\qquad
\text{for every }\gamma\in\Gamma
\text{ and }x\in X_{\eq}.
\]

\begin{proposition}
\label{theta}
Let $\act$ be a continuous action and let
$\mu\in\cM_e(\act)$ be fully supported. For the $G_\mu$-action on
$X_{\eq}$ defined above, there exists a continuous $G_\mu$-equivariant map
\[
\theta_\mu\colon Z_\mu\to X_{\eq}
\]
characterized by
\[
f\circ\pi_{\eq}
=
f\circ\theta_\mu\circ\pi_\mu
\quad \mu\text{-a.e.}
\qquad\text{for every }f\in C(X_{\eq}).
\]
\end{proposition}

\begin{proof}
By \cref{L-equi in Kronecker},
\[
{\pi_{\eq}}^*\bigl(C(X_{\eq})\bigr)
\subset
\pi^*_\mu\bigl(C(Z_\mu)\bigr)
\subset L^\infty(X,\mu).
\]
Since $\nu_\mu$ has full support, the pullback
$
\pi^*_\mu\colon C(Z_\mu)\to L^\infty(X,\mu)$
is injective and hence is a unital algebra isomorphism onto its range.
Therefore,
\[
\Phi
=
\bigl(\pi^*_\mu|_{C(Z_\mu)}\bigr)^{-1}
\circ {\pi_{\eq}}^*
\colon C(X_{\eq})\to C(Z_\mu)
\]
is a well-defined unital algebra homomorphism.

For each $z\in Z_\mu$, the map
\[
\operatorname{ev}_z\circ\Phi\colon C(X_{\eq})\to\Cb
\]
is a unital algebra homomorphism. By \cref{thm:gelfand-CX}, there exists a unique point
$\theta_\mu(z)\in X_{\eq}$ such that
\[
\Phi(f)(z)=f\bigl(\theta_\mu(z)\bigr)
\qquad\text{for every }f\in C(X_{\eq}).
\]
Consequently,
\[
f\circ\pi_{\eq}
=
\pi^*_\mu\bigl(\Phi(f)\bigr)
=
\Phi(f)\circ\pi_\mu
=
f\circ\theta_\mu\circ\pi_\mu
\quad\mu\text{-a.e.}
\]
for every $f\in C(X_{\eq})$.

Since
\[
f\circ\theta_\mu=\Phi(f)\in C(Z_\mu)
\qquad\text{for every }f\in C(X_{\eq}),
\]
the map $\theta_\mu$ is continuous. Moreover,
${\pi_{\eq}}^*$ and $\pi^*_\mu$ are $G_\mu$-equivariant, so $\Phi$ and $\theta_\mu$ are also $G_\mu$-equivariant.

The defining identity determines $\theta_\mu$ uniquely because
$\pi^*_\mu$ is injective and the functions in $C(X_{\eq})$ separate
points.
\end{proof}

\begin{lemma}
\label{cor:theta-measure}
Let $\act$ be a continuous action and let
$\mu\in\cM_e(\act)$ be fully supported. The map $\theta_\mu$ satisfies
\[
(\theta_\mu)_*\nu_\mu=\nu_{\eq}.
\]
Consequently, $\theta_\mu$ is surjective.
\end{lemma}

\begin{proof}
For every $f\in C(X_{\eq})$, we have
For every \(f\in C(X_{\eq})\),
\[
\int_{X_{\eq}} f\,d(\theta_\mu)_*\nu_\mu
=
\int_{Z_\mu}f\circ\theta_\mu\,d\nu_\mu
=
\int_X f\circ\theta_\mu\circ\pi_\mu\,d\mu
=
\int_X f\circ\pi_{\eq}\,d\mu
=
\int_{X_{\eq}}f\,d\nu_{\eq}.
\]
Therefore,
\[
(\theta_\mu)_*\nu_\mu=\nu_{\eq}.
\]

Since $\theta_\mu(Z_\mu)$ is compact, it is closed in $X_{\eq}$. Moreover,
\[
\nu_{\eq}\bigl(\theta_\mu(Z_\mu)\bigr)
=
(\theta_\mu)_*\nu_\mu\bigl(\theta_\mu(Z_\mu)\bigr)
=
1.
\]
Because $\nu_{\eq}$ has full support, no proper closed subset of
$X_{\eq}$ can have full measure. Hence
\[
\theta_\mu(Z_\mu)=X_{\eq},
\]
so $\theta_\mu$ is surjective.
\end{proof}

\begin{remark}
\label{rem:homoge}
Since $\theta_\mu$ is surjective and $G_\mu$-equivariant, the homogeneity of
$Z_\mu$ passes to $X_{\eq}$. Indeed, given $x,y\in X_{\eq}$, choose
$z,w\in Z_\mu$ such that
\[
\theta_\mu(z)=x
\qquad\text{and}\qquad
\theta_\mu(w)=y.
\]
There exists $g\in G_\mu$ such that $gz=w$, and hence
\[
gx
=
g\theta_\mu(z)
=
\theta_\mu(gz)
=
\theta_\mu(w)
=
y.
\]
\end{remark}

 \begin{lemma}\label{lem:constant}
    Let $\act$ be a continuous action and $\mu\in \cM_e(\act)$ fully supported. Let $X_{\mu}'$ be given by \cref{prop_kronecker}.  There exists a $\Gamma$-invariant Borel full-measure subset $X''_{\mu}\subset X'_\mu$, such that the diagram
\begin{eqnarray*} 
\begin{tikzcd}
             & X''_\mu \arrow[ld, "\pi_\mu"] \arrow[dd, "\pi_{\eq}"]\\
Z_\mu \arrow[dr, "\theta_\mu"] & \\
&X_{\eq}
\end{tikzcd}
\end{eqnarray*}
commutes. Furthermore $|\theta_\mu^{-1}(y)|$ is constant for all $y\in X_{\eq}$.
\end{lemma}
\begin{proof}

Let $\Omega$ be a countable dense subset of $C(X_{\eq})$. By
\cref{theta}, for every $g\in\Omega$ there exists a Borel set
$A_g\subset X'_\mu$ such that $\mu(A_g)=1$ and
\[
g\circ\pi_{\eq}
=
g\circ\theta_\mu\circ\pi_\mu
\qquad\text{on }A_g.
\]
Set
\[
X''_\mu
=
\bigcap_{\gamma\in\Gamma}
\bigcap_{g\in\Omega}
\gamma^{-1}A_g.
\]
Since $\Gamma$ and $\Omega$ are countable, $X''_\mu$ is a Borel set of
full measure. Moreover, it is $\Gamma$-invariant and
$X''_\mu\subset X'_\mu$. Since the identity belongs to $\Gamma$, we
have
\[
g\circ\pi_{\eq}
=
g\circ\theta_\mu\circ\pi_\mu
\qquad\text{on }X''_\mu
\]
for every $g\in\Omega$.

Since $\Omega$ is dense in $C(X_{\eq})$, it separates the points of
$X_{\eq}$. Therefore,
\[
\pi_{\eq}
=
\theta_\mu\circ\pi_\mu
\qquad\text{on }X''_\mu.
\]
Thus the diagram commutes.

It remains to prove that the cardinality of the fibers of $\theta_\mu$
is constant. Let $y,y'\in X_{\eq}$. By \cref{rem:homoge}, there exists
$t\in G_\mu$ such that
\[
ty=y'.
\]
Since $\theta_\mu$ is $G_\mu$-equivariant,
\[
\theta_\mu^{-1}(y')
=
\theta_\mu^{-1}(ty)
=
t\theta_\mu^{-1}(y).
\]
Since multiplication by $t$ is a bijection of $Z_\mu$, we conclude that
\[
\bigl|\theta_\mu^{-1}(y')\bigr|
=
\bigl|\theta_\mu^{-1}(y)\bigr|. \qedhere
\]\end{proof}

\subsection{Disintegration over the maximal equicontinuous factor} \label{S-multiplicity}

For $z\in Z_\mu$, denote by $\rStab_{G_\mu}(z)$ the \textbf{stabilizer} of $z$ under the action of $G_\mu$, i.e. 
$$\rStab_{G_\mu}(z)=\{s\in G_\mu: sz=z\}.$$

The following is the main result of the subsection.

\begin{proposition}
\label{L-RIM Kronecker}
Let $\act$ be a continuous action and $\mu\in \cM_e(\act)$ fully supported. For each $y\in X_{\eq}$, there is a unique $\rStab_{G_\mu}(y)$-invariant Borel probability measure $\nu_{\mu, y}$ on $\theta_\mu^{-1}(y)$. The map \[
X_{\eq}\to \mathcal M(Z_\mu),
\qquad
y\mapsto \nu_{\mu,y}
\]is continuous for the weak$^*$-topology on $\cM(Z_\mu)$. Furthermore, 

\begin{align} \label{E-set Kronecker}
 \nu_\mu(A)=\int_{X_{\eq}} \nu_{\mu, y}(A)\, d\nu_{\eq}(y)
 \end{align}
for each $A\in \sB_{Z_\mu}$. 
\end{proposition}  

We will prove \cref{L-RIM Kronecker} at the end of the subsection.

\begin{lemma}
Let $\act$ be a continuous action and let
$\mu\in\cM_e(\act)$ be fully supported. For every $z\in Z_\mu$, the orbit map
\[
\rStab_{G_\mu}\bigl(\theta_\mu(z)\bigr)
\longrightarrow
\theta_\mu^{-1}\bigl(\theta_\mu(z)\bigr),
\qquad
g\longmapsto gz,
\]
induces a homeomorphism
\[
\rStab_{G_\mu}\bigl(\theta_\mu(z)\bigr)
\big/
\rStab_{G_\mu}(z)
\cong
\theta_\mu^{-1}\bigl(\theta_\mu(z)\bigr).
\]
\end{lemma}

\begin{proof}
Fix $z\in Z_\mu$. By the $G_\mu$-equivariance of $\theta_\mu$,
\[
\rStab_{G_\mu}(z)
\subset
\rStab_{G_\mu}\bigl(\theta_\mu(z)\bigr).
\]
Moreover, if
$g\in\rStab_{G_\mu}\bigl(\theta_\mu(z)\bigr)$, then
\[
\theta_\mu(gz)=g\theta_\mu(z)=\theta_\mu(z),
\]
so the orbit map takes values in
$\theta_\mu^{-1}(\theta_\mu(z))$.

Conversely, if $w\in\theta_\mu^{-1}(\theta_\mu(z))$, the homogeneity of
$Z_\mu$ gives $g\in G_\mu$ such that $gz=w$. Then
\[
g\theta_\mu(z)
=
\theta_\mu(gz)
=
\theta_\mu(w)
=
\theta_\mu(z),
\]
and hence
$g\in\rStab_{G_\mu}(\theta_\mu(z))$. Thus, the orbit map is surjective.
Its fibers are precisely the cosets of $\rStab_{G_\mu}(z)$, so it induces
a continuous bijection
\[
\rStab_{G_\mu}\bigl(\theta_\mu(z)\bigr)
\big/
\rStab_{G_\mu}(z)
\longrightarrow
\theta_\mu^{-1}\bigl(\theta_\mu(z)\bigr).
\]
Since the domain is compact and the codomain is Hausdorff, this bijection
is a homeomorphism.
\end{proof}

For $G$ a compact group, $H_1\subset H_2 \leq G$ closed subgroups and
$y=gH_2\in G/H_2$, we define

\[
q\colon H_2\longrightarrow G/H_1,
\qquad
q(h)=ghH_1,
\]
and
\[
\nu_y
=
q_{*}\nu_{H_2}
\in\cM(G/H_1).
\]
This definition is independent of the representative $g$ of $y$. For
$f\in C(G/H_1)$, set
\[
E_\theta f(y)
=
\int_{G/H_1} f\,d\nu_y.
\]

\begin{lemma} \label{L-Borel}
Let $X$ be a compact metrizable space. For every  $A\in \sB_X$, the function $f_A: \cM(X)\rightarrow [0, 1]$ sending $\mu$ to $\mu(A)$ is Borel. 
\end{lemma}
\begin{proof} Denote by ${\mathscr A}$ the set of $A\in \sB_X$ satisfying that $f_A$ is Borel. 

Let $U$ be an open subset of $X$. We can find a sequence $\{f_n\}_{n\in \Nb}$ in $C(X)$ such that $0\le f_n\le 1$ and $f_n\le f_{n+1}$ for all $n\in \Nb$ and $f_n\rightarrow 1_U$ pointwisely on $X$ as $n\to \infty$. Then the function $\varphi_n: \cM(X)\rightarrow \Rb$ sending $\mu$ to $\int_X f_n\, d\mu$ is continuous for each $n\in \Nb$. Note that $\varphi_n\rightarrow f_U$ pointwisely on $\cM(X)$ as $n\to \infty$. Thus $f_U$ is Borel, and hence $U\in {\mathscr A}$. 

Denote by ${\mathscr O}$ the family of open subsets of $X$. Then ${\mathscr O}\subset {\mathscr A}$. Note that  $\sB_X$ is the $\sigma$-algebra generated by ${\mathscr O}$. Since ${\mathscr O}$ is closed under taking finite intersections and contains $X$, $\sB_X$ is the smallest collection of subsets of $X$ containing ${\mathscr O}$ and closed under taking complements and countable disjoint unions \cite[Theorem 10.1.iii]{Kechris1995}. Clearly ${\mathscr A}$ is closed under taking complements and countable disjoint unions. 
Thus $\sB_X\subset {\mathscr A} \subset \sB_X$, and hence $\sB_X={\mathscr A}$. 
\end{proof}

\begin{lemma}
\label{L-RIM}
Let $G$ be a compact metrizable group, let $H_1\subset H_2$ be closed
subgroups, and let
\[
\theta\colon G/H_1\longrightarrow G/H_2
\]
be the natural quotient map. For every $y\in G/H_2$, the measure $\nu_y$, as defined above, is the unique
$\rStab_G(y)$-invariant Borel probability measure on $G/H_1$ supported on $\theta^{-1}(y)$.
Moreover, the map
\[
G/H_2\longrightarrow\cM(G/H_1),
\qquad
y\longmapsto\nu_y,
\]
is weak$^*$ continuous, and
 for every $A\in \cB_{G/H_1}$,
\[
\nu_{G/H_1}(A)
=
\int_{G/H_2}\nu_y(A)\,d\nu_{G/H_2}(y).
\]

\end{lemma}

\begin{proof}
Let $y=gH_2\in G/H_2$. Then
\[
\rStab_G(y)=gH_2g^{-1}
\qquad\text{and}\qquad
\theta^{-1}(y)=\{ghH_1:h\in H_2\}.
\]
Thus $\rStab_G(y)$ acts continuously on $\theta^{-1}(y)$ and is transitive on points; indeed,
if $gh_1H_1,gh_2H_1\in\theta^{-1}(y)$, then
\[
gh_2h_1^{-1}g^{-1}\in\rStab_G(y)
\]
and
\[
\bigl(gh_2h_1^{-1}g^{-1}\bigr)gh_1H_1=gh_2H_1.
\]

If $x\in\theta^{-1}(y)$, the orbit map
\[
\rStab_G(y)\to\theta^{-1}(y),
\qquad
g\mapsto gx,
\]
induces a homeomorphism
\[
\rStab_G(y)/\rStab_G(x)
\cong
\theta^{-1}(y).
\]
Consequently, there is a unique $\rStab_G(y)$-invariant Borel probability
measure supported on $\theta^{-1}(y)$, this measure does not depend on the choice of $x$ and must coincide with  $\nu_y$. 

We next show that the map
\[
G/H_2\to\cM(G/H_1),
\qquad
y\mapsto\nu_y,
\]
is weak$^*$ continuous. Let $f\in C(G/H_1)$. Then the orbit map
$F_f\colon G\to C(G/H_1)$, is continuous in the supremum norm.

Let $y_0=H_2\in G/H_2$. For every $s\in G$, the measure $s\nu_{y_0}$ is
supported on
\[
s\theta^{-1}(y_0)=\theta^{-1}(sy_0)
\]
and is invariant under the action of the subgroup
\[
s\rStab_G(y_0)s^{-1}=\rStab_G(sy_0).
\]
By uniqueness shown above,
\[
s\nu_{y_0}=\nu_{sy_0}.
\]
Hence, for every $f\in C(G/H_1)$,
\[
\int_{G/H_1}f\,d\nu_{sy_0}
=
\int_{G/H_1}f\,ds\nu_{y_0}
=
\int_{G/H_1}s^{-1}f\,d\nu_{y_0}.
\]
The right-hand side depends continuously on $s$. Thus, the map
\[
G\to \cM(G/H_1),
\qquad
s\mapsto\nu_{sy_0},
\]
is weak$^*$ continuous. Since this map factors through the quotient map
$G\to G/H_2$, the map 
\[
G/H_2\to \cM(G/H_1),
\qquad
y\mapsto\nu_{y},
\]
is weak$^*$ continuous.

 Let $A\in \sB_{G/H_1}$. By Lemma~\ref{L-Borel} the function $\cM(G/H_1)\rightarrow [0, 1]$ sending $\nu$ to $\nu(A)$ is Borel, whence the function $G/H_2\rightarrow [0, 1]$ sending $y$ to $\nu_y(A)$ is Borel. Thus we may define
 $$\omega(A)=\int_{G/H_2} \nu_y(A)\, d\nu_{G/H_2}(y)\in [0, 1].$$
 Clearly $\omega$ is a Borel probability measure on $G/H_1$. 
 
 Let $s\in G$. For each $y\in G/H_2$, as we have seen above, one has $s\nu_y=\nu_{sy}$. For each $A\in \sB_{G/H_1}$, we have
 \begin{align*}
 s\omega(A)&=\int_{G/H_2} \nu_y(s^{-1}A)\, d\nu_{G/H_2}(y)=\int_{G/H_2} (s\nu_y)(A)\, d\nu_{G/H_2}(y)=\int_{G/H_2} \nu_{sy}(A)\, d\nu_{G/H_2}(y)\\
 &=\int_{G/H_2} \nu_y(A)\, d(s\nu_{G/H_2})(y)=\int_{G/H_2} \nu_y(A)\, d\nu_{G/H_2}(y)=\omega(A).
 \end{align*}
 Thus $\omega$ is $G$-invariant. Therefore $\omega=\nu_{G/H_1}$.

\end{proof}

\begin{proof}[Proof of \cref{L-RIM Kronecker}]
By \cref{rm:fullys}, identify
\[
(Z_\mu,\nu_\mu)
=
(G_\mu/H_\mu,\nu_{G_\mu/H_\mu}),
\]
and let
\[
z_0=H_\mu,\qquad
y_0=\theta_\mu(z_0),\qquad
H_2=\rStab_{G_\mu}(y_0).
\]
Since $\theta_\mu$ is $G_\mu$-equivariant by \cref{theta}, we have
$H_\mu\subset H_2$. Moreover, $\theta_\mu$ is surjective by
\cref{cor:theta-measure}; hence, by \cref{rem:homoge}, the orbit map induces
a $G_\mu$-equivariant homeomorphism
\[
G_\mu/H_2\cong X_{\eq}.
\]
Under these identifications, $\theta_\mu$ corresponds to the natural quotient
map
\[
G_\mu/H_\mu\longrightarrow G_\mu/H_2.
\]
The canonical invariant measures correspond to $\nu_\mu$ and $\nu_{\eq}$,
respectively, by \cref{rm:fullys,cor:theta-measure}. The conclusion now follows
directly from \cref{L-RIM}.
\end{proof}

\subsection{Applications to the IT-rank} \label{S-application to IT}

Using \Cref{lem:constant} we have that $|\theta_\mu^{-1}(y)|$ is independent of $y\in X_{\eq}$. Hence the following is well defined. 
\begin{definition} \label{D-multiplicity}
Let $\act$ be a continuous action and let
$\mu\in\cM_e(\act)$ fully supported. We define the {\it multiplicity} of $\mu$ as 
\begin{align*}
\mult(\mu)= 
\left\{ 
           \begin{array}{ll}
                    |\theta_\mu^{-1}(y)| & \mbox{if } \theta_\mu^{-1}(y) \mbox{ is finite }  \\
                   \infty         & \mbox{otherwise } 
           \end{array}        
\right.                 
\end{align*}
for  $y\in X_{\eq}$.
\end{definition} 

\begin{lemma} \label{L-finite mult}
Let $\Gamma\curvearrowright X$ be a minimal continuous action. If $R_{\IT}(X,\Gamma)<\infty$, then $\mult(\mu)<\infty$ for every $\mu\in \cM_e(\act)$. 
\end{lemma}

\begin{proof}
If $\cM_e(\act)= \emptyset$ there is nothing to prove. Assume that $\cM_e(\act)\neq \emptyset$.
By \cref{q-it}, there are an integer $1\leq k<\infty$ and a Borel subset
$Y\subset X_{\eq}$ with $\nu_{\eq}(Y)=1$ such that
\[
|\pi_{\eq}^{-1}(y)|=k
\]
for every $y\in Y$.

Let $\mu\in\cM_e(\act)$. Since the action is minimal, $\mu$ is fully
supported. For every $y\in X_{\eq}$, the fiber
$\theta_\mu^{-1}(y)$ is a homogeneous $\rStab_{G_\mu}(y)$-space, and $\nu_{\mu,y}$ is its unique
$\rStab_{G_\mu}(y)$-invariant Borel probability measure. Consequently,
if $\mult(\mu)=\infty$, then every fiber $\theta_\mu^{-1}(y)$ is infinite
and every measure $\nu_{\mu,y}$ is non-atomic.

Let $X''_\mu$ be the conull Borel subset of $X$ on which
\[
\pi_{\eq}=\theta_\mu\circ\pi_\mu.
\]
By Lusin's theorem and inner regularity, there is a compact set
$A\subset X''_\mu$ such that $\mu(A)>1/2$ and
$\pi_\mu|_A$ is continuous. Hence $\pi_\mu(A)$ is compact and
\[
\nu_\mu(\pi_\mu(A))
=
\mu\bigl(\pi_\mu^{-1}(\pi_\mu(A))\bigr)
\geq\mu(A)>\frac12.
\]
By \eqref{E-set Kronecker} in \cref{{L-RIM Kronecker}},
\[
\frac12
<
\nu_\mu(\pi_\mu(A))
=
\int_{X_{\eq}}
\nu_{\mu,y}(\pi_\mu(A))\,d\nu_{\eq}(y)
=
\int_Y
\nu_{\mu,y}
\bigl(\pi_\mu(A)\cap\theta_\mu^{-1}(y)\bigr)
\,d\nu_{\eq}(y).
\]
Therefore, there is some $y\in Y$ such that
\[
\nu_{\mu,y}
\bigl(\pi_\mu(A)\cap\theta_\mu^{-1}(y)\bigr)>0.
\]

Set
\[
B=\pi_\mu(A)\cap\theta_\mu^{-1}(y).
\]
For every $z\in B$, choose $x\in A$ such that $\pi_\mu(x)=z$. Since
$A\subset X''_\mu$,
\[
\pi_{\eq}(x)
=
\theta_\mu(\pi_\mu(x))
=
\theta_\mu(z)
=
y.
\]
Thus $\pi_\mu$ maps $A\cap\pi_{\eq}^{-1}(y)$ onto $B$, and hence
\[
|B|
\leq
|A\cap\pi_{\eq}^{-1}(y)|
\leq
|\pi_{\eq}^{-1}(y)|
=
k.
\]
On the other hand, if $\mult(\mu)=\infty$, then $\nu_{\mu,y}$ is
non-atomic. Since $\nu_{\mu,y}(B)>0$, the set $B$ must be infinite,
contradicting $|B|\leq k$. Therefore,
\[
\mult(\mu)<\infty. \qedhere
\]\end{proof}

Recall that $K(X)$ denotes the space of nonempty compact subsets of $X$,
equipped with the Vietoris topology (see Section 5.3). Let $Z$ be a topological space. A map
\[
\psi\colon Z\longrightarrow X
\]
is \textbf{lower semicontinuous} if, for every open set $V\subset X$, the set
\[
\{z\in Z:\psi(z)\cap V\neq\emptyset\}
\]
is open in $Z$.

\begin{lemma}
\label{L-lower-cont-to-Borel}
Let $X$ and $Z$ be compact metrizable spaces. Then every lower
semicontinuous map
\[
\psi\colon Z\longrightarrow K(X)
\]
is Borel.
\end{lemma}

\begin{proof}
A subbasis of the Vietoris topology of $K(X)$ consists of
the sets
\[
U^-=\{K\in K(X):K\cap U\neq\emptyset\},
\qquad
U^+=\{K\in K(X):K\subset U\},
\]
where $U\subset X$ is open \cite[Exercise~4.21]{Kechris1995}. Since $K(X)$ is compact and metrizable,
each open set in $K(X)$ is a countable union of sets of the form
\[
U_0^+\cap U_1^-\cap\cdots\cap U_n^-,
\]
with $n\in\mathbb N$ and $U_0,U_1,\ldots,U_n$ open sets in $X$.

By lower semicontinuity,
\[
\psi^{-1}(U^-)
=
\{z\in Z:\psi(z)\cap U\neq\emptyset\}
\]
is open.

It remains to consider $U^+$. Let $\rho$ be a compatible metric of $X$. Put $C=X\setminus U$ and, for $n\geq1$,
\[
C_n=\{x\in X:\rho(x,C)<1/n\}.
\]
Since the values of $\psi$ are compact,
\[
\psi(z)\cap C\neq\varnothing
\quad\Longleftrightarrow\quad
\psi(z)\cap C_n\neq\emptyset
\quad\text{for every }n.
\]
Consequently,
\[
Z\setminus\psi^{-1}(U^+)
=
\bigcap_{n\geq1}
\{z\in Z:\psi(z)\cap C_n\neq\emptyset\}.
\]
Each set on the right of the equation is open by lower semicontinuity. Thus
$\psi^{-1}(U^+)$ is Borel. We conclude that $\psi$ is Borel.
\end{proof}

Using the Portmanteau theorem
\cite[Theorem~2.1]{Billingsley99}, we obtain the following result. 

\begin{lemma}
\label{L-support lower cont}
The map
\[
\psi\colon\cM(X)\longrightarrow K(X),
\qquad
\mu\longmapsto\supp(\mu),
\]
is lower semicontinuous and hence Borel.
\end{lemma}

We will use the following well-known fact about hyperspaces; see
\cite[Lemma~1.11.6]{vanMill2001}.

\begin{lemma}
\label{L-union}
Let $X$ be a compact metrizable space, and let $Z$ be a compact subset of
$K(X)$. Then
\[
\bigcup Z
\]
is compact.
\end{lemma}

For a proof of the following result, see \cite[Theorem~5.14]{EinsiedlerWard2011}.
\begin{theorem}[Disintegration theorem]
Let $\act$ be a continuous action and $\mu\in\cM_e(\act)$. There exists a
$\Gamma$-invariant conull Borel set $Z_\mu'\subset Z_\mu$ and a Borel map
\[
Z_\mu'\longrightarrow\cM(X),
\qquad
z\longmapsto\mu_z,
\]
such that
\[
\mu=\int_{Z_\mu}\mu_z\,d\nu_\mu(z),
\]
with the following properties:
\begin{enumerate}
\item $\nu_\mu(Z_\mu')=1$;
\item $\mu_z\bigl(\pi_\mu^{-1}(z)\bigr)=1$ for every $z\in Z_\mu'$;
\item $\mu_{sz}=s\mu_z$ for every $z\in Z_\mu'$ and $s\in\Gamma$.
\end{enumerate}
We refer to $(\mu_z)_{z\in Z_\mu'}$ as the disintegration of $\mu$ over
$\nu_\mu$ along $\pi_\mu$.
\end{theorem}

\begin{lemma} \label{L-INL for ergodic}
Let  $\act$ be a continuous action and $\mu\in \cM_e(\act)$.  Then 
$$e^{h_\mu^*(\act)}=R_{\SE}(X,\Gamma,\mu)=|\supp(\mu_z)|$$ 
for $\nu_\mu$-a.e. $z\in Z_\mu'$. 
\end{lemma}
\begin{proof}
    
The equality $e^{h_\mu^*(\act)}=|\supp(\mu_z)|$
for $\nu_\mu$-a.e. $z\in Z_\mu'$ is given by \cite[Theorem 2.9]{liu2025sequenceIT}. The other equality is \cref{thm:seqformula}.
\end{proof}

\begin{theorem}
\label{T-length}
Let $\Gamma\curvearrowright X$ be a continuous minimal action on a compact
metrizable space $X$. Then
\[
R_{\IT}(X,\Gamma)
\geq
\sum_{\mu\in\cM_e(\act)}
\mult(\mu)R_{\SE}(X,\Gamma,\mu).
\]

\end{theorem}

\begin{proof}
We may assume that $R_{\IT}(X,\Gamma)<\infty$ and that $X$ has a
$\Gamma$-invariant Borel probability measure. By \cref{q-it}, there exists a
Borel set $Y\subset X_{\eq}$ such that $\nu_{\eq}(Y)=1$ and
\[
|\pi_{\eq}^{-1}(y)|\leq R_{\IT}(X,\Gamma)
\]
for every $y\in Y$.

Let $\mu\in\cM_e(\act)$. By \cref{lem:constant} we assume that $\pi_{\mu}$ is defined on $X''_{\mu}$. Using \cref{L-finite mult}, we have
$\mult(\mu)<\infty$. By \cref{cor:theta-measure}
\[
(\theta_\mu)_*\nu_\mu=\nu_{\eq},
\]
thus the set $\theta_\mu^{-1}(Y)$ has full $\nu_\mu$-measure. Therefore, by
\cref{L-INL for ergodic}, we can find a Borel set
\[
Z_\mu''
\subset
Z_\mu'\cap\theta_\mu^{-1}(Y)
\]
such that $\nu_\mu(Z_\mu'')=1$ and
\[
|\supp(\mu_z)|=R_{\SE}(X,\Gamma,\mu)
\]
for every $z\in Z_\mu''$.

For every $z\in Z_\mu''$, since $\theta_{\mu}\circ \pi_{\mu}=\pi_{eq}$ on $X''_{\mu}$, then
\[
\pi_\mu^{-1}(z)
\subset
\pi_{\eq}^{-1}(\theta_\mu(z)),
\]
and hence
\[
|\pi_\mu^{-1}(z)|
\leq
|\pi_{\eq}^{-1}(\theta_\mu(z))|
\leq
R_{\IT}(X,\Gamma).
\]
Since $\mu_z(\pi_\mu^{-1}(z))=1$ and $\pi_\mu^{-1}(z)$ is finite, it follows
that
\[
\supp(\mu_z)\subset\pi_\mu^{-1}(z).
\]

We shall find Borel sets
\[
Y_\mu\subset X_{\eq},
\qquad
Z_\mu^\sharp\subset Z_\mu,
\qquad
A_\mu\subset X
\]
satisfying the following conditions:
\begin{enumerate}
\item $Y_\mu\subset Y$, $Z_\mu^\sharp\subset Z_\mu''$, and
      $A_\mu\subset X_\mu''$;
\item $\nu_{\eq}(Y_\mu)=\nu_\mu(Z_\mu^\sharp)=1$;
\item $Z_\mu^\sharp=\theta_\mu^{-1}(Y_\mu)$
 and
$A_\mu=\bigcup_{z\in Z_\mu^\sharp}\supp(\mu_z).$

\end{enumerate}

By \cref{L-support lower cont}, the map
\[
\cM(X)\longrightarrow K(X),
\qquad
\omega\longmapsto\supp(\omega),
\]
is Borel. Since the map $z\mapsto\mu_z$ is Borel, it follows that
\[
Z_\mu''\longrightarrow K(X),
\qquad
z\longmapsto\supp(\mu_z),
\]
is Borel.

For each $n\geq 1$, by Lusin's theorem we can find a compact set
$Z_{\mu,n}^\dagger\subset Z_\mu''$ such that
\[
\nu_\mu(Z_{\mu,n}^\dagger)>1-\frac1n
\]
and the map
\[
Z_{\mu,n}^\dagger\longrightarrow K(X),
\qquad
z\longmapsto\supp(\mu_z),
\]
is continuous. Its image is a compact subset of $K(X)$. Hence, by
\cref{L-union}, the set
\[
A_{\mu,n}^\dagger
=
\bigcup_{z\in Z_{\mu,n}^\dagger}\supp(\mu_z)
\]
is compact. Moreover,
\[
\mu(A_{\mu,n}^\dagger)
=
\int_{Z_\mu'}
\mu_z(A_{\mu,n}^\dagger)\,d\nu_\mu(z)
\geq
\int_{Z_{\mu,n}^\dagger}
\mu_z(A_{\mu,n}^\dagger)\,d\nu_\mu(z)
=
\nu_\mu(Z_{\mu,n}^\dagger)
>
1-\frac1n.
\]

Set
\[
A_\mu^\dagger
=
\bigcup_{n\geq1}A_{\mu,n}^\dagger
\quad\text{and}\quad
Z_\mu^\dagger
=
\bigcup_{n\geq1}Z_{\mu,n}^\dagger.
\]
Then $A_\mu^\dagger$ and $Z_\mu^\dagger$ are $F_\sigma$ subsets of $X$ and
$Z_\mu$, respectively,
\[
\mu(A_\mu^\dagger)=1
\quad\text{and}\quad
\nu_\mu(Z_\mu^\dagger)=1,
\]
and
\[
A_\mu^\dagger
=
\bigcup_{z\in Z_\mu^\dagger}\supp(\mu_z)
\subset
\pi_\mu^{-1}(Z_\mu^\dagger)
\subset
X_\mu''.
\]

By \cref{L-RIM Kronecker}, the map
\[
X_{\eq}\longrightarrow\cM(Z_\mu),
\qquad
y\longmapsto\nu_{\mu,y},
\]
is continuous, and since $Z_\mu^\dagger$ is Borel, the function
\[X_{\eq}\longrightarrow [0,1], \qquad
y\longmapsto\nu_{\mu,y}(Z_\mu^\dagger)
\]
is Borel, and
\[
\int_{X_{\eq}}
\nu_{\mu,y}(Z_\mu^\dagger)\,d\nu_{\eq}(y)
=
\nu_\mu(Z_\mu^\dagger)
=
1.
\]
 As the function $y\longmapsto\nu_{\mu,y}(Z_\mu^\dagger)$ takes values in $[0,1]$ and has integral $1$, it is equal to
$1$ for $\nu_{\eq}$-almost every $y$. Set
\[
Y_\mu
=
\{y\in Y:\nu_{\mu,y}(Z_\mu^\dagger)=1\}.
\]
Then $Y_\mu$ is Borel, $Y_\mu\subset Y$, and
\[
\nu_{\eq}(Y_\mu)=1.
\]

By \cref{L-RIM Kronecker}, for every $y\in X_{\eq}$ the fiber
$\theta_\mu^{-1}(y)$ is a homogeneous space of
$\rStab_{G_\mu}(y)$, and $\nu_{\mu,y}$ is the unique
$\rStab_{G_\mu}(y)$-invariant Borel probability measure on this fiber.
Moreover,
\[
|\theta_\mu^{-1}(y)|=\mult(\mu)<\infty.
\]
Thus $\nu_{\mu,y}$ is the uniform probability measure on
$\theta_\mu^{-1}(y)$. In particular,
\[
\nu_{\mu,y}(\{z\})>0
\]
for every $z\in\theta_\mu^{-1}(y)$. Consequently, $Y_\mu$ is precisely the
set of points $y\in Y$ such that
\[
\theta_\mu^{-1}(y)\subset Z_\mu^\dagger.
\]

Set
\[
Z_\mu^\sharp=\theta_\mu^{-1}(Y_\mu).
\]
Then
\[
Z_\mu^\sharp\subset Z_\mu^\dagger\subset Z_\mu'',
\]
and $Z_\mu^\sharp$ is Borel. Since
$(\theta_\mu)_*\nu_\mu=\nu_{\eq}$, we have
\[
\nu_\mu(Z_\mu^\sharp)=\nu_{\eq}(Y_\mu)=1.
\]

Let $\psi_\mu$ be the restriction of $\pi_\mu$ to $A_\mu^\dagger$. Then
$\psi_\mu$ is Borel. We define the Borel set
\[
A_\mu=\psi_\mu^{-1}(Z_\mu^\sharp).
\]
 Since $Z_\mu^\sharp\subset Z_\mu^\dagger$ and the supports
$\supp(\mu_z)$ lie in the corresponding fibers of $\pi_\mu$, we have
\[
A_\mu
=
\bigcup_{z\in Z_\mu^\sharp}\supp(\mu_z).
\]
Thus $Y_\mu$, $Z_\mu^\sharp$, and $A_\mu$ satisfy the required conditions.

Now let $\mu_1,\mu_2\in\cM_e(\act)$ be distinct. We claim that
\[
\pi_{\eq}(A_{\mu_1}\cap A_{\mu_2})
\]
is Borel and
\[
\nu_{\eq}\bigl(\pi_{\eq}(A_{\mu_1}\cap A_{\mu_2})\bigr)=0.
\]
Since distinct ergodic invariant measures are mutually singular, we can find
a Borel partition $\{B_1,B_2\}$ of $X$ such that
\[
\mu_i(B_i)=1
\]
for $i=1,2$.

Set
\[
E_2=A_{\mu_1}\cap A_{\mu_2}\cap B_2.
\]
Since $\mu_1(B_2)=0$, we have
\begin{align*}
0
=\mu_1(E_2)=
\int_{Z_{\mu_1}'}
\mu_{1,z}(E_2)\,d\nu_{\mu_1}(z)=
\int_{Z_{\mu_1}^\sharp}
\mu_{1,z}(E_2)\,d\nu_{\mu_1}(z).
\end{align*}
The last equality follows from
$\nu_{\mu_1}(Z_{\mu_1}^\sharp)=1$.

Set
\[
Z_{\mu_1,\mu_2}
=
\{z\in Z_{\mu_1}^\sharp:\mu_{1,z}(E_2)>0\}.
\]
Then $Z_{\mu_1,\mu_2}$ is Borel and
\[
\nu_{\mu_1}(Z_{\mu_1,\mu_2})=0.
\]
For every $z\in Z_{\mu_1}^\sharp$, the measure $\mu_{1,z}$ has finite
support, and hence every point of $\supp(\mu_{1,z})$ has positive
$\mu_{1,z}$-measure. Since
\[
A_{\mu_1}
=
\bigcup_{z\in Z_{\mu_1}^\sharp}\supp(\mu_{1,z})
\]
and
\[
\supp(\mu_{1,z})\subset\pi_{\mu_1}^{-1}(z),
\]
we have
\[
\mu_{1,z}(E_2)>0
\quad\Longleftrightarrow\quad
z\in\pi_{\mu_1}(E_2).
\]
Therefore,
\[
Z_{\mu_1,\mu_2}=\pi_{\mu_1}(E_2).
\]

By \cref{L-RIM Kronecker},
\[
0
=
\nu_{\mu_1}(Z_{\mu_1,\mu_2})
=
\int_{X_{\eq}}
\nu_{\mu_1,y}(Z_{\mu_1,\mu_2})\,d\nu_{\eq}(y).
\]
Set
\[
Y_{\mu_1,\mu_2}
=
\{y\in X_{\eq}:
\nu_{\mu_1,y}(Z_{\mu_1,\mu_2})>0\}.
\]
Then $Y_{\mu_1,\mu_2}$ is Borel and
\[
\nu_{\eq}(Y_{\mu_1,\mu_2})=0.
\]
Since $\nu_{\mu_1,y}$ has full support on the finite fiber
$\theta_{\mu_1}^{-1}(y)$, we have
\[
Y_{\mu_1,\mu_2}
=
\theta_{\mu_1}(Z_{\mu_1,\mu_2}).
\]
Since $E_2\subset A_{\mu_1}\subset X_{\mu_1}''$, the factor identity gives
\begin{align*}
\pi_{\eq}(E_2)
=
\theta_{\mu_1}(\pi_{\mu_1}(E_2))=
\theta_{\mu_1}(Z_{\mu_1,\mu_2})=
Y_{\mu_1,\mu_2}.
\end{align*}
Therefore, $\pi_{\eq}(E_2)$ is Borel and
\[
\nu_{\eq}(\pi_{\eq}(E_2))=0.
\]

Interchanging the roles of $\mu_1$ and $\mu_2$, and using
$\mu_2(B_1)=0$, we similarly obtain that
\[
\pi_{\eq}(A_{\mu_1}\cap A_{\mu_2}\cap B_1)
\]
is Borel and has $\nu_{\eq}$-measure zero. Since $\{B_1,B_2\}$ is a
partition of $X$,
\begin{align*}
\pi_{\eq}(A_{\mu_1}\cap A_{\mu_2})
=
\pi_{\eq}(A_{\mu_1}\cap A_{\mu_2}\cap B_1)\cup
\pi_{\eq}(A_{\mu_1}\cap A_{\mu_2}\cap B_2).
\end{align*}
Consequently, $\pi_{\eq}(A_{\mu_1}\cap A_{\mu_2})$ is Borel and
\[
\nu_{\eq}\bigl(\pi_{\eq}(A_{\mu_1}\cap A_{\mu_2})\bigr)=0.
\]

Let
\[
\mu_1,\ldots,\mu_n\in\cM_e(\act)
\]
be distinct, and set
\[
Y_0
=
\bigcup_{1\leq i<j\leq n}
\pi_{\eq}(A_{\mu_i}\cap A_{\mu_j}).
\]
The preceding argument shows that $Y_0$ is Borel and
\[
\nu_{\eq}(Y_0)=0.
\]
Set
\[
Y^*
=
\left(\bigcap_{j=1}^nY_{\mu_j}\right)\setminus Y_0.
\]
Then $\nu_{\eq}(Y^*)=1$, so $Y^*$ is nonempty. Fix $y\in Y^*$.

For each $1\leq j\leq n$, we have
\[
\theta_{\mu_j}^{-1}(y)\subset Z_{\mu_j}^\sharp
\]
and
\[
\bigcup_{z\in\theta_{\mu_j}^{-1}(y)}
\supp(\mu_{j,z})
\subset
\pi_{\mu_j}^{-1}(\theta_{\mu_j}^{-1}(y))
\subset
\pi_{\eq}^{-1}(y).
\]
Since $y\notin Y_0$, the sets
\[
\bigcup_{z\in\theta_{\mu_j}^{-1}(y)}
\supp(\mu_{j,z}),
\qquad 1\leq j\leq n,
\]
are pairwise disjoint. Moreover, for each fixed $j$, the sets
$\supp(\mu_{j,z})$, with $z\in\theta_{\mu_j}^{-1}(y)$, are pairwise
disjoint because they lie in distinct fibers of $\pi_{\mu_j}$.

Therefore,
\begin{align*}
R_{\IT}(X,\Gamma)
&\geq
|\pi_{\eq}^{-1}(y)|\geq
\left|
\bigcup_{j=1}^n
\bigcup_{z\in\theta_{\mu_j}^{-1}(y)}
\supp(\mu_{j,z})
\right|=
\sum_{j=1}^n
\sum_{z\in\theta_{\mu_j}^{-1}(y)}
|\supp(\mu_{j,z})|\\
&=
\sum_{j=1}^n
\sum_{z\in\theta_{\mu_j}^{-1}(y)}
R_{\SE}(X,\Gamma,\mu_j)=
\sum_{j=1}^n
\mult(\mu_j)R_{\SE}(X,\Gamma,\mu_j).
\end{align*}
Since the finite collection $\mu_1,\ldots,\mu_n$ was arbitrary, taking the
supremum over all finite subsums proves the theorem.
\end{proof}

We have the following consequence of \cref{T-length}.

\begin{corollary} \label{C-number of ergodic}
Let $\Gamma\curvearrowright X$ be a continuous minimal action of $\Gamma$ on a compact metrizable space $X$. Then $|\cM_e(\act)|\le R_{\IT}(X,\Gamma)$. 
\end{corollary}

\subsection{Multiplicity and eigenvalues for actions of abelian groups}
\label{S-spectral-index}

Let $\Gamma$ be abelian, $\act$ a continuous action, $\mathbb T$ the unit circle group
and $\mu \in \cM_e(\act)$ fully supported. 
For a compact abelian group $\mathcal K$, write 
\[
\widehat{\mathcal K}=\operatorname{Hom}_{\mathrm{cont}}(\mathcal K,\mathbb T)
\]
for its Pontryagin dual, that is, the group of continuous homomorphisms $\mathcal{K} \to \mathbb T$ (see Section 5.1). We also write $[L:H]$ for the index of a subgroup
$H$ in a group $L$.

Since $\Gamma$ is abelian and $\phi_\mu(\Gamma)$ is
dense in $G_\mu$, the compact group $G_\mu$ is abelian. Moreover, from \cref{mackey}, one has
\[
\{e_{G_\mu}\}=\bigcap_{g\in G_\mu}gH_\mu g^{-1}=H_\mu.
\]

 Thus, we identify the Kronecker factor space with $Z_\mu=G_\mu.$

The maximal equicontinuous factor is likewise a rotation on a compact
abelian group. After translating its origin, we may assume that
\[
\theta_\mu(e_{Z_\mu})=e_{X_{\eq}}.
\]
The $\Gamma$-equivariance of $\theta_\mu$ and the density of
$\phi_\mu(\Gamma)$ in $Z_\mu$ then imply that $\theta_\mu$ is a
continuous group homomorphism. It is surjective by the minimality of
$X_{\eq}$. Consequently,
\[
X_{\eq}\cong Z_\mu/\ker(\theta_\mu).
\]
That is, $X_{\eq}$ is a homogenous $G_{\mu}$-space with
\[
g y=\theta_\mu(g)y,
\]
for $g\in G_{\mu}$ and $y\in X_{\eq}$, and the stabilizer of $e_{X_{\eq}}$ is $\ker(\theta_\mu)$.

 A group homomorphism
$\chi\colon\Gamma\to\mathbb T$ is a \textbf{continuous eigenvalue} of
$\Gamma\curvearrowright X$ if there exists $0\neq f\in C(X)$ such that
\[
\gamma f=\chi(\gamma)f
\qquad
\text{for every }\gamma\in\Gamma.
\]
We denote the group of continuous eigenvalues by
$\mathcal E_{\mathrm{top}}(X,\Gamma)$. 

Similarly, a group homomorphism
$\chi\colon\Gamma\to\mathbb T$ is a \textbf{measurable eigenvalue} of
$\Gamma\curvearrowright(X,\mu)$ if there exists
$0\neq f\in L^2(X,\mu)$ such that
\[
U_{\mu}(\gamma)f=\chi(\gamma)f
\qquad
\text{for every }\gamma\in\Gamma.
\]
We denote the group of measurable eigenvalues by
$\mathcal E_\mu(X,\Gamma)$. Then
\[
\mathcal E_{\mathrm{top}}(X,\Gamma)
\subset
\mathcal E_\mu(X,\Gamma).
\]

\begin{proposition}
\label{T-spectral-index}
Let $\Gamma$ be abelian, $\act$ a continuous action and $\mu\in\cM_e(\act)$ with full support. Then
\[
\mathcal E_\mu(X,\Gamma)/
\mathcal E_{\mathrm{top}}(X,\Gamma)
\cong
\widehat{\ker(\theta_\mu)}.
\]
Consequently,
\[
\bigl[
\mathcal E_\mu(X,\Gamma):
\mathcal E_{\mathrm{top}}(X,\Gamma)
\bigr]
=
\mult(\mu).
\]
\end{proposition}

\begin{proof}
We have that 
\[
\mathcal E_\mu(X,\Gamma)\cong \widehat{Z_\mu},
\qquad
\mathcal E_{\mathrm{top}}(X,\Gamma)\cong \widehat{X_{\eq}}.
\]
Under these identifications, the inclusion of continuous eigenvalues
into measurable eigenvalues is the dual map
\[
\widehat{\theta_\mu}\colon \widehat{X_{\eq}}\longrightarrow \widehat{Z_\mu},
\qquad
\eta\longmapsto\eta\circ\theta_\mu.
\]
Let
\[
K_\mu=\ker(\theta_\mu).
\]
Dualizing the exact sequence
\[
0\longrightarrow K_\mu
\longrightarrow Z_\mu
\xrightarrow{\ \theta_\mu\ }
X_{\eq}
\longrightarrow0
\]
gives
\[
0\longrightarrow \widehat{X_{\eq}}
\xrightarrow{\ \widehat{\theta_\mu}\ }
\widehat{Z_\mu}
\longrightarrow \widehat{K_\mu}
\longrightarrow0.
\]
Hence
\[
\mathcal E_\mu(X,\Gamma)/
\mathcal E_{\mathrm{top}}(X,\Gamma)
\cong \widehat{K_\mu}.
\]

If $K_\mu$ is finite, then Pontryagin duality gives
\[
|\widehat{K_\mu}|=|K_\mu|=\mult(\mu).
\]
If $K_\mu$ is infinite, then $\widehat{K_\mu}$ is infinite and every fiber
of $\theta_\mu$ is a coset of $K_\mu$, so $\mult(\mu)=\infty$. Therefore,
(with values in $\Nb\cup\{\infty\}$),
\[
\bigl[
\mathcal E_\mu(X,\Gamma):
\mathcal E_{\mathrm{top}}(X,\Gamma)
\bigr]
=
|\widehat{K_\mu}|
=
\mult(\mu). \qedhere
\]\end{proof}

\begin{corollary}
\label{cor:mainineq}
Let $\Gamma$ be abelian and $\Gamma\curvearrowright X$ be a minimal continuous action. Then
\[
R_{\IT}(X,\Gamma)
\geq
\sum_{\mu\in\cM_e(\act)}
\bigl[
\mathcal E_\mu(X,\Gamma):
\mathcal E_{\mathrm{top}}(X,\Gamma)
\bigr]
R_{\SE}(X,\Gamma,\mu).
\]
\end{corollary}

\begin{proof}
Use \cref{T-spectral-index} and \cref{T-length}.
\end{proof}

 \bibliographystyle{abbrv}

 \newcommand{\noop}[1]{}

\Addresses

\end{document}